\documentclass{article}[12pt]
\usepackage{fullpage}
\usepackage{amsmath}
\usepackage{amssymb}
\usepackage{newtxtext, newtxmath}

\usepackage{amsthm}
\usepackage[english]{babel}
\usepackage{caption}
\usepackage{xcolor}
\usepackage{algorithm}
\usepackage{algpseudocode}
\usepackage{tikz}
\usepackage{listings}
\usepackage{hyperref,cleveref}
\usetikzlibrary{fit,positioning,backgrounds,decorations.pathreplacing,calc}
\usetikzlibrary{graphs, shapes, calc, arrows.meta}
\usepackage{subcaption}  

\newtheorem{theorem}{Theorem}[section]
\newtheorem{lemma}[theorem]{Lemma}
\newtheorem{prop}[theorem]{Proposition}

\newcounter{claim-nb}[theorem]
\newcounter{claim-nbs}[section]
\newtheorem{CL}[theorem]{Claim}

\newenvironment{cproof}
{\begin{proof}
 [Proof of Claim.]
 \vspace{-1.2\parsep}}
{ \end{proof}}

\tikzset{>={Latex[length=3pt,width=2.6pt]}}

\newcommand{\R}{\mathbb{R}}

\newcommand{\1}{\mathbf{1}}

\newcommand{\Z}{\mathbb{Z}}

\newcommand{\cR}{\mathbb R}
\newcommand{\cZ}{\mathbb Z}

\newcommand{\zF}{\mathcal F}

\newcommand{\zJ}{\mathcal J}

\newcommand{\zB}{\mathcal B}

\DeclareMathOperator{\cone}{cone}
\DeclareMathOperator{\conv}{conv}
\DeclareMathOperator{\lin}{lin}
\DeclareMathOperator{\lat}{lat}
\DeclareMathOperator{\scr}{SCR}
\DeclareMathOperator{\sco}{SCO}
\DeclareMathOperator{\dij}{DIJ}
\DeclareMathOperator{\aff}{aff}
\DeclareMathOperator{\supp}{supp}

\DeclareMathOperator{\poly}{Poly}

\newtheorem{CO}[theorem]{Corollary}
\newtheorem{LE}[theorem]{Lemma}
\newtheorem{PR}[theorem]{Proposition}
\newtheorem{CN}[theorem]{Conjecture}

\newtheorem{DE}[theorem]{Definition}

\title{Lattice Structure and Efficient Basis Construction for Strongly Connected Orientations}

\author{Siyue Liu \thanks{Carnegie Mellon University. Email: \href{mailto:siyueliu@andrew.cmu.edu}{\texttt{siyueliu@andrew.cmu.edu}}.}\and 
  Olha Silina\thanks{Carnegie Mellon University. 
  Email: \href{siolga12@gmail.com}{\texttt{siolga12@gmail.com}}.}}
\date{}

\begin{document}
\maketitle

\begin{abstract}
Let $\vec{G}=(V,E^+\cup E^-)$ be a bidirected graph whose underlying undirected graph $G=(V,E)$ is $2$-edge-connected. A strongly connected orientation (SCO)  is defined as a subset of arcs that contains exactly one of $e^+,e^-$ for every $e\in E$ and induces a strongly connected subgraph of $\vec{G}$. Given a family $\mathcal{F}$ of proper subsets of $V$, we call an SCO tight if there is exactly one arc entering $U$ for every $U\in \mathcal{F}$. We give a polynomial-time algorithm to construct a set $\mathcal{B}$ consisting of tight SCO's which forms an integral basis for the linear hull of tight SCO's. That is, $\mathcal{B}$ is a linearly independent subset of tight SCO's, and every integer vector in the linear hull of tight SCO's can be written as an integral combination of $\mathcal{B}$. This extends the main result of Abdi, Conu\'ejols, Liu and Silina (IPCO 2025), who gave a non-constructive proof of the existence of such a basis in an equivalent setting. While the previous proof uses polyhedral theory, our proof is purely combinatorial and yields a polynomial-time algorithm. As an application of our algorithm, we show that parity-constrained tight strongly connected orientations can be solved in deterministic polynomial time. Along the way, we discover appealing connections to the theory of perfect matching lattices.
\end{abstract}

\section{Introduction}
Given a $2$-edge-connected undirected graph $G=(V,E)$, its \emph{bidirected} graph $\vec{G}=(V,E^+\cup E^-)$ is obtained from copying every edge $e\in E$ twice and orienting them in the opposite directions, denoted as $e^+\in E^+$ and $e^-\in E^-$. Here, $E^+$ and $E^-$ are two arbitrary orientations of $E$ that are opposite to each other. A \emph{strongly connected orientation (SCO)} is an arc subset $\vec{E}\subseteq E^+\cup E^-$ which contains exactly one of $e^+,e^-$ for every $e\in E$, and satisfies $|\delta_{\vec{E}}^-(U)|\geq 1$ for every $\emptyset\neq U\subsetneqq V$, denoted as $\sco(G)$.

In this paper,  which is the full version
of the extended abstract \cite{LiuSilina2026LatticeSCO}, we consider strengthenings that satisfy additional constraints defined by an arbitrary family $\zF\subseteq 2^V\setminus \{\emptyset, V\}$ of proper subsets of $V$.
Given $\zF$, we say an SCO $\vec{E}$ is \emph{tight} for $(G,\zF)$ if for every $U\in\zF$ there is exactly one arc coming into $U$, i.e., $|\delta_{\vec{E}}^-(U)|=1$. We denote by $\sco(G,\zF)$ the set of tight SCO's of $(G,\zF)$. We study the \emph{lattice of tight strongly connected orientations}, i.e., the lattice generated by the indicator vectors of tight SCO's, denoted by $\lat(\sco(G,\zF))$. Denote by $\lin(\sco(G,\zF))$ the linear hull of tight SCO's. It is easy to see that $\lat(\sco(G,\zF))\subseteq \lin(\sco(G,\zF))\cap\Z^A$. Given a rational linear subspace $L\subseteq \cR^{E^+\cup E^-}$, an \emph{integral basis for $L$} is a subset $\zB\subseteq L\cap \cZ^{E^+\cup E^-}$ of linearly independent vectors such that every vector in $L\cap \cZ^{E^+\cup E^-}$ can be written as an integral linear combination of $\zB$. Since the dimension of $\lin(\sco(G))$ is at most $|E|$, there are at most $|E|$ many sets $U\in \zF$ for which the constraints $|\delta_{\vec{E}}^-(U)|=1$ are linearly independent. Thereby we may assume w.l.o.g. that $|\mathcal{F}|\le |E|$, which implies the input size $(G,\zF)$ is bounded by the size of the graph. Our main result is the following:
\begin{theorem}\label{main-sco-theorem}
    Let $G=(V,E)$ be a $2$-edge-connected undirected graph. For an arbitrary $\zF\subseteq 2^V\setminus \{\emptyset, V\}$ with $|\zF|\le |E|$, we can construct an integral basis for $\lin(\sco(G,\zF))$ consisting of tight SCO's in polynomial time.
\end{theorem}
This result implies two things: first, $\lat(\sco(G,\zF))=\lin(\sco(G,\zF))\cap\Z^{E^+\cup E^-}$; second, there is a basis for $\lat(\sco(G,\zF))$ that consists of tight SCO's. 
This result is essentially the algorithmic counterpart of Abdi, Conu\'ejols, Liu and Silina \cite{abdi2025strongly}, who gave a non-constructive proof of the same result in terms of \emph{strongly connected reorientations (SCR)}, or \emph{strengthenings} in digraphs. In a digraph, a subset of arcs is a \emph{strengthening} if the digraph obtained from flipping the arcs in $J$ is strongly connected. Given a family $\zF\subseteq 2^V\setminus \{\emptyset, V\}$, we say a strengthening $J$ is \emph{tight} for $(D,\zF)$ if after flipping the arcs of $J$, there is exactly one arc coming into every $U\in \zF$. Denote by $\scr(D,\zF)$ the set of tight strengthenings of $(D,\zF)$. The main result of \cite{abdi2025strongly} is the following:
\begin{theorem}[\cite{abdi2025strongly}]\label{main-scr-theorem}
    Let $D=(V,A)$ be a digraph whose underlying undirected graph is $2$-edge-connected. Let $\zF\subseteq 2^V\setminus \{\emptyset, V\}$ such that $|\zF|\le |A|$ and $\gcd\{1-|\delta^-(U)|:U\in \zF\}=1$. Then there exists an integral basis for $\lin(\scr(D,\zF))$ consisting of tight strengthenings.
\end{theorem}
One can easily go between tight strengthenings in a digraph $D=(V,A)$ and tight SCO's in the bidirected graph $\vec{G}=(V,A\cup A^{-1})$, where $A^{-1}$ are the reverse arcs of $A$, via an invertible linear transformation, so our main \Cref{main-sco-theorem} implies \Cref{main-scr-theorem} (we prove this in \Cref{sec:reduction-to-digrafts}). The proof in \cite{abdi2025strongly} is an existence proof obtained from studying the implicit equalities and facets of the polytope of tight strengthenings. The main contribution of this paper is developing a new combinatorial proof which leads to a polynomial-time algorithm to construct such a basis.

Our main motivation for studying tight strengthenings is the following conjecture by Woodall \cite{Woodall78}. 
In a digraph, a \emph{dicut} is the set of arcs leaving a nonempty proper vertex subset with no incoming arcs, i.e., an arc set of the form $\delta^+(U)$ for some $\emptyset\neq U\subsetneqq V$ such that $\delta^-(U)=\emptyset$. A \emph{dijoin} is an arc subset that intersects every dicut at least once. In order to make the digraph strongly connected, at least one arc from each dicut has to be flipped, and thus, every strengthening is a dijoin. Woodall conjectured that the maximum number of disjoint dijoins is equal to the minimum size of a dicut. Schrijver \cite{schrijverobervation} reformulated Woodall's conjecture in the following:
\begin{CN}[\cite{schrijverobervation}]\label{Woodall-CN-2}
    Let $\tau\geq 2$ be an integer, and let $D=(V,A)$ be a digraph where the minimum size of a dicut is $\tau$. Then $A$ can be partitioned into $\tau$ strengthenings.
\end{CN}
Observe that if such a statement were true, the only strengthenings allowed in a partition would be the ones that intersect every dicut of size $\tau$ \emph{exactly} once. Therefore, it suffices to restrict ourselves to those strengthenings, which are precisely $\scr(D,\zF)$ for $\zF$ being the family of sets $U\subseteq V$ that induce dicuts $\delta^+(U)$ of size $\tau$. \Cref{main-scr-theorem} implies a relaxation of \Cref{Woodall-CN-2} that allows adding and subtracting tight strengthenings \cite{abdi2025strongly}. Formally, it implies that we can find an assignment $\lambda_J\in \cZ$ to every strengthening $J$ that intersects every minimum dicut exactly once such that $\sum_{J}\lambda_J\1_J = \1_A$ and $\sum_J  \lambda_J = \tau$.
As a new application, we address the \emph{parity strongly connected orientation (SCO)} problem. Given a $2$-edge-connected undirected graph $G=(V,E)$, whose bidirected graph is $\vec{G}=(V,E^+\cup E^-)$. The arcs of $\vec{G}$ are colored red or blue. Let $R\subseteq E^+\cup E^-$ be a subset of red arcs and let $\zF\subseteq 2^V\setminus \{\emptyset,V\}$. We define the \emph{parity strongly connected orientation (SCO)} problem as follows:

\paragraph{\textbf{Parity strongly connected orientation (Parity SCO)}} Find a tight SCO with an odd (even) number of red arcs if it exists.

We provide a deterministic polynomial-time algorithm for this problem as an application of our basis construction algorithm in \Cref{main-sco-theorem}.
\begin{theorem}\label{thm:parity-sco}
    Parity SCO can be solved in deterministic polynomial time.
\end{theorem}

\subsection{Proof overview}
We outline the proof of \Cref{main-scr-theorem} here. 

\paragraph{\textbf{Reduction to digrafts}}
First, we reduce from SCO's in undirected graphs to dijoins in bipartite digraphs. A digraph is \emph{bipartite} if every node is a source or a sink. The following notion of \emph{digrafts} is defined similarly in \cite{abdi2025strongly}:
\begin{DE}[digraft]
    A \emph{digraft} is a pair $(D,\zF)$ where $D=(V=S\dot\cup T,A)$ is a bipartite digraph with sources $S$ and sinks $T$ such that $|S|\leq |T|$, whose underlying undirected graph is $2$-edge-connected, and $\zF\subseteq 2^V\setminus \{\emptyset,V\}$ such that (a) if $U\in \zF$ then it induces a dicut $\delta^+(U)$, (b) $V\setminus v\in \zF$ for every sink $v\in T$. 
\end{DE}

We will be interested in dijoins satisfying degree constraints defined by $v\in V^t$. We say a dijoin is \emph{tight} if it intersects $\delta^+(U)$ exactly once for every $U\in \zF$. Denote by $\zJ(D,\zF)$ the set of tight dijoins of $(D,\zF)$. A digraft is \emph{tight dijoin-covered} if every arc $e\in A$ is contained in some tight dijoin. A dicut is \emph{tight} if every tight dijoin intersects it exactly once. A node $v\in V$ is \emph{tight} if every tight dijoin has degree $1$ on $v$. The order in which we define tight dijoins and tight dicuts (nodes) is important. Tight dijoins are explicitly defined as degree-constrained dijoins, while tight dicuts (nodes) are implicit, dependent on tight dijoins.

For an arbitrary $2$-edge-connected undirected graph $G$ and a family $\zF\subseteq 2^V\setminus\{\emptyset, V\}$, there is an identity map between $\sco(G,\zF)$ and $\zJ(D',\zF')$ for some digraft $(D',\zF')$ \cite{Cornuejols24,abdi2025strongly}. Thereby we may restrict ourselves to the setting of digrafts. We give the proof of the reduction to digrafts in \Cref{sec:reduction-to-digrafts}.


\paragraph{\textbf{Non-basic digrafts}}
We further reduce general digrafts to the following class of digrafts, where the family of dicuts $\zF$ consists only of shores of trivial dicuts.

Formally, let $D=(V=S\dot\cup T,A)$ be a bipartite digraph and let $S^t\subseteq S$ be a subset of sources. We consider digrafts $(D,\zF)$ where $\zF:=\big\{\{v\}:v\in S^t\big\}\bigcup\big\{V\setminus v:v\in T\big\}$. For simplicity, we denote these types of digrafts as $(D,S^t)$.
\begin{DE}[digraft $(D,S^t)$]\label{def:digraft-special}
    A \emph{digraft} $(D,S^t)$ is a pair of bipartite digraph $D=(V=S\dot\cup T,A)$ with sources $S$ and sinks $T$ such that $|S|\leq |T|$, whose underlying undirected graph is $2$-edge-connected, and a subset of sources $S^t\subseteq S$ with $|S^t|\le |S|-1$. 
\end{DE}

Denote by $V^t:=S^t\cup T$. In this setting, a tight dijoin is simply a dijoin that has exactly one arc incident to every $v\in V^t$.
Denote the set of tight nodes as $\widebar{V^t}$ and the set of tight sources as $\widebar{S^t}$. Note that every $v\in V^t$ is a tight node, and thus $V^t\subseteq \widebar{V^t}$ and $S^t\subseteq \widebar{S^t}$, but there could be more vertices in $\widebar{V^t}\setminus V^t$ that are forced to have degree $1$ for every tight dijoin. Also notice that every sink is tight by definition, so $\widebar{V^t}\setminus V^t$ consists solely of sources. By definition of tight nodes, $\zJ(D,S^t)=\zJ(D,\widebar{S^t})$. We remark that the assumption $|S|\le |T|$ is necessary for the existence of a tight dijoin. The assumption that $|S^t|\le |S|-1$ can also be made without loss of generality. This is because if $|S|<|T|$, then $(D,S^t)$ having a tight dijoin implies $|S^t|\le |S|-1$. If $|S|=|T|$, then every tight dijoin has degree $1$ for every $v\in S$, regardless of the choice of $S^t$. In other words, $\widebar{S^t}=S$ for an arbitrary $S^t\subseteq S$, which implies $\zJ(D,S^t)$ is invariant under the choice of $S^t$.

General digrafts $(D,\zF)$ can be reduced to digrafts $(D,S^t)$ through \emph{tight dicut decomposition}, which is described as follows. A dicut $\delta^+(U)$ is \emph{nontrivial} if $1<|U|<|V|-1$; otherwise, it is \emph{trivial}. If there is some $U\in \zF$ with $1<|U|<|V|-1$, it induces a nontrivial tight dicut $C=\delta^+(U)$. By contracting the shores $U$, $V\setminus U$ while preserving the family $\zF$ of tight dicut out-shores induced on the smaller digraphs, one obtains two digrafts, called \emph{$C$-contractions}. When $C$ is a nontrivial tight dicut, there is a natural way to obtain an integral basis for $\lin(\zJ(D,\zF))$ by combining the bases of the two $C$-contractions (\Cref{lemma:short_basis-going-up} from \cite{abdi2024strconn}), which can be done in polynomial time. After repeatedly applying this process, we may assume $\zF$ contains only trivial dicut out-shores. This way, we reduce $(D,\zF)$ to a list of digrafts of the form $(D,S^t)$.

A digraft $(D,S^t)$ may still have (implicit) nontrivial tight dicuts. The next step is to further apply tight dicut decomposition on $(D,S^t)$. To do this algorithmically, we identify two types of tight dicuts: \emph{barrier dicuts} and \emph{$2$-separation dicuts}. The terminology is inspired by matching theory \cite{lovasz87,LP09}. The following theorem allows us to find a nontrivial tight dicut for a digraft $(D,S^t)$ in polynomial time, if one exists.
\begin{theorem}\label{thm:barrier_tight_equiv}
    A tight dijoin-covered digraft $(D,S^t)$ has no nontrivial tight dicuts if and only if it does not contain any nontrivial barrier dicuts or $2$-separation dicuts. 
\end{theorem}

After applying tight dicut decomposition, we obtain a list of \emph{basic digrafts}, defined as follows.
\begin{DE}[basic digraft]
    A digraft $(D,S^t)$ is \emph{basic} if it is tight dijoin-covered, and does not contain any nontrivial tight dicuts.
\end{DE} 
We further differentiate between two types of basic digrafts.
\begin{DE}[brick, brace]
    A basic digraft $(D,S^t)$ is a \emph{brick} if $|S|<|T|$; it is a \emph{brace} if $|S|=|T|$.
\end{DE}

The terminology of bricks and braces is inspired by the theory of matching lattices by Lov\'asz \cite{lovasz87}, and Edmonds, Pulleyblank and Lov\'asz \cite{Edmonds82}. Interestingly, in a tight dijoin-covered digraft $(D,S^t)$ with $|S|=|T|$, a subset $J\subseteq A$ is a tight dijoin if and only if it is a perfect matching (see \Cref{sec:elemantary}). So, our definition of braces is identical to the one for perfect matchings in bipartite graphs \cite{lovasz87}. This is not the case for bricks. Nevertheless, we reuse the terminology due to its crucial role in determining the dimension of the lattice (\Cref{main-digraft}). Similar to the perfect matching setting \cite{Edmonds82}, we show that the results of any two tight dicut decomposition procedures on a digraft $(D,\zF)$ are the same list of bricks and braces up to the multiplicities of the arcs (\Cref{prop:unique_decomposition}). In particular, the number of bricks after applying tight dicut decomposition to $(D,\zF)$ is invariant, which we denote as $b(D,\zF)$.
The following is our main theorem for digrafts:
\begin{theorem}\label{main-digraft}
    Let $(D=(V,A),\zF)$ be a tight dijoin-covered digraft. Then, we can construct an integral basis $\zB$ for $\lin(\zJ(D,\zF))$ that consists of tight dijoins in polynomial time. Moreover, $|\zB|=|A|-|\widebar{V^t}|-b(D,\zF)+2$.
\end{theorem}

From now, we may assume the digraft is either a brick or a brace. The case of braces is identical to the setting of bipartite matchings so we omit it. We sketch the main steps in our proof on how to deal with bricks.
The first step is to reduce the digraft to \emph{robust} ones, and then to \emph{elementary} ones. We present how to construct integral basis for $\lin(\zJ(D,\zF))$ in the bottom-up order: elementary, robust, and non-robust digrafts.

\paragraph{Elementary digrafts} For the base case of the proof we define a new class of digrafts.
\begin{DE}[elementary digraft]
A digraft $(D,S^t)$ is \emph{elementary} if it is tight dijoin-covered and $|S^t|=|S|-1$.
\end{DE}
In an elementary digraft, every tight dijoin has a fixed degree on every $v\in V$. In particular, when $|S|=|T|$, a tight dijoin is a perfect matching. 
Thus, it includes the case where $(D,S^t)$ is a brace. 
We generalize the \emph{ear decomposition} for bipartite matchings to elementary digrafts (see e.g. \cite{plummer1986matching} Ch 4.1). We build a sequence of elementary digrafts by removing one ear at a time. The addition of each ear will increase the dimension of the lattice by exactly one, so we will add an appropriate tight dijoin to the basis. This gives us an algorithm to construct an integral basis for $\lin(\zJ(D,S^t))$.

\paragraph{Non-elementary robust digrafts}
An \emph{edge cover} is a subset of edges that has degree at least $1$ on every $v\in V$. We say an edge cover is \emph{tight} if it has degree $1$ on every $v\in V^t$. Since in a digraft every node $v\in V$ induces a dicut $\delta(v)$, every tight dijoin is a tight edge cover. We say a tight dijoin-covered digraft is \emph{robust} if the converse also holds:
\begin{DE}[robust digraft]\label{def:robust}
A digraft $(D,S^t)$ is \emph{robust} if it is tight dijoin-covered and every tight edge cover is a tight dijoin. 
\end{DE}

\begin{DE}[separating dicut]
    In a tight dijoin-covered digraft $(D,S^t)$, a dicut $C$ is \emph{separating} if it is nontrivial and there is a tight edge cover $J$ such that $J\cap C=\emptyset$.
\end{DE}
Clearly, a tight dijoin-covered digraft is robust if and only if it does not have a separating dicut. We give a polynomial time to test whether a digraft is robust:
\begin{theorem}\label{thm:testing-robust}
    There exists a polynomial-time algorithm that, given a digraft $(D,S^t)$, checks whether $(D,S^t)$ is robust. Moreover, if the answer is no, it returns a separating dicut.
\end{theorem}
In robust digrafts, a tight dijoin only needs to cover trivial dicuts as opposed to the entire family of dicuts. This makes robust digrafts exhibit many nice properties. We show that in a robust digraft $(D,S^t)$ with $|S^t|<|S|-1$, the digraft $(D,S^t\cup \{v\})$ obtained from setting an arbitrary $v\in S\setminus S^t$ to be tight remains robust (\Cref{cor:add-tight-vert}). Furthermore, there exists a tight dijoin $J$ of $(D,S^t)$ that has exactly $2$ arcs incident with $v$ (\Cref{CO:degree2-robust}). We show that adding $J$ to an integral basis for $\lin(\zJ(D,S^t\cup \{v\}))$ gives us an integral basis for $\lin(\zJ(D,S^t))$. Repeatedly setting the vertices in $S\setminus S^t$ to be tight, we arrive at the case of elementary digrafts where $|S^t|=|S|-1$.

\paragraph{Non-robust bricks} 
The most challenging part of our proof is to reduce a brick $(D,S^t)$ if it is not robust.
To this end, we will either move a non-tight vertex $v\in S\setminus \widebar{S^t}$ into $S^t$, or contract a nontrivial dicut. Both cases are equivalent to setting some non-tight dicut to be tight. In the first case, we will eventually reach elementary digrafts; in the second case, we reduce the sizes of the digrafts.

We say a digraft $(D,S^t)$ is a \emph{near-brick} if it is tight dijoin-covered and contains exactly one brick, i.e., $b(D,S^t)=1$. Ideally, we would like to find a dicut $C$ such that after setting it to be tight, the dimension of the lattice only drops by one. For this, we define the following notion of \emph{good dicuts}, which turns out to have the aforementioned property:
\begin{DE}[good dicut]
    In a brick $(D,S^t)$, a trivial dicut induced by node $u\in S\setminus S^t$ is \emph{good} if $(D,(S^t)':=S^t\cup \{u\})$ is a near-brick and $\widebar{(S^t)'}=(S^t)'$. A nontrivial dicut $C$ is \emph{good} if both $C$-contractions $(D_i,S_i^t)$ are near-bricks and satisfy $\widebar{S_i^t}=S_i^t$, $i=1,2$.
\end{DE}

We now elaborate on the algorithm of finding a good dicut. Recall that if a digraft is not robust, we can use \Cref{thm:testing-robust} to find a separating dicut. This dicut will be a starting point of a search for a good dicut.
We show that good dicuts can be described as maximal dicuts with respect to a certain partial order on the set of all non-tight dicuts. The algorithm of finding a good dicut is a dual procedure where we keep looking for improving dicuts $C$ in the partial order (as summarized in Theorems \ref{thm:separating-contractible} and \ref{thm:separating-good}). We show that the dicuts we find are linearly independent and thus the length of the procedure is bounded by the dimension of the linear hull. The structural \Cref{thm:barrier_tight_equiv} plays a crucial role in finding improving dicuts.


Finally, after finding a good dicut $C$, we set $C$ to be tight if it is trivial or contract $C$ if it is nontrivial. We construct a basis for the tight dijoins that intersect $C$ exactly once. It turns out we only need to add one additional tight dijoin $J$ to the basis, with the property that $|J\cap C|=2$ (\Cref{lemma:integral_basis}). The existence of such a tight dijoin follows from the \emph{Jump-free lemma} (\Cref{lemma:continuity} from \cite{abdi2024strconn}) and can be found in polynomial time. This concludes the basis construction for non-robust bricks.

\Cref{fig:flowchart} summarizes the main steps of the proof.
\begin{figure}
    \centering
\includegraphics[width=0.8\linewidth]{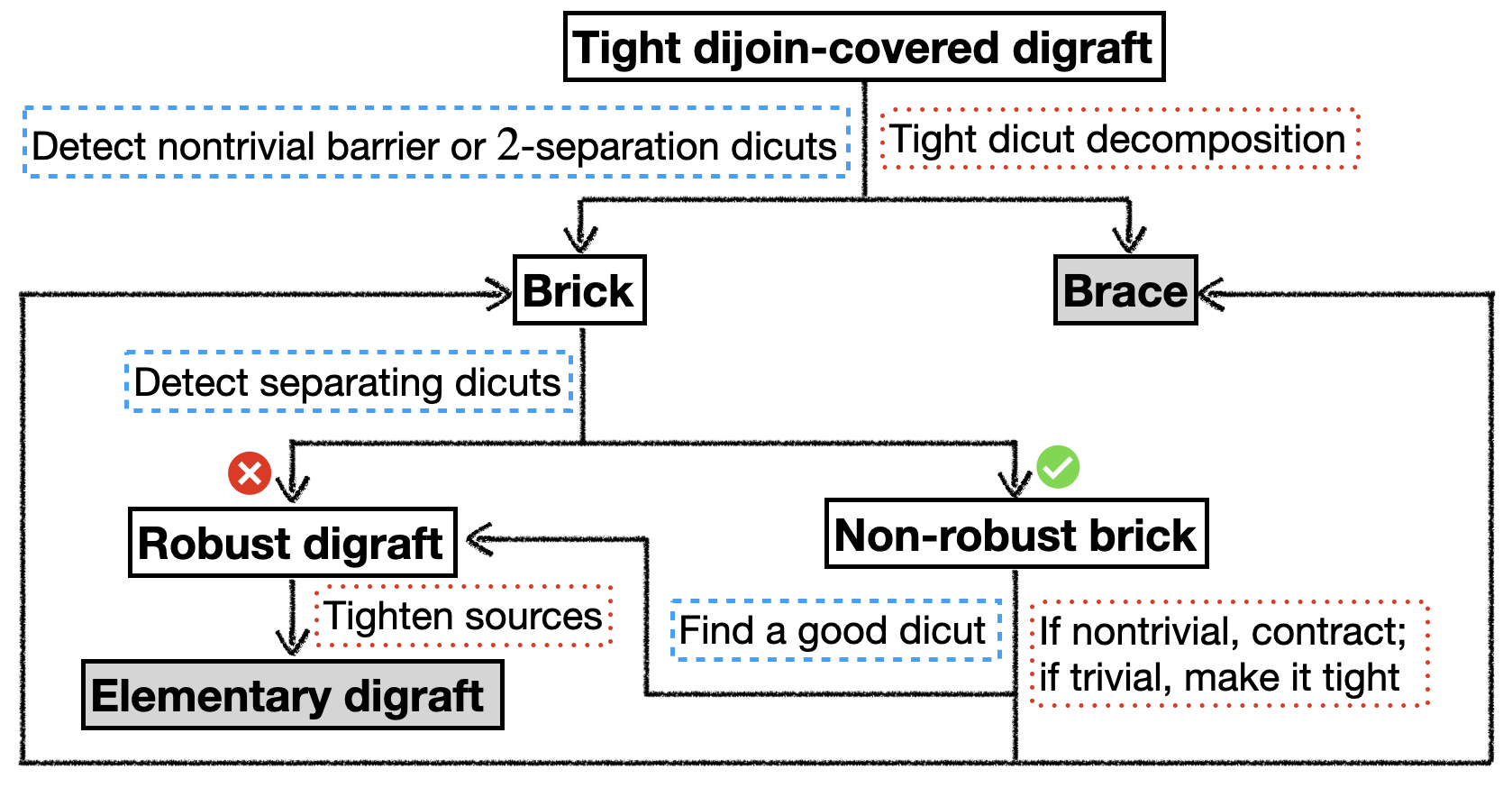}
    \caption{The solid boxes are types of digrafts. The dashed boxes are steps involved in finding specific types of dicuts. The dotted boxes are steps executed to reduce the digrafts. The end of the procedure are elementary digrafts and braces.}
    \label{fig:flowchart}
\end{figure}

\subsection{Related work}
\paragraph{\textbf{Connections to Woodall's conjecture}} The polytope of strongly connected orientations of a graph $G$ is given as follows by \cite{Edmonds77}:
\[\conv(\sco(G))=\big\{x\in \R_{\ge 0}^{E^+\cup E^-}:x_{e^+}+x_{e^-}=1\ \forall e\in E;\quad x(\delta_{\vec{G}}^+(U))\ge 1\  \forall \emptyset\neq U\subsetneqq V\big\}.\] 
Schrijver \cite{Schrijver80} refuted a conjecture of Edmonds and Giles \cite{Edmonds77}, which by \cite{Cornuejols24} is equivalently saying that in a $2$-edge-connected undirected graph $G$, the strongly connected orientations of $G$ form a \emph{Hilbert basis}, i.e., every integer vector in the cone generated by $\sco(G)$ can be written as a \emph{nonnegative} integer linear combination of SCO's. 
If such a statement were true, it has to be true for every face of $\cone(\sco(G))$:
restricting to a face defined by $\{x\in \R^{E^+\cup E^-}:x(\delta_{\vec{G}}^+(U))=1\ \forall U\in \zF\}$ for some family $\zF\subseteq 2^V\setminus \{\emptyset,V\}$, one obtains that every integer $x\in \cone(\sco(G,\zF))\cap \Z^{R^+\cup R^-}$ can be written as a nonnegative integer linear combination of tight SCO's of $(G,\zF)$. 
\Cref{main-sco-theorem} shows that if we relax the nonnegativity of the coefficients in the integer combination, such a statement would be true. 
Although Edmonds-Giles conjecture has been refuted, such a viewpoint remains interesting, since Woodall's conjecture \cite{Woodall78} is equivalently saying that every ``nowhere-zero" integer $x\in \cone(\sco(G,\zF))\cap \Z_{\ge 1}^{R^+\cup R^-}$ can be written as a nonnegative integer combination of tight SCO's of $(G,\zF)$ \cite{Cornuejols24}, which is still widely open. 
Another consequence of \Cref{main-sco-theorem} is that for every prime $p\geq 2$ and every digraph with minimum dicut size $\tau$, there is a $p$-adic packing of dijoins of value $\tau$ \cite{abdi2025strongly}, extending a result of Guenin and Hwang \cite{guenin2025dyadic} which proves the existence of a dyadic packing of value $\tau$. Our main \Cref{main-sco-theorem} allows us to find such a packing in polynomial time. This can be done by first restricting ourselves to polynomially many variables as in the sparse packing procedure in \cite{abdi2024strconn} and then using an algorithm for $p$-adic programming in \cite{abdi2024dyadic}.

\paragraph{\textbf{Comparison to existing proofs}} Abdi, Cornu\'ejols, Liu and Silina \cite{abdi2025strongly} initiated the study of the lattice of tight strengthenings. They proved a non-constructive version of \Cref{main-scr-theorem}. Their argument can be summarized as follows. First, they use the same reduction to the digraft setting $(D,\zF)$. They closely work with the polyhedron of tight dijoins, $\dij(D,\zF):=\{x\in \R_{\geq 0}^A: x(\delta^+(U))\geq 1\ \forall \text{ dicut } U;\ x(\delta^+(U))=1\ \forall U\in \zF\}$, which is given by \cite{Lucchesi78}. They reduce to basic digrafts by contracting tight dicuts. This step is non-constructive. Observe that $\dij(D,\zF)$ of the basic digraft does not have implicit equalities except those defined by trivial tight dicuts induced by tight nodes. If all facet-defining inequalities (FDI) are defined by trivial dicuts or nonnegativity constraints, the digraft is robust (this is indeed equivalent to \Cref{def:robust} because the bipartite edge cover polytope is integral, see e.g. \cite{Schrijver03} Ch 19.5). Using the integer decomposition property of tight edge covers and a result of Gerards and Seb\H{o} \cite{Gerards87}, they show that an integral basis consisting of tight dijoins exists. Otherwise, there is a nontrivial dicut $C=\delta^+(U)$ that induces an FDI. They set $C$ to be tight, which reduces the dimension of $\lin(\zJ(D,\zF))$ by one. Using the Jump-free lemma, they find one additional tight dijoin $J$ with $|J\cap C|=2$ to add to the basis. 

To make their proof constructive, the step of finding an FDI defined by a nontrivial dicut must be made algorithmic, which is highly nontrivial. For a general polyhedra, the problem of FDI recognition is $D^p$-hard \cite{PapadimitriouYannakakis1984}. 
Even for well-structured polyhedra such as perfect matching polytope, the similar question of finding an FDI defined by a nontrivial odd set is open. We discuss this in detail in the next paragraph.
We overcome these difficulties by first, providing a recognition algorithm for robust digrafts and second, finding a good dicut that is, in fact, an FDI in polynomial time.

\paragraph{\textbf{Perfect matching lattices}}Some results of this work are reminiscent of the setting of perfect matchings. In an undirected graph $G=(V,E)$ a \emph{perfect matching} is a set of edges incident with each $v\in V$ exactly once. A graph is \emph{matching-covered} if every edge $e\in E$ belongs to some perfect matching. Denote the set of perfect matchings in $G$ by $\mathcal{M}(G)$.
A well-celebrated result due to Lov\'asz \cite{lovasz87} in the theory of matching covered graphs is that for every $x\in \lin(\mathcal{M}(G))\cap \Z^E$ we have $2x \in \lat(\mathcal{M}(G))$. Moreover, if $G$ has no Petersen minor, $\lat(\mathcal{M}(G))=\lin(\mathcal{M}(G))\cap \Z^E$. Abdi and Silina \cite{abdi2025integral} recently gave a short proof of Lov\'asz's result.
In a series of follow-up works of Carvalho, Lucchesi, and Murty, it was shown that the lattice of perfect matchings $\lat(\mathcal{M}(G))$ contains a basis consisting of perfect matchings \cite{carvalho1997decomposicao,carvalho1999ear,Carvalho02,deCarvalhoLucchesiMurty2002a,deCarvalhoLucchesiMurty2002b} (see the excellent survey book \cite{LucchesiMurty2024}).
In proving these results, the authors leverage many structural results on matching-covered graphs, such as tight cut decomposition established by Edmonds, Pulleyblank and Lov\'asz \cite{Edmonds82}. They define the notions of \emph{barrier cuts} and \emph{$2$-separation cuts} which are crucial for the existence of nontrivial tight odd cuts. 

We give a characterization of when a digraft has a tight dijoin that has an interesting resemblance to Tutte's theorem for perfect matchings \cite{Tutte1950}. This allows us to identify similar notions of barrier dicuts and $2$-separation dicuts in digrafts. 
However, there are important distinctions between perfect matchings and tight dijoins in digrafts. For example, 
a well-studied problem in matching theory is testing whether a graph is \emph{Birkhoff-von Neumann (BvN)}, i.e., the perfect matching polytope is characterized solely by nonnegativity and degree constraints \cite{balas1981integer,de2004perfect,de2020birkhoff,de2006build}. It is open whether recognition of BvN graphs is in $\mathcal{NP}$. 
In contrast, we are able to recognize robust digrafts in polynomial time, which are precisely the digrafts whose polytope of tight dijoins are solely defined by nonnegativity and degree constraints. This distinction allows us to greatly simplify the proof of Carvalho, Lucchesi, and Murty, and bypass the hardness of finding the analogous notion of good cuts in perfect matchings.

\paragraph{\textbf{Congruency-constrained optimization}} The literature of congruency-constrained combinatorial optimization is quite rich, e.g., 
\cite{ArtmannWZ17,eisenbrand2024sensitivity,horsch2024problems,liu2024congruency,ElMaaloulySteinerWulf2023,McCuaig2004,NageleSZ22,nagele_submodular_2019,RobertsonSeymourThomas1999}. A central open question in this area is the \emph{exact perfect matching} problem: given a graph $G=(V,E)$, a subset of edges $R\subseteq E$ that are colored red, and a number $k$, find a perfect matching with $k$ red edges if it exists. A randomized polynomial algorithm exists by \cite{MulmuleyVaziraniVazirani1987}, but whether a deterministic algorithm
exists is open. As a relaxation, El Maalouly, Steiner, and Wulf \cite{ElMaaloulySteinerWulf2023} gave an algorithm to find a perfect matching with the correct parity in deterministic polynomial time using the theory of perfect matching lattices.
Tight strengthenings can be captured by submodular flows \cite{Edmonds77}, whereas tight strongly connected orientations can be captured by matroid intersection \cite{frank1984matroids}. H\"orsch et al. \cite{horsch2024problems} show that finding a common basis with an odd (even) number of red elements for two general matroids requires exponentially many independence queries. Our result on parity strongly connected orientations provides a new class of matroid intersection whose parity-constrained problem can be solved in polynomial time, generalizing the result on perfect matchings for bipartite graphs.

\paragraph{\textbf{Organization}} In \Cref{Sec:preliminary}, we provide preliminary results that are needed later in the proof. In \Cref{sec:tight-dicut-decomp}, we give characterization of tight dijoin-covered digrafts, tight dicuts, and establish the uniqueness of tight dicut decomposition. In particular we prove \Cref{thm:barrier_tight_equiv}. In \Cref{sec:robust-case}, we give the proof of \Cref{thm:testing-robust} and describe how to construct basis for robust digrafts. In \Cref{sec:nonrobust}, we deal with non-robust bricks, which is the most challenging case. In \Cref{sec:elemantary}, we show how to construct basis for elementary digrafts. In \Cref{sec:parity}, we prove \Cref{thm:parity-sco} on parity-constrained SCO's.

\section{Preliminaries}
\label{Sec:preliminary}
A subset $L\subseteq \cR^A$ is a \emph{lattice} if it is the set of integer linear combinations of finitely many vectors. Given a finite subset $R\subseteq \cR^A$, the \emph{lattice generated by $R$}, denoted $\lat(R)$, is the set of all integer linear combinations of the vectors in $R$. A \emph{lattice basis for $L$} is a set $\zB$ of linearly independent vectors that generates the lattice, i.e., $L = \lat(B)$. Every lattice has a basis.

Let $D=(V,A)$ be a digraph. For a subset of arcs $J\subseteq A$ and $v\in V$, denote by $d_J(v)$ the number of arcs in $J$ that are incident to $v$. For a proper subset $\emptyset\neq X\subsetneqq V$, denote by $\delta(X)$ the cut induced by $X$, $\delta_J(X):=\delta(X)\cap J$, and $d_J(X):=|\delta_J(X)|$. Denote by $\delta^-(X), \delta^+(X)$ the arcs entering and leaving from $X$, respectively. For $X\subseteq V$, denote by $N(X)$ the neighbors of $X$ in $D$. For $J\subseteq A$, denote by $N_J(X)$ the neighbors of $X$ in the subdigraph $(V,J)$. Let $X,Y\subseteq V$. Denote by $E(X)$ the arcs that have both endpoints in $X$. Denote by $E(X,Y)$ the arcs that have one endpoint in $X$ and another endpoint in $Y$. 
Denote by $\sigma_D(X)$ the number of (weakly) connected components in the subgraph of $D$ induced by the vertex set $X$. We omit the subscript $D$ if the context is clear.

Let $(D,S^t)$ be a digraft and let $b:S\rightarrow \Z_{\geq 0}$. We say a subset of arcs $J$ is a \textit{perfect $b$-matching} if $d_J(v)=b(v),\ \forall v\in S$, and $d_J(v)=1,\ \forall v\in T$. A subset of arcs $J$ is a \textit{$b$-dijoin} if $J$ is a perfect $b$-matching that intersects all dicuts. 
Abdi, Cornu\'ejols and Zlatin \cite{Abdi23-dijoins} reduce Woodall's conjecture to sink-regular bipartite digraphs. A different reduction can be found in \cite{schrijverobervation,Cornuejols24}. We state some facts that demonstrate the usefulness of the digraft setting. All of the results in this section are proved in \Cref{sec:proof-prelim}.

\begin{LE}\label{lemma:minimal_dicuts}
    In a digraft $(D,S^t)$, a minimal dicut that is not of the form $\delta^-(v)$ for some $v\in T$ is of the form $\delta^-(X\cup N(X))$ for some $X\subsetneqq S, X\neq \emptyset$.
\end{LE}
The following lemma shows that whether a subset of arcs $J\subseteq A$ is a dijoin only depends on its degree sequence.
\begin{LE}\label{lemma:b-matching}
    Given a digraft $(D,S^t)$ and $b:S\rightarrow \Z_{\geq 0}$, there exists a $b$-dijoin if and only if
    \begin{equation}\label{eq:b-matching}
    \begin{aligned}
        b(X)&\leq |N(X)|-1, ~\forall \emptyset\neq X\subsetneqq S,\\
        b(S)&=|T|.
    \end{aligned}
    \end{equation}
    Moreover, for $b$ satisfying \eqref{eq:b-matching}, every perfect $b$-matching is a $b$-dijoin.
\end{LE}
The following lemma appears in \cite{abdi2024strconn}. We give a short proof for completeness.
\begin{LE}[\cite{abdi2024strconn}, Theorem 3.2]\label{lemma:tight-dijoin-covered}
    A digraft $(D,S^t)$ is tight dijoin-covered if and only if it has at least one tight dijoin.
\end{LE}
Observe that the feasible degree sequences of tight dijoins, which are precisely $b\in \Z_{\geq 0}^S$ satisfying \eqref{eq:b-matching}, form a polymatroid. The exchange property of polymatroids allows \cite{abdi2025strongly} to establish the following lemma. We give a short constructive proof.
\begin{LE}[\cite{abdi2025strongly}, Jump-free Lemma]\label{lemma:continuity}
    Let $C$ be a dicut in a digraft $(D,S^t)$. If there are two tight dijoins $J'$ and $J''$ such that $|J'\cap C|=\lambda'$ and $|J''\cap C|=\lambda''$ for some $\lambda'<\lambda''$, then for every integer $\lambda'\leq\lambda\leq\lambda''$, we can find a tight dijoin $J$ with $|J\cap C|=\lambda$ in polynomial time.
\end{LE}

We say that two proper subsets $\emptyset\neq U,W \subsetneqq V$ \emph{cross} if $U\setminus W, W\setminus U, U\cap W\neq\emptyset$ and $U\cup W\neq V$. A family $\zF\subseteq 2^V\setminus \{\emptyset,V\}$ is a \emph{crossing family} if for every $U,W\in \zF$ that cross, $U\cap W, U\cup W\in \zF$. The following lemma is a consequence of the modularity of dicut size function $|\delta^+(U)|,\ \emptyset\neq U\subsetneqq V$.
\begin{LE}\label{lemma:tight_dicut_uncross}
    Let $(D,\zF)$ be a digraft, and let $\emptyset\neq U,W\subsetneqq V$ be such that $\delta^+(U),\delta^+(W)$ are dicuts and $U,W$ cross. Then, $\zJ\big(D,\zF\cup \{U\}\cup \{W\}\big)=\zJ\big(D,\zF\cup \{U\cap W\}\cup \{U\cup W\}\big)$
\end{LE}

\section{Tight dijoins, tight dicuts, and decomposition}\label{sec:tight-dicut-decomp}

\subsection{Tight dijoins via degree-constrained strongly connected orientations}\label{sec:degree_SCO}
First, we give a characterization of when a tight dijoin exists. Recall that $\sigma(V-X)$ denotes the number of (weakly) connected components after deleting some subset $X\subseteq V$ from $D$.
\begin{theorem}\label{thm:tight_dijoin}
    A digraft $(D,S^t)$ has at least one tight dijoin if and only if
    \begin{equation}\label{eq:tight_dijoin}
        \begin{aligned}
            \sigma(V-X)\leq |X|, \quad \forall \emptyset \neq X\subseteq T\\
            \sigma(V-X)\leq |X|, \quad \forall \emptyset\neq X\subseteq S^t.
        \end{aligned} 
    \end{equation}
\end{theorem}
This theorem essentially follows from the characterization for the existence of degree-constrained strongly connected orientations by Frank and Gy\'arf\'as \cite{gyarfas1978orient}.
\begin{theorem}[\cite{gyarfas1978orient}]\label{thm:degree_SCO}
    Let $f:V\rightarrow \Z_+\cup\{-\infty\}$ and $g:V\rightarrow \Z_+\cup\{\infty\}$ be two functions satisfying $f\leq g$. A $2$-edge-connected undirected graph $G=(V,E)$ has a strongly connected orientation $\vec{E}$ such that $f(v)\leq |\delta^-_{\vec{E}}(v)|\leq g(v)$ if and only if the following two conditions hold:
    \begin{subequations}\label{eq:sco}
    \begin{align}
        f(X)&\leq |E(X,V\setminus X)|+|E(X)|-\sigma(V-X),\quad \forall X\subseteq V, X\neq\emptyset;\label{eq:sco:lower_bound}
        \\
        g(X)&\geq |E(X)|+\sigma(V-X),\quad \forall X\subseteq V, X\neq\emptyset.\label{eq:sco:upper_bound}
    \end{align}
    \end{subequations}
\end{theorem}
\begin{proof}[Proof of Theorem \ref{thm:tight_dijoin}]
Notice that every tight dijoin $J$ in digraft $(D,S^t)$ is (inclusion-wise) minimal, in the sense that there is no other dijoin $J'\subsetneqq J$. Indeed, each dijoin $J'$ of $D$ should have at least one arc incident to each vertex in $T$, and thus $|J'|\geq |T|$. If $J$ is a tight dijoin, then $|J|=|T|$ and so it is minimal. It is known that every minimal dijoin is a strengthening (see e.g. \cite{lovasz2007combinatorial} Ch 6). Let $J^{-1}$ be the arcs obtained by flipping the directions of arcs in $J$. Thus, a subset of arcs $J\subseteq A$ is a tight dijoin if and only if $\vec{E}:=(A\setminus J) \cup J^{-1}$ is a strongly connected orientation and $f(v)\leq |\delta_{\vec{E}}^-(v)|\leq g(v)$ for every $v\in V$, where
\begin{equation}\label{eq:degree_constraints}
\begin{aligned}
    f(v)=\begin{cases}
        d(v)-1, & v\in T\\
        -\infty, & \text{ otherwise}
    \end{cases}
\end{aligned} \quad \text{ and } \quad
\begin{aligned} 
    g(v)=\begin{cases}
        1, & v\in S^t\\
        \infty, & \text{ otherwise.}
    \end{cases}
\end{aligned}
\end{equation}
The theorem follows directly from \Cref{thm:degree_SCO} applied to the functions above.
\end{proof}
\Cref{thm:tight_dijoin} essentially gives us an algorithm to check whether a digraft $(D,S^t)$ contains a tight dijoin, or equivalently, is tight dijoin-covered. We delay the proof to \Cref{sec:proof-tight-dicut-decomp}.
\begin{CO}\label{CO:tight-dijoin-polytime}
    There is a polynomial-time algorithm to check whether a digraft $(D,S^t)$ has at least one tight dijoin, or equivalently, is tight dijoin-covered. Moreover, if the answer is no, it returns some $X\subseteq V$ that violates \eqref{eq:tight_dijoin}.
\end{CO}

\subsection{Two types of tight dicuts and basic digraft recognition}
Consider an arbitrary tight dijoin-covered digraft $(D,S^t)$. We begin by introducing two special classes of dicuts (Figure \ref{fig:tikz_side_simple} illustrates both classes).

\paragraph{Barrier dicuts.} First, if there is some $\emptyset\neq X\subseteq S^t \text{ or } \emptyset\neq X\subseteq T$, such that $\sigma(V\setminus X)=|X|$, then each connected component of $V\setminus X$ induces a tight dicut. Indeed, suppose $|X|=k$ and let $U_1,...,U_k$ be the connected components of $V\setminus X$. Assume $X\subseteq T$ without loss of generality. 
Then, for each $U_i$, $\delta^-(U_i)=\emptyset$, which means $\delta^+(U_i)$ is a dicut. 
Thus, every dijoin $J$ should satisfy $|J\cap \delta^+(U_i)|\geq 1$. In particular, $J$ should have at least one arc going from $U_i$ to $X$ for each $U_i$ since $U_1,...,U_k$ are disconnected from each other. Therefore, $d_J(X)\geq k$. Moreover, every tight dijoin $J$ has $d_J(v)=1\ \forall v\in X$ since $X\subseteq T$. Thus, $d_J(X)=\sum_{v\in X}d_J(v)=|X|=k$. Therefore, there is exactly one arc going from $X$ to each $U_i$, which implies $\delta^+(U_i)$ is a tight dicut. We call $X$ a \textit{barrier} 
and each dicut $\delta^+(U_i)$ a \emph{barrier dicut}. 
If all the barrier dicuts induced by $U_i$'s are trivial, then either $|X|=1$ or $|U_i|=1\ \forall i\in [k]$. In the second case, each $\{u_i\}:=U_i$ is a source because $D$ is connected. Therefore, $X=T$ and $|S|=|T|=k$. Thus, if $|S|<|T|$, then $X$ being a nontrivial barrier implies that there is at least one nontrivial barrier dicut induced by $U_i$ for some $i\in [k]$.

\paragraph{$2$-separation dicuts.} 
Second, if there exist $u\in S^t, v\in T$ such that $\{u,v\}$ forms a $2$-cut, i.e., the digraph induced by vertices $V\setminus \{u,v\}$ is disconnected, then another type of tight dicut arises.
Indeed, let $U$ and $U'$ be unions of components of $V\setminus \{u,v\}$ such that $U\dot\cup U'=V\setminus \{u,v\}$. For every tight dijoin $J$, one has $1+1=d_J(u)+d_J(v)=d_J(U\cup \{u\})+d_J(U'\cup \{u\})\geq 1+1$, where the inequality follows from the fact that both $\delta^+(U\cup \{u\})$ and $\delta^+(U'\cup \{u\})$ are dicuts. Therefore, $d_J(U\cup \{u\})=d_J(U'\cup \{u\})=1$.
This implies that both $\delta^+(U\cup \{u\})$ and $\delta^+(U'\cup \{u\})$ are tight dicuts. We call such a $2$-cut $\{u,v\}$, where $u\in S^t, v\in T$, a \emph{$2$-separation}. We refer to the dicuts $\delta^+(U\cup \{u\})$ and $\delta^+(U'\cup \{u\})$ as \emph{$2$-separation dicuts}.

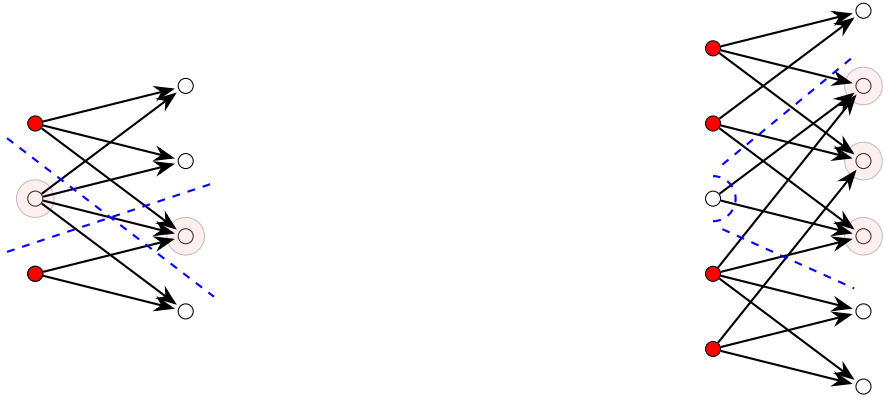
\begin{figure}[H]\label{pic:barrier-and-sep}
    \centering
        \begin{minipage}{0.4\textwidth}
        \centering
  	\begin{tikzpicture}[
  scale = 1,
  vx/.style = {draw, fill=white, circle, minimum size=2mm, inner sep=0pt},
  vxr/.style = {vx, fill=red}, 
  vxp/.style = {
    vx,
    fill=pink, opacity=0.25, transform shape, scale=2.5, 
  },
  dedge/.style = {thick, -{Stealth[length=8pt, sep]}},
  cut/.style = {blue, thick, dashed}
]

\node[vxr] (a1) at (-1,4.5) {};  
\node[vxr] (a2) at (-1,3.5) {};
\node[vx] (a3) at (-1,2.5) {};
\node[vxr] (a4) at (-1,1.5) {};
\node[vxr] (a5) at (-1,0.5) {};

\node[vx] (b1) at (1,5) {}; 
\node[vx] (b2) at (1,4) {};
\node[vx] (b3) at (1,3) {};
\node[vx] (b4) at (1,2) {};
\node[vxp] (x) at (1,4) {};
\node[vxp] (y) at (1,3) {};
\node[vxp] (z) at (1,2) {};
\node[vx] (b5) at (1,1) {};
\node[vx] (b6) at (1,0) {};

        \draw[dedge] (a1) to (b1);
        \draw[dedge] (a1) to (b2);
        \draw[dedge] (a1) to (b3);
        \draw[dedge] (a2) to (b1);
        \draw[dedge] (a2) to (b3);
        \draw[dedge] (a2) to (b4);
        \draw[dedge] (a3) to (b2);
        \draw[dedge] (a3) to (b4);
        \draw[dedge] (a4) to (b2);
        \draw[dedge] (a4) to (b4);
        \draw[dedge] (a4) to (b5);
        \draw[dedge] (a4) to (b6);
        \draw[dedge] (a5) to (b3);
        \draw[dedge] (a5) to (b5);
        \draw[dedge] (a5) to (b6);

\node (P) at (-1, 2.85) {};
\node (Q) at (1, 4.5) {};
\node (A) at (-1, 2.15) {};
\node (B) at (1, 1.25) {};
    \draw[cut] (P) to (Q);	
    \draw[cut] (A) to (B);	

    \draw[cut] (-1,2.2) arc (-90:90:0.3);
	\end{tikzpicture}
    \end{minipage}
    \hfill
    \begin{minipage}{0.4\textwidth}
        \centering
	\begin{tikzpicture}[
  scale = 1,
  vx/.style = {draw, fill=white, circle, minimum size=2mm, inner sep=0pt},
  vxr/.style = {vx, fill=red}, 
  vxp/.style = {
    vx,
    fill=pink, opacity=0.25, transform shape, scale=2.5, 
  },
  dedge/.style = {thick, -{Stealth[length=8pt, sep]}},
  cut/.style = {blue, thick, dashed}
]

\node[vxr] (a1) at (-1,2.5) {};  
\node[vx] (a2) at (-1,1.5) {};
\node[vxp] (x) at (-1,1.5) {};

\node[vxr] (a3) at (-1,0.5) {};
\node[vxr] (a3) at (-1,0.5) {};

\node[vx] (b1) at (1,3) {}; 
\node[vx] (b2) at (1,2) {};
\node[vx] (b3) at (1,1) {};
\node[vxp] (y) at (1,1) {};
\node[vx] (b4) at (1,0) {};

        \draw[dedge] (a1) to (b1);
        \draw[dedge] (a1) to (b2);
        \draw[dedge] (a1) to (b3);
        \draw[dedge] (a2) to (b1);
        \draw[dedge] (a2) to (b2);
        \draw[dedge] (a2) to (b3);
        \draw[dedge] (a2) to (b4);
        \draw[dedge] (a3) to (b3);
        \draw[dedge] (a3) to (b4);

\node (P) at (-1.5, 0.75) {};
\node (Q) at (1.5, 1.75) {};
\node (A) at (-1.5, 2.4) {};
\node (B) at (1.5, 0.1) {};
    \draw[cut] (P) to (Q);	
    \draw[cut] (A) to (B);
	\end{tikzpicture}
    \end{minipage}
    \caption{(Left): an example of barrier dicuts. The vertices in $S\setminus S^t$ are filled solid. The vertices of the barrier are circled and the arising barrier dicuts are in dashed. (Right): an example of $2$-separation dicuts. The vertices of a $2$-separation are circled and the two arising tight dicuts are in dashed.}
    \label{fig:tikz_side_simple}
\end{figure}

The importance of barrier dicuts and $2$-separation dicuts is highlighted in \Cref{thm:barrier_tight_equiv}. Its proof is quite technical, so we delay it to \Cref{sec:proof-tight-dicut}. We note that this theorem does not imply that every tight dicut is a barrier or $2$-separation dicut. It also gives a polynomial-time algorithm to recognize a basic digraft, which is proved in \Cref{sec:proof-tight-dicut-decomp}.
\begin{CO}\label{CO:tight-dicut-polytime}
    There is a polynomial-time algorithm to check whether a tight dijoin-covered digraft $(D,S^t)$ is basic. Moreover, if the answer is no, it returns some barrier dicut or $2$-separation dicut.
\end{CO}
Using the fact that $D$ is $2$-edge-connected, the next proposition shows that if a digraft $(D,S^t)$ has no nontrivial barrier, then $\widebar{S^t}=S^t$. In particular, if $(D,S^t)$ is a brick, having no nontrivial barrier dicut implies that it has no nontrivial barrier, and thereby $\widebar{S^t}=S^t$. This fact will be useful in later sections.

\begin{PR}\label{prop:new_tight}
Let $(D,S^t)$ be a tight dijoin-covered digraft. If there exists $s\in \widebar{S^t}\setminus S^t$, then there exists a barrier $X\subseteq T$ such that $N(s)\subseteq X$. 
\end{PR}
\begin{proof}
     As we argued in the proof of \Cref{thm:tight_dijoin}, $J$ is a tight dijoin if and only if $\vec{E}:=(A\setminus J) \cup J^{-1}$ is a strongly connected orientation obeying the degree constraints described in \eqref{eq:degree_constraints}. Since $s\in \widebar{S^t}$ is a tight node, there is no strongly connected orientation $\vec{E}$ satisfying $f(v)\leq |\delta_{\vec{E}}^-(v)|\leq g(v)$ for every $v\in V$, where
\begin{equation*}
\begin{aligned}
    f(v)=\begin{cases}
        d(v)-1, & v\in T\\
        2, & v=s\\
        -\infty, & o.w.
    \end{cases}
\end{aligned} \quad \text{ and } \quad
\begin{aligned} 
    g(v)=\begin{cases}
        1, & v\in S^t\\
        \infty, & o.w.
    \end{cases}
\end{aligned}
\end{equation*}
By \Cref{thm:degree_SCO}, there exists $X'$ that violates either \eqref{eq:sco:lower_bound} or \eqref{eq:sco:upper_bound}. This is not the case when the lower bound $f(s)=2$ is not imposed, since $(D,S^t)$ is tight dijoin-covered. This means we must have $s\in X'$ and \eqref{eq:sco:lower_bound} is violated, i.e.,
\[
f(X')\geq |E(X',V\setminus X')|+|E(X')|-\sigma(V-X')+1.
\]
Then, there exists $X\subseteq T$ such that $X'=X\cup \{s\}$ satisfying 
\[
\begin{aligned}
    \sum_{v\in X} (d(v)-1)+2&\geq \sum_{v\in X} d(v)+d(s)-|E(s,X)|-\sigma\big(V-(X\cup\{s\})\big)+1,\\
    i.e.,\quad \sigma\big(V-(X\cup\{s\})\big)&\geq |X|+d(s)-|E(s,X)|-1.
\end{aligned}
\]
Now, if we focus on a component $C$ in $V\setminus (X\cup\{s\})$ that contains some neighbor $t\in N(s)$, then $C\setminus \{t\}$ is nonempty because otherwise $s$ would be the only neighbor of $t$, which is not the case. This implies that a component in $V-(X\cup \{s\})$ remains nonempty in $V-(X\cup N(s))$. Clearly, if deleting $X\cup \{s\}$ separates two components, then deleting $X\cup N(s)$ separates them as well. Additionally, deleting $X\cup N(s)$ would isolate $s$ as an additional component. Therefore, $\sigma\big(V\setminus (X\cup N(s))\big)\geq \sigma\big(V\setminus (X\cup\{s\})\big)+1$. Thus,
\[
\sigma\big(V\setminus (X\cup N(s))\big)\geq \sigma\big(V\setminus (X\cup\{s\})\big)+1\geq |X|+d(s)-|E(s,X)|=|X|+|E(s,T\setminus X)|\geq |X\cup N(s)|.
\]
Therefore, we obtain a barrier $X\cup N(s)$ containing $N(s)$. 
\end{proof}

A direct corollary of \Cref{prop:new_tight} is the following.
\begin{CO}\label{brick-maximally-valid}
    A brick $(D,S^t)$ has $\widebar{S^t}=S^t$.
\end{CO}

\subsection{Tight dicut decomposition}\label{subsec:tight-decomp}
We can decompose a digraft into two smaller digrafts by contracting a nontrivial dicut $C$. The \emph{$C$-contractions} of a digraft $(D,S^t)$ are the two smaller digrafts $(D_1,S_1^t)$, $(D_2,S_2^t)$ obtained from contracting $V\setminus U$, $U$ into $\bar{u}$, $u$ resp., and setting the contracted vertices to be tight.
\begin{DE}[$C$-contractions of $(D,S^t)$]\label{def:dicut-contraction}
    Let $C=\delta^+(U)$ be a nontrivial dicut of $(D,S^t)$. 
    Let $U_1:=V\setminus U$ and $U_2:=U$. Let $D_i=(V_i,A_i)$ be the digraph obtained from $D$ after contracting $U_i$ to a singleton $u_i$; so $V_i=\{u_i\}\cup U_{3-i}$. Let $V_i^t:=(V^t\cap V_i)\cup \{u_i\}$ and $S_i^t:=V_i^t\cap S_i$. We refer to $(D_i,S_i^t),i=1,2$ as the \emph{$C$-contractions} of $(D,S^t)$. 
\end{DE}

The definition of $C$-contractions can be extended to general digrafts $(D,\zF)$ (see \Cref{def:dicut-contraction-general}). The following lemma from \cite{abdi2025strongly} implies that if $C$ is a tight dicut of a digraft $(D,\zF)$, then we can combine bases of two $C$-contractions into a basis of the original digraft. 
This is a fairly common argument that appears in many graph objects, e.g. perfect matchings, that involves ``gluing" along a tight cut.
\begin{lemma}[\cite{abdi2025strongly}, Lemma 10]\label{lemma:short_basis-going-up}
    Let $(D,\zF)$ be a digraft and $C=\delta^+(U)$ be a dicut such that $\zJ(D,\zF\cup \{U\})\neq\emptyset$. Let $(D_i,\zF_i),i=1,2$ be the $C$-contractions of $(D,\zF)$. Suppose $\zB_i\subseteq \zJ(D_i,\zF_i)$ is an integral basis for $\lin(\zJ(D_i,\zF_i))$, $i=1,2$. Then one can combine them into an integral basis $\zB\subseteq \zJ(D,\zF\cup \{U\})$ for $\lin(\zJ(D,\zF\cup \{U\}))$, where $|\zB|=|\zB_1|+|\zB_2|-|C|$.
\end{lemma}

Therefore, we can keep decomposing a digraft along tight dicuts until there is no nontrivial tight dicut. This will end up with a list of basic digrafts. We state a useful fact about tight dicut decomposition and delay its proof to \Cref{sec:proof-tight-dicut-decomp}.

\begin{prop}\label{prop:unique_decomposition}
    The results of any two tight dicut decomposition procedures on a digraft $(D,\zF)$ are the same list of bricks and braces up to the multiplicities of the arcs.
\end{prop}

Recall that $b(D,S^t)$ is the number of bricks that appear at the end of any decomposition procedure, which is an invariant by \Cref{prop:unique_decomposition}. 
The following lemma reduces the proof of \Cref{main-digraft} to the one for the basic digrafts. We give its proof in \Cref{sec:reduction-to-basic}.
\begin{lemma}\label{lemma:non-basic}
    Assume for every basic digraft $(D,S^t)$, one can construct an integral basis $\zB$ for $\lin(\zJ(D,S^t))$ that consists of tight dijoins in polynomial time, where $|\zB|=|A|-|\widebar{V^t}|-b(D,S^t)+2$. Then, for every tight dijoin-covered digraft $(D=(V,A),\zF)$, one can construct an integral basis $\zB$ for $\lin(\zJ(D,\zF))$ that consists of tight dijoins in polynomial time. Moreover, $|\zB|=|A|-|\widebar{V^t}|-b(D,\zF)+2$.
\end{lemma}

\section{Basis construction for robust digrafts}\label{sec:robust-case}
In this section, we show how to construct an integral basis for $\lin(\zJ(D,S^t))$ when $(D,S^t)$ is a robust digraft. Recall that a digraft is robust if its tight edge covers coincide with tight dijoins.
The following lemma characterizes the conditions for the existence of a tight edge cover. Its proof can be found in \Cref{sec:proof-robust}.
\begin{LE}\label{lemma:tight_edge_cover}
    A tight edge cover exists if and only if
    \begin{subequations}\label{eq:tight_edge_cover}
        \begin{align}
        |N(X)|&\geq |X|,\quad \forall X\subseteq S
        \label{eq:tight_edge_cover1}\\
        |N(Y)|&\geq |Y|,\quad \forall Y\subseteq T,\ N(Y)\subseteq S^t.
        \label{eq:tight_edge_cover2}
    \end{align}
    \end{subequations}
\end{LE}

Below we give two characterizations of robust digrafts. 

\begin{theorem}\label{thm:robust_digraft}
Let $(D,S^t)$ be a digraft. The following are equivalent:
\begin{enumerate}
    \item[(1)] $(D,S^t)$ is a robust digraft.
    \item[(2)] For every $b: S\rightarrow \Z_{\geq 0}$ such that
    \begin{equation}\label{eq:degree-edge-cover}
    \begin{aligned}
        b(v)&\geq 1, \quad \forall v\in S\\
        b(v)&=1,\quad \forall v\in S^t\\
        b(S)&=|T|,
    \end{aligned}
    \end{equation}
    there exists a $b$-dijoin.
    \item[(3)] It holds that \begin{subequations}\label{eq:robust}
        \begin{align}
        |N(X)|-|X|&\geq |T|-|S|+1,\quad \forall \emptyset \neq X\subsetneqq S, X\not\subseteq S^t
        \label{eq:robust_1}
        \\
        |N(X)|-|X|&\geq 1, \quad \forall \emptyset \neq X\subsetneqq S, X \subseteq S^t.
        \label{eq:robust_2}
    \end{align}
    \end{subequations}
\end{enumerate}
\end{theorem}

The condition \eqref{eq:robust} in particular leads to a polynomial-time recognition algorithm for robust digrafts (\Cref{thm:testing-robust}). It follows from \Cref{lemma:b-matching} that Condition \eqref{eq:robust_2} is needed for every tight dijoin-covered digraft. The stronger condition \eqref{eq:robust_1} is more interesting. It turns out that a set $X\in \arg\min_{\emptyset\neq X\subsetneqq S, X\not\subseteq S^t}\big(|N(X)|-|X|\big)$ that violates \eqref{eq:robust_1} would be a separating dicut, which can be found efficiently using submodular minimization. To prove \Cref{thm:robust_digraft}, we need the following key lemma.
\begin{lemma}\label{lemma:violate_robust}
    Let $(D,S^t)$ be a digraft that has a tight edge cover. Let $$X\in\arg\min_{\emptyset\neq X\subsetneqq S, X\not\subseteq S^t}\big(|N(X)|-|X|\big).$$ Let $C:=\delta^-(X\cup N(X))$. If $|N(X)|-|X|\leq |T|-|S|$, then there exists a tight edge cover $J$ such that $J\cap C=\emptyset$.
\end{lemma}
\begin{proof}
Let $(D_1=(V_1=S_1\dot\cup T_1, A_1),S_1^t)$ be the induced digraft on vertices $V_1:=X\cup N(X)$, where $S_1:=X$, $T_1:=N(X)$, $A_1:=E(X\cup N(X))$, $S_1^t:=S^t\cap X$ and $V_1^t:=S_1^t\cup T_1$. Let $Y=T\setminus N(X)$. 
Then $N(Y)\subseteq S\setminus X$. 
Let $(D_2=(V_2=S_2\dot\cup T_2, A_2), S_2^t)$ be the induced digraft on vertices $V_2:=Y\cup (S\setminus X)$, where $S_2:=S\setminus X$, $T_2:=Y$, $A_2:=E(Y\cup (S\setminus X))$, $S_2^t:=S^t\setminus X$ and $V_2^t:=S_2^t\cup T_2$. Figure \ref{fig:find-violated} illustrates the choice. We claim that each $(D_i,S_i^t)$ has a tight edge cover $J_i$, for $i=1,2$. This will complete the proof, as $J:=J_1\cup J_2$ is a tight edge cover for $(D,S^t)$ such that $J\cap C=\emptyset$. It remains to prove the claim.

Suppose for the sake of contradiction that $(D_1,S_1^t)$ does not have a tight edge cover. By \Cref{lemma:tight_edge_cover}, \eqref{eq:tight_edge_cover1} or \eqref{eq:tight_edge_cover2} is violated. For $Z\subseteq V_1$, denote by $N'(Z)$ the neighbors of $Z$ in $D_1$. Note that for every $X_1\subseteq S_1$, $N'(X_1)=N(X_1)$. It follows from the fact that $D$ has a tight edge cover and \Cref{lemma:tight_edge_cover} that
$|N'(X_1)|=|N(X_1)|\geq |X_1|$ for every $X_1\subseteq S_1$, which means \eqref{eq:tight_edge_cover1} is satisfied by $(D_1,S_1^t)$. Therefore, there exists $Y_1\subseteq T_1, N'(Y_1)\subseteq S_1^t$ that violates \eqref{eq:tight_edge_cover2}, i.e., $|N'(Y_1)|\leq |Y_1|-1$. Let $X':=X\setminus N'(Y_1)$. Then $N'(X')\subseteq T_1\setminus Y_1=N(X)\setminus Y_1$ which means 
\[
\begin{aligned}
    |N(X')|-|X'|=&|N'(X')|-|X'|\le |N(X)\setminus Y_1|-|X\setminus N'(Y_1)|\\
    =&|N(X)|-|X|+(|N'(Y_1)|-|Y_1|)\leq |N(X)|-|X|-1.
\end{aligned}
\] 
Moreover, since $X\not\subseteq S^t$, $N'(Y_1)\subseteq S_1^t\subseteq S^t$, we have $X'=X\setminus N'(Y_1)\not\subseteq S^t$. In particular, $X'\neq\emptyset$. Further, $X'\subseteq X\subsetneqq S$. Therefore, $X'$ contradicts our choice of $X$ which minimizes $|N(X)|-|X|$.

\begin{figure}
    \centering
    \begin{tikzpicture}[
    font=\small,
    dot/.style={circle,fill,inner sep=1.7pt},
    edge/.style = {thick, -{Stealth[length=8pt, sep]}},
    cutedge/.style={thick,dashed,-{Latex[length=2mm]}},
    setbox/.style={draw, rounded corners, thick, inner sep=5pt},
    subsetbox/.style={draw, rounded corners, dashed, thick, inner sep=3pt},
    xddbox/.style={draw, rounded corners, dashed, very thin, inner xsep=18pt, inner ysep=10pt},
    cut/.style = {blue, thick, dashed}
]


\node[dot] (x1) at (0,5.0) {};
\node[dot] (x2) at (0,4.1) {};
\node[dot] (x3) at (0,3.2) {};

\node[dot] (u1) at (0,1.5) {};
\node[dot] (u2) at (0,0.6) {};
\node[dot] (u3) at (0,-0.3) {};

\node[dot] (a1) at (4.2,5.0) {};
\node[dot] (a2) at (4.2,4.1) {};
\node[dot] (a3) at (4.2,3.2) {};

\node[dot] (b1) at (4.2,1.25) {};
\node[dot] (b2) at (4.2,-0.15) {};


\draw[edge] (x1) -- (a1);
\draw[edge] (x2) -- (a2);
\draw[edge] (x2) -- (a3);
\draw[edge] (x3) -- (a2);
\draw[edge] (x3) -- (a3);

\draw[edge] (u1) -- (b1);
\draw[edge] (u2) -- (b1);
\draw[edge] (u3) -- (b2);


\draw[edge] (u1) -- (a2);
\draw[edge] (u2) -- (a1);
\draw[edge] (u3) -- (a3);

\node (P) at (-1, 2.7) {};
\node (Q) at (5, 2) {};
\draw[cut] (P) to (Q);
\node (c) at (5.3,2) {\textcolor{blue}{$C$}};

\coordinate (c1) at ($(u1)!0.56!(a2)$);
\coordinate (c2) at ($(u2)!0.56!(a1)$);
\coordinate (c3) at ($(u3)!0.56!(a3)$);



\begin{scope}[on background layer]

    \node[xddbox, fit=(x1)(x2)(x3)(u1)(u2)] (Xddbox) {};

    \node[setbox, fit=(x1)(x2)(x3)] (Xbox) {};
    \node[setbox, fit=(a1)(a2)(a3)] (NXbox) {};
    \node[setbox, fit=(u1)(u2)(u3)] (NYbox) {};
    \node[setbox, fit=(b1)(b2)] (Ybox) {};

    \node[subsetbox, fit=(a1)] (Yone) {};
    \node[subsetbox, fit=(x1)] (NYone) {};
    \node[subsetbox, fit=(x2)(x3)] (Xprime) {};
    \node[subsetbox, fit=(u1)(u2)] (Xtwo) {};
    \node[subsetbox, fit=(b1)] (NNXtwo) {};
\end{scope}


\node[anchor=south] at ($(x1.north)+(0,0.45)$) {$S$};
\node[anchor=south] at ($(a1.north)+(0,0.45)$) {$T$};


\node[anchor=west] at ($(Xbox.west)+(-2,0)$) {$S_1=X$};

\node[anchor=west] at ($(NYbox.west)+(-2,0)$)
    {$S_2=S\setminus X$};

\node[anchor=east] at ($(Ybox.east)+(2,0)$)
    {$T_2=Y$};


\node[anchor=east] at ($(NYone.west)+(-0,0)$) {$N'(Y_1)$};
\node[anchor=west] at ($(Yone.east)+(0,0)$) {$Y_1$};

\node[anchor=east] at ($(Xprime.west)+(-0,0)$) {$X'$};
\node[anchor=east] at ($(Xtwo.west)+(-0,0)$) {$X_2$};
\node[anchor=west] at ($(NNXtwo.east)+(0,0)$) {$N''(X_2)$};


\node[anchor=east] at ($(Xddbox.west)+(1,-0.5)$)
    {$X''$};



\node[anchor=east] at ($(NXbox.east)+(2,-0)$)
    {$T_1=N(X)$};

\end{tikzpicture}
    \caption{Illustration of a set $X$ found in \Cref{lemma:violate_robust}. The dicut $C=\delta^-(X\cup N(X))$ is a separating dicut. In the case where $(D_1,S_1^t)$, the digraft induced by $S_1\cup T_1$, has no tight edge cover, the set $Y_1$ violates the condition \eqref{eq:tight_edge_cover2} and $X'=X\setminus N(Y_1')$ violates the minimality of $X$. In the case where $(D_2,S_2^t)$, the digraft induced by $S_2\cup T_2$, has no tight edge cover, the set $X_2$ violates the condition \eqref{eq:tight_edge_cover1} and $X''=X\cup X_2$ violates the minimality of $X$.} 
    \label{fig:find-violated}
\end{figure}

Suppose for the sake of contradiction that $(D_2,S_2^t)$ does not have a tight edge cover. By \Cref{lemma:tight_edge_cover}, \eqref{eq:tight_edge_cover1} or \eqref{eq:tight_edge_cover2} is violated. For $Z\subseteq V_2$, denote by $N''(Z)$ the neighbors of $Z$ in $D_2$. Note that for every $Y_2\subseteq T_2$, $N''(Y_2)=N(Y_2)$. It follows from the fact that $D$ has a tight edge cover and \Cref{lemma:tight_edge_cover} that
$|N''(Y_2)|=|N(Y_2)|\geq |Y_2|$ for every $Y_2\subseteq T_2$, $N''(Y_2)\subseteq S_2^t$, which means \eqref{eq:tight_edge_cover2} is satisfied by $(D_2,S_2^t)$. Therefore, there exists $X_2\subseteq S_2$ that violates \eqref{eq:tight_edge_cover1}, i.e., $|N''(X_2)|\leq |X_2|-1$. Let $X'':=X\cup X_2$. Then, 
\[
\begin{aligned}
    |N(X'')|-|X''|=&|N(X)\cup N''(X_2)|-|X\cup X_2|\\
    =& |N(X)|-|X|+(|N''(X_2)|-|X_2|)
    \leq |N(X)|-|X|-1.
\end{aligned}
\]
Since $|N(X)|-|X|\leq |T|-|S|$, we have $|N(X'')|-|X''|\leq |T|-|S|-1$. In particular, it implies that $X''\subsetneqq S$. Moreover, since $X\subseteq X''$ and $X\not \subseteq S^t$, it follows that $\emptyset \neq X''\not \subseteq S^t$. Therefore, $X''$ contradicts the choice of $X$.
           
\end{proof}

\begin{proof}[Proof of \Cref{thm:robust_digraft}]
    We first show that (3) $\Rightarrow$ (2). Let $b\in\Z_{\geq 0}^S$ satisfy \eqref{eq:degree-edge-cover}. We need to show that $b$ satisfies the condition \eqref{eq:b-matching} in \Cref{lemma:b-matching}. Let $\emptyset \neq X\subsetneqq S$. If $X\subseteq S^t$, then $b(X)=|X|\leq |N(X)|-1$, where the inequality follows from \eqref{eq:robust_2}. If $X\not\subseteq S^t$, then $b(X)=|T|-b(S\setminus X)\leq |T|-|S\setminus X|=|T|-|S|+|X|\leq |N(X)|-1$, where the last inequality follows from \eqref{eq:robust_1}. Thus, we conclude that there is a $b$-dijoin.

    We next show that (2) $\Rightarrow$ (1). By the definition of a digraft, $|S^t|\le |S|-1$, which implies there is a $b\in \Z_{\geq 0}^S$ satisfying \eqref{eq:degree-edge-cover}. It follows that there is a $b$-dijoin, which is a tight dijoin in particular. This proves that $(D,S^t)$ is a tight dijoin-covered. Next, we show that $(D,S^t)$ is robust. To see this, for an arbitrary tight edge cover $J$, let $b(v)=d_J(v)\ \forall v\in S$. Clearly, $b$ satisfies \eqref{eq:degree-edge-cover}. It follows that there exists a $b$-dijoin. It follows from \Cref{lemma:b-matching} that $J$ itself is a $b$-dijoin, which is a tight dijoin. Thus, $(D,S^t)$ is a robust digraft.

    Finally, we show that (1) $\Rightarrow$ (3). Since $(D,S^t)$ is tight dijoin-covered, there exists a $b$-dijoin for some $b: S\rightarrow \Z_{\geq 0}$ that satisfies \eqref{eq:degree-edge-cover}. Then,
\begin{equation}\label{eq:neighbor_digraft}
    |X|=b(X)\leq |N(X)|-1,\ \forall \emptyset\neq X\subsetneqq S,\ X\subseteq S^t,
\end{equation}
where the inequality follows from \Cref{lemma:b-matching}. This proves \eqref{eq:robust_2}. Assume for the sake of contradiction that \eqref{eq:robust_1} does not hold.
We pick an arbitrary $X\in\arg\min_{\emptyset\neq X\subsetneqq S, X\not\subseteq S^t}|N(X)|-|X|$ and let $C:=\delta^-(X\cup N(X))$. We know that  $|N(X)|-|X|\leq |T|-|S|$. Applying \Cref{lemma:violate_robust}, there exists a tight edge cover $J$ such that $J\cap C=\emptyset$. Since $C$ is a dicut, $J$ is not a dijoin, which contradicts the fact that $(D,S^t)$ is robust.
\end{proof}

This gives us a polynomial-time recognition algorithm for robust digrafts (\Cref{thm:testing-robust}). The proof is in \Cref{sec:proof-robust}. \Cref{thm:robust_digraft} has some immediate consequences.

\begin{CO}\label{cor:add-tight-vert}
    If $(D,S^t)$ is a robust digraft, then for an arbitrary $S^t \subseteq (S^t)'\subsetneqq S$, $(D,(S^t)')$ is also a robust digraft.
\end{CO}
\begin{proof}
    We use characterization (2) of robust digrafts in \Cref{thm:robust_digraft}. Since $S^t\subseteq (S^t)'$, any $b$ satisfying \eqref{eq:degree-edge-cover} for $(D,(S^t)')$ also satisfies \eqref{eq:degree-edge-cover} for $(D,S^t)$. It follows from the robustness of $(D,S^t)$ and \Cref{thm:robust_digraft} that there exists a $b$-dijoin. This means for every $b$ satisfying  \eqref{eq:degree-edge-cover} for $(D,(S^t)')$, there exists a $b$-dijoin, which again by \Cref{thm:robust_digraft} implies that $(D,(S^t)')$ is a digraft.
\end{proof}

\begin{CO}\label{CO:degree2-robust}
    Let $(D,S^t)$ be a robust digraft with $|S|<|T|$. If $|S^t|<|S|-1$, then for an arbitrary $v\in S\setminus S^t$, there exists a tight dijoin $J$ such that $d_J(v)=2$, which can be found in polynomial time.
\end{CO}
\begin{proof}
    Pick an arbitrary $v'\in S\setminus (S^t\cup \{v\})$. The existence of such $v'$ is guaranteed because $|S\setminus S^t|>1$. Define $b:S\rightarrow \Z_{\ge 0}$ as $b(v)=2$, $b(v')=|T|-|S|$, and $b(v'')=1$ for every $v''\neq v,v'$. Clearly, $b$ satisfies \eqref{eq:degree-edge-cover}. It follows from \Cref{thm:robust_digraft} that there exists a $b$-dijoin $J$, which is a tight dijoin such that $d_J(v)=2$. We can find such a $b$-dijoin using an algorithm that finds a perfect $b$-matching (see e.g. \cite{schrijver2003combinatorial} Ch 21.5). 
\end{proof}

The following lemma shows that one can combine an integral basis for $\lin(\{J\in \zJ(D,S^t): |J\cap C|=1\})$ and a tight dijoin $J_0$ with $|J_0\cap C|=2$ into an integral basis for $\lin(\{J\in \zJ(D,S^t): |J\cap C|=1\}\cup \{J_0\})$. We will use this fact repeatedly. Its proof is in \Cref{sec:proof-robust}.
\begin{LE}\label{lemma:integral_basis}
    Let $(D,S^t)$ be a tight dijoin-covered digraft and $C$ be a dicut. Let $\zB'$ be an integral basis of $\lin(\{J\in \zJ(D,S^t): |J\cap C|=1\})$ that consists of tight dijoins that intersect $C$ exactly once. If $J_0$ is a tight dijoin with $|J_0\cap C|=2$, then $\zB:=\zB'\cup \{J_0\}$ is an integral basis for $\lin(\{J\in \zJ(D,S^t): |J\cap C|=1\}\cup \{J_0\})$.
\end{LE}

We are ready to describe the algorithm to construct a basis for the lattice of tight dijoins of robust digrafts.

\begin{theorem}\label{thm:robust_basis}
    Let $(D,S^t)$ be a robust digraft with $|S|<|T|$. Then, we can construct an integral basis $\zB$ for $\lin(\zJ(D,S^t))$ that consists of tight dijoins in polynomial time. Moreover, $|\zB|=|A|-|V^t|+1$.
\end{theorem}
\begin{proof}
    To construct a basis $\zB$, we inductively set one vertex in $S\setminus S^t$ to be tight and add one tight dijoin to $\zB$. 
    
    Let $r:=|S\setminus S^t|-1$ and assume $S\setminus S^t=\{s_0,s_1,\ldots, s_r\}$. For $i\in [r]$, define $(D,S_i^t)$ as the digraft such that $S_i^t=S_{i-1}^t\cup \{s_i\}$ where $S_0^t:=S^t$. It follows from \Cref{cor:add-tight-vert} that $(D,S_i^t)$ is a robust digraft. Since $|S_{i-1}^t|<|S|-1$, it follows from \Cref{CO:degree2-robust} that we can find a tight dijoin $J_i$ of $(D,S_{i-1}^t)$ such that $d_{J_i}(s_i)=2$. Add $J_i$ to $\zB$. Finally, $(D,S_r^t)$ has $|S_r^t|=|S|-1$ and thus is an elementary digraft. We use \Cref{thm:ear_decomposition} to construct an integral basis $\zB'=\{J_{r+1},\ldots, J_k\}$ for $\lin(J(D,S_r^t))$. Add $\zB'$ to $\zB$.

    We prove that $\zB$ is indeed an integral basis for $\lin(J(D,S^t))$. Suppose for $i\in [r]$, $\zB_{i}:=\{J_{i+1},\ldots J_k\}$. We inductively prove that $|\zB_i|=|A|-|V_i^t|+1$ and $\zB_i$ is an integral basis for $(D,S_i^t)$. Assume $\zB_i$ is an integral basis for $(D,S_i^t)$. We show that $\zB_{i-1}=\zB_{i}\cup \{J_i\}$ is an integral basis for $(D,S_{i-1}^t)$. By applying \Cref{lemma:integral_basis} to the dicut $\delta^+(s_i)$ and dijoin $J_i$, we conclude that    $\zB_{i-1}$ is an integral basis for $\lin(\zJ(D,S_i^t)\cup \{J_i\})$. Since 
    \[
    \begin{aligned}
        |\zB_{i-1}|=&\dim(\lin(\zJ(D,S_i^t)\cup \{J_i\}))\leq \dim(\lin(\zJ(D,S_{i-1}^t)))\leq |A|-|V_{i-1}^t|+1\\
        =&|A|-(|V_i^t|-1)+1=|\zB_i|+1=|\zB_{i-1}|,
    \end{aligned}
    \] $\zB_{i-1}$ forms an integral basis for $\lin(\zJ(D,S_{i-1}^t))$ and $|\zB_{i-1}|=|A|-|V_{i-1}^t|+1$. This concludes the proof of the induction. 
\end{proof}

\section{Basis construction for non-robust bricks}\label{sec:nonrobust}
In this section, we deal with non-robust bricks, which is the most challenging case. The proof puts together everything we established so far. To start with, we define a partial order over all non-tight dicuts. 
Let $C,C'$ be non-tight dicuts of $(D,S^t)$. We say $C'$ \emph{dominates} $C$ if for every tight dijoin $J$, $|C\cap J|\geq |C'\cap J|$, denoted as $C'\succeq C$. We say $C'$ \emph{strictly dominates} $C$ if further for at least one tight dijoin $J$, $|C\cap J|> |C'\cap J|$, denoted as $C'\succ C$. We say $C'$ and $C$ are \emph{equivalent} if they dominate each other. We say $C$ is \emph{maximal under $\succeq$}, or simply \emph{maximal}, if no non-tight dicut strictly dominates $C$. 

Our goal is to find a good dicut to either contract if it is nontrivial or set to be tight if it is trivial. For this, we need to define two types of intermediate dicuts.
\begin{DE}[stable dicut, contractible dicut]
    In a tight dijoin-covered digraft $(D,S^t)$, a dicut is \emph{stable} if it is induced by a singleton $u\in S\setminus S^t$ and $(D,S^t\cup \{u\})$ is tight dijoin-covered. A dicut $C$ is \emph{contractible} if it is nontrivial and both $C$-contractions are tight dijoin-covered.
\end{DE}

In particular, a good dicut has to be one of the above, but satisfies an even stronger property. The following theorem gives a procedure to go from a separating dicut to a contractible dicut by finding dominating dicuts under $\succeq$.


\begin{theorem}\label{thm:separating-contractible}
    Let $(D,S^t)$ be a tight dijoin-covered digraft. If a separating dicut is maximal under $\succeq$, then it is contractible. Moreover, there exists a polynomial-time algorithm which, given an arbitrary separating dicut, outputs a contractible dicut by finding a sequence of at most $|V|$ separating dicuts, each strictly dominating the previous one.
\end{theorem}
\begin{proof}
    Let $C=\delta^+(U)$ be a maximal separating dicut. Let $(D_1=(V_1=S_1\cup T_1,A_1),S^t_1)$ be the $C$-contraction obtained from contracting $V\setminus U$ into $\bar{u}$. First, we show that $(D_1,S_1^t)$ is tight dijoin-covered. Suppose not. Using \Cref{CO:tight-dijoin-polytime}, we can in polynomial time find $X\subseteq S_1^t$ or $X\subseteq T_1$ such that $\sigma_{D_1}(V_1\setminus X)>|X|$. Then, $\bar{u}\in X$ since otherwise $\sigma_D(V\setminus X)>|X|$, contradicting the fact that $(D,S^t)$ is tight dijoin-covered. In particular, we know $X\subseteq T_1$. Let $s:=|X|$ and suppose $X=\{v_1,\ldots,v_{s-1},\bar{u}\}$. Let $k:=\sigma_{D_1}(V_1\setminus X)>s$ and let $F_1,F_2, \ldots, F_k$ be the dicuts induced by the connected components of $V_1\setminus X$ in $D_1$. In the original digraph $D$, they are the dicuts induced by connected components of $V\setminus \big((X\setminus \{\bar{u}\})\cup (V\setminus U)\big)$. Figure \ref{fig:find-better-dicut} illustrates the found structure.

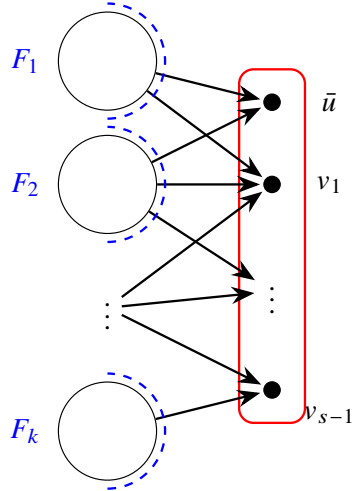
\begin{figure}[htbp]
  \centering 
  \resizebox{0.3\linewidth}{!}{
\begin{tikzpicture}[scale=1, 
 mynode/.style={circle, draw, fill=black, inner sep=2pt}, 
 comp/.style={circle, draw, fill=white, inner sep=10pt}, 
 dedge/.style = {thick, -{Stealth[length=8pt, sep]}},
 cut/.style = {blue, thick, dashed}
 ]

\node[comp] (b1) at (0,0) {};  
\node (b2) at (0,1.5) {$\vdots$}; 
\node[comp] (b3) at (0,3) {}; 
\node[comp] (b4) at (0,4.5) {};

\node (f1) at (-1,0) {$F_k$};
\node (f3) at (-1,3) {$F_{2}$}; 
\node (f4) at (-1,4.5) {$F_1$};

\node[mynode] (a1) at (2,0.5) {};  
\node (a2) at (2,1.7) {$\vdots$}; 
\node[mynode] (a3) at (2,3) {}; 
\node[mynode] (a4) at (2,4) {}; 

\node (vs) at (2.8,0.2) {$v_{s-1}$};
\node (v3) at (2.8,3) {$v_1$}; 
\node (v4) at (2.8,4) {$\bar{u}$};

\begin{scope}
\draw[thick, rounded corners=5pt]
($(a1) + (-0.4,-0.4)$) rectangle ($(a4) + (0.4,0.4)$);
\end{scope}

\draw[dedge] (b1) to (a1);
\draw[dedge] (b2) to (a1);
\draw[dedge] (b2) to (a2);
\draw[dedge] (b3) to (a2);
\draw[dedge] (b2) to (a3);
\draw[dedge] (b3) to (a3);
\draw[dedge] (b4) to (a3);
\draw[dedge] (b3) to (a4);
\draw[dedge] (b4) to (a4);

\draw[cut] 
(0,3.8) arc (90:270:-0.7);
\draw[cut] 
(0,2.3) arc (90:270:-0.7);
\draw[cut] 
(0,-0.7) arc (90:270:-0.7);

\end{tikzpicture}}
  \caption{$(D_1,S_1^t)$ with a set $X$ circled on the right-hand side; dicuts $F_1,\ldots, F_k$ are in dashed.}
  \label{fig:find-better-dicut}
\end{figure}

    Since $C$ is a separating dicut in $(D,S^t)$, there exists a tight edge cover $J^*$ such that $J^*\cap C=\emptyset$. By double counting the arcs between $X$ and $V\setminus X$,
    \[
    \sum_{i=1}^k|J^*\cap F_i|=|J^*\cap C|+\sum_{i=1}^{s-1}|J^*\cap \delta_D^-(v_i)|=s-1<k-1.
    \]
    We conclude that at for least one $j\in[k]$, $J^*\cap F_j=\emptyset$. Further, any such $F_j$ is nontrivial because $J^*$ is an edge cover. Thus, $F_j$ is a separating dicut of $D$.
    
    Finally, we show that $F_j$ strictly dominates $C$, which contradicts the fact that $C$ is maximal under $\succeq$. Indeed, for every tight dijoin $J$ of $(D,S^t)$,
    \[
    |J\cap F_j|+(k-1)\leq \sum_{i=1}^k|J\cap F_i|=|J\cap C|+\sum_{i=1}^{s-1}|J\cap \delta_D^-(v_i)|=|J\cap C|+(s-1)<|J\cap C|+(k-1).
    \]
    Thus, $|J\cap F_j|<|J\cap C|$ for every tight dijoin $J$, which means $F_j$ strictly dominates $C$. This concludes the proof that $(D_1,S_1^t)$ is tight dijoin-covered. 

    Analogously, we show that the $C$-contraction obtained from contracting $U$ is also tight dijoin-covered. This concludes the proof that a maximal separating dicut $C$ is contractible. 
    Starting from an arbitrary separating dicut $C$, if it is non-contractible, the above argument gives us an algorithm to find another separating dicut $F_j\succ C$ such that $|J\cap F_j|<|J\cap C|$ for every tight dijoin $J$. Let $C=C_1\prec\ldots \prec C_t$ be the dicuts found by the algorithm, where $C_t$ is a maximal separating dicut. Fix an arbitrary tight dijoin $J_0$, we know that $|T|=|J_0|\geq |J_0\cap C_1|>\ldots >|J_0\cap C_t|\geq 1$. Therefore, the process terminates in at most $|T|\leq |V|$ steps, at which point we find a contractible dicut.
\end{proof}

Once we have a contractible dicut, the following theorem gives a procedure to find a good dicut. The process consists of two steps. First, we find a sequence of at most $|A|$ contractible or stable dicuts, each \emph{strictly} dominating the previous one, at the end of which is a maximal dicut $C$. If $C$ is a stable dicut, then it is good and the algorithm terminates. If $C$ is a maximal contractible dicut, then we move to the second step: find a sequence of at most $|V|$ equivalent contractible dicuts until we reach a good dicut. A similar procedure of finding a ``good" cut exists in perfect matching lattices \cite{carvalho1997decomposicao,deCarvalhoLucchesiMurty2002b,LucchesiMurty2024}, but we have to overcome nontrivial technical difficulties to make it work for the digraft setting.

\begin{theorem}\label{thm:separating-good}
    Let $(D,S^t)$ be a brick. \begin{enumerate}
        \item[(a)] If a stable dicut is maximal under $\succeq$, then it is good.
        \item[(b)] If a contractible dicut is maximal under $\succeq$ and not equivalent to a stable dicut, then it is good.
        \item[(c)] There exists a polynomial-time algorithm which, given an arbitrary contractible dicut, outputs a good dicut by finding a sequence of at most $|A|+|V|$ contractible or stable dicuts, each dominating the previous one.
    \end{enumerate}   
\end{theorem}

To prove \Cref{thm:separating-good}, we need a lemma from \cite{abdi2025strongly}. We include a proof for completeness.
\begin{LE}\label{lemma:contractible-def}
    In a tight dijoin-covered digraft $(D,S^t)$, a nontrivial dicut $C$ is contractible if and only if there is a tight dijoin $J$ such that $|J\cap C|=1$.
\end{LE}
\begin{proof}
    Suppose $C=\delta^+(U)$ for some $\emptyset\neq U\subsetneqq V$. Let $(D_1=(V_i,A_i),S^t_1)$ be the $C$-contraction obtained from contracting $V\setminus U$ into $u_1$; let $(D_2,S^t_2)$ be the $C$-contraction obtained from contracting $U$ into $u_2$. We first prove ``only if". If $C$ is contractible, both $(D_i,S^t_i)$'s are tight dijoin-covered. Pick an arbitrary $e\in C$. There exist $J_i\in \zJ(D_i,S_i)$ such that $e\in J_i,\ i=1,2$. Further, $u_i$ being a tight node in $(D_i,\zF_i)$ implies that $J_i\cap C=\{e\},\ i=1,2$. Then, it follows from \Cref{de-composition} ``composition" that $J:=J_1\cup J_2$ is a tight dijoin of $(D,S^t)$ such that $|J\cap C|=1$. We then prove ``if". Suppose there is a tight dijoin $J$ of $(D,S^t)$ such that $|J\cap C|=1$. It follows from \Cref{de-composition} ``decomposition" that $J_i:=J\cap A_i$ is a tight dijoin of $(D_i,S_i^t),\ i=1,2$. By \Cref{lemma:tight-dijoin-covered}, $(D_i,S_i^t)$ is tight dijoin-covered for $i=1,2$, which means $C$ is a contractible dicut.
\end{proof}

We are now ready to prove \Cref{thm:separating-good}.
\begin{proof}[Proof of \Cref{thm:separating-good}]
    We prove by induction on the number of vertices $|V|$. In the case where $C=\delta^+(U)$ is a maximal contractible dicut, let $(D_1=(V_1=S_1\cup T_1,A_1),S^t_1)$ be the $C$-contraction obtained from contracting $V\setminus U$ into $\bar{u}$, and let $(D_2=(V_2=S_2\cup T_2,A_2),S^t_2)$ be the $C$-contraction obtained from contracting $U$ into $u$. In the case where $C=\delta^+(u)$ is a maximal stable dicut, let $(D_2=(V_2=S_2\cup T_2,A_2),S^t_2)$ be the digraft obtained by setting $u$ to be tight in $(D,S^t)$. The proof that $(D_2,S_2^t)$ is a near-brick is identical in both settings, so we will address them together. We first prove that $(D_2,S_2^t)$ is a near-brick and $\widebar{S_2^t}=S_2^t$.  We prove this by induction on the size of $V_2$. By definition of a contractible dicut (stable dicut), $(D_2,S_2^t)$ is tight dijoin-covered. 
\begin{CL}\label{claim:imbalance}
    $|S_2|<|T_2|$.
\end{CL}
\begin{cproof}
    Since $(D,S^t)$ is a brick, $C$ is not a tight dicut. Pick a tight dijoin $J$ of $(D,S^t)$ such that $|J\cap C|> 1$. By double counting the arcs in $D_2$,
    \[
    1+(|S_2|-1)< |J\cap C|+\sum_{s\in S_2\setminus \{u_2\}}|J\cap \delta^+(s)|=\sum_{t\in T_2} |J\cap \delta^-(t)|=|T_2|,
    \]
    which implies $|S_2|<|T_2|$.
\end{cproof}
\begin{CL}\label{claim:no-2-sep}
    $(D_2,S_2^t)$ has no $2$-separation dicuts.
\end{CL}
\begin{cproof}
Suppose for the sake of contradiction that there exists a $2$-separation $\{v,w\}\subseteq S_2^t\cup T_2$. This can be found in polynomial time by checking all possible pairs of vertices. One of the vertices has to be $u$, since otherwise $\{v,w\}$ is a $2$-separation of $(D,S^t)$, contradicting the fact that $(D,S^t)$ is a brick. Say $v=u$. Let $U_1$ and $U_2$ be unions of components of $V_2\setminus \{u,w\}$ such that $U_1\dot\cup U_2=V_2\setminus \{u,w\}$. Let $F_1$ and $F_2$ be two dicuts induced by $u\cup U_1$ and $u\cup U_2$, respectively. They are both nontrivial dicuts. For every tight dijoin $J$ of $(D,S^t)$, we have
    \[
    |J\cap F_1|+|J\cap F_2|=|J\cap C|+|J\cap \delta_D^-(w)|=|J\cap C|+1,
    \]
    rearranging which we have
    \begin{equation}
        \sum_{i=1,2}(|J\cap F_i|-1)=|J\cap C|-1.
    \end{equation}
    Since $C$ is contractible (stable), there exists a tight dijoin $J_1$ of $(D,S^t)$ such that $|J_1\cap C|=1$. Then,
    \begin{equation*}
    0\leq \sum_{i=1,2}(|J_1\cap F_i|-1)=|J_1\cap C|-1=0,
    \end{equation*}
    which implies that $|J_1\cap F_1|=|J_1\cap F_2|=1$. This implies that both $F_1$ and $F_2$ are contractible. 
    We show that at least one $F_j$, $j=1,2$ strictly dominates $C$, leading to a contradiction that $C$ is maximal. Indeed, for every tight dijoin $J$ of $(D,S^t)$ and every $j\in [2]$, 
    \begin{equation*}
    |J\cap F_j|-1\leq \sum_{i=1,2}(|J\cap F_i|-1)=|J\cap C|-1,
    \end{equation*}
    which means $|J\cap F_j|\leq |J\cap C|$. This means $F_j$ dominates $C$. It suffices to show there exists $F_j$, $j=1,2$ that strictly dominates $C$. Since $(D,S^t)$ is a brick, $C$ is not tight. Together with the fact that $C$ is contractible (stable), we can find a tight dijoin $J_2$ of $(D,S^t)$ such that $|J_2\cap C|=2$ according to \Cref{lemma:continuity}. Then, 
    \begin{equation*}
    \sum_{i=1,2}(|J_2\cap F_i|-1)=|J_2\cap C|-1=1.
    \end{equation*}
    Since both terms on the left-hand side are nonnegative, exactly one $F_j$ satisfies $|J_2\cap F_j|=1<|J_2\cap C|$. This implies that $F_j$ strictly dominates $C$, a contradiction.
\end{cproof}
    \begin{CL}\label{claim:near-brick}
        $(D_2,S_2^t)$ is a near-brick.
    \end{CL}
    \begin{cproof}
    It follows from \Cref{thm:barrier_tight_equiv}, \Cref{claim:imbalance}, and \Cref{claim:no-2-sep} that if $(D_2,S_2^t)$ has no nontrivial barrier dicut, then $(D_2,S_2^t)$ is a brick and we are done. Otherwise, we can in polynomial time find a nontrivial barrier $X\subseteq S_2^t$ or $X\subseteq T_2$ using \Cref{CO:tight-dicut-polytime}. It must contain $u$ since otherwise $\sigma_D(V\setminus  X)\geq \sigma_{D_2}(V\setminus X)=|X|$, contradicting the fact that $(D,S^t)$ is a brick. In particular, we have $X\subseteq S^t\cup \{u\}$. Let $k:=|X|=\sigma_{D_2}(V_2\setminus X)$ and suppose $X=\{v_1,\ldots,v_{k-1},u\}$. Let $F_1,F_2, \ldots, F_k$ be the dicuts induced by the connected components $U_1,\ldots, U_k$ of $V_2\setminus X$ in $D_2$. For every tight dijoin $J$, we have
    \[
    \sum_{i=1}^k|J\cap F_i|=|J\cap C|+\sum_{i=1}^{k-1}|J\cap \delta_D^+(v_i)|=|J\cap C|+(k-1),
    \]
    rearranging which we have
    \begin{equation}
        \sum_{i=1}^k(|J\cap F_i|-1)=(|J\cap C|-1).
    \end{equation}
    Since $C$ is contractible (stable), there exists a tight dijoin $J_1$ such that $|J_1\cap C|=1$. Then,
    \begin{equation*}
    0\leq \sum_{i=1}^k(|J_1\cap F_i|-1)=|J_1\cap C|-1=0,
    \end{equation*}
    which implies that $|J_1\cap F_i|=1\ \forall i\in [k]$. This implies that each $F_i$ is either a trivial dicut induced by a sink or a contractible dicut. Notice that there at least one $F_j$ is non-tight, since otherwise $C$ is a tight dicut, which is not the case. We claim that exactly one $F_j$ is non-tight. Indeed, for every tight dijoin $J$ of $(D,S^t)$ and every $j\in [k]$, 
    \begin{equation*}
    |J\cap F_j|-1\leq \sum_{i=1}^k(|J\cap F_i|-1)=|J\cap C|-1,
    \end{equation*}
    which means $|J\cap F_j|\leq |J\cap C|$. This means $F_j$ dominates $C$. To find an $F_j$, $j\in [k]$ that strictly dominates $C$, we find a tight dijoin $J_2$ such that $|J_2\cap C|=2$ using \Cref{lemma:continuity}. Then, 
    \begin{equation*}
    \sum_{i=1}^k(|J_2\cap F_i|-1)=|J_2\cap C|-1=1.
    \end{equation*}
    Since all terms on the left-hand side are nonnegative, exactly one $F_j$ satisfies $|J_2\cap F_j|=2$ and every other $i\in [k]\setminus j$ has $|J_2\cap F_i|=1<|J_2\cap C|$. If $F_i$ is non-tight for some $i\in [k]\setminus j$, then $F_i$ strictly dominates $C$, a contradiction. Therefore, exactly one $F_j$ is non-tight, and hence is a contractible dicut. Other dicuts $F_i$ are induced by singleton components $U_i=\{u_i\}$ for $i\in [k]\setminus \{j\}$. (Figure \ref{fig:find-equiv-dicut} illustrates this case.) Since $F_j$ dominates $C$, $F_j$ is a maximal contractible dicut.
    Let $(D_{21}=(S_{21}\dot\cup T_{21},A_{21}),S_{21}^t)$, $(D_{22}=(S_{22}\dot\cup T_{22},A_{22}),S_{22}^t)$ be the $F_j$-contractions obtained from contracting $U_j$, $V\setminus U_j$ in $(D_{2},S_{2}^t)$ into singletons $u_j$, $\bar{u}_j$, respectively. Since $|V_{22}|<|V_2|$, by the inductive hypothesis we have that $(D_{22},S_{22}^t)$ is a near-brick. Since $(D_{21},S_{21}^t)$ satisfies $S_{21}=\{v_1,\ldots, v_{k-1},u\}$ and $T_{21}=\{u_1,\ldots, u_k\}$, we have $|S_{21}|=|T_{21}|$ and thus it does not have a brick. Therefore, the number of bricks $b(D_2,S_2^t)=b(D_{21},S_{21}^t)+b(D_{22},S_{22}^t)=1$, which means $(D_2,S_2^t)$ is a near-brick. 

    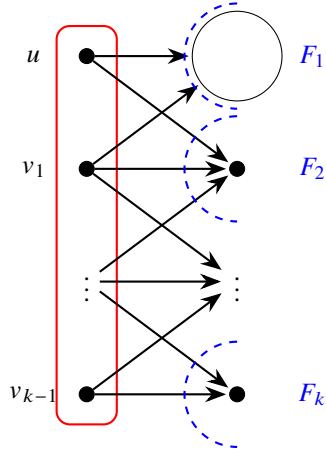
\begin{figure}[htbp]
  \centering
  \begin{tikzpicture}[scale=1, 
 mynode/.style={circle, draw, fill=black, inner sep=2pt}, comp/.style={circle, draw, fill=white, inner sep=10pt}, dedge/.style = {thick, -{Stealth[length=8pt, sep]}},
 cut/.style = {blue, thick, dashed}
 ]

\node[mynode] (a1) at (0,0) {};  
\node (a2) at (0,1.5) {$\vdots$}; 
\node[mynode] (a3) at (0,3) {}; 
\node[mynode] (a4) at (0,4.5) {}; 
\node (vs) at (-0.8,0) {$v_{k-1}$};  

\node (v3) at (-0.7,3) {$v_1$}; 
\node (v4) at (-0.7,4.5) {$u$};

\node[mynode] (b1) at (2,0) {};  
\node (b2) at (2,1.5) {$\vdots$}; 
\node[mynode] (b3) at (2,3) {}; 
\node[comp] (b4) at (2,4.5) {};

\node (f1) at (3,0) {$F_k$};  

\node (f3) at (3,3) {$F_{2}$}; 
\node (f4) at (3,4.5) {$F_1$};

\begin{scope}
\draw[thick, rounded corners=5pt]
($(a1) + (-0.4,-0.4)$) rectangle ($(a4) + (0.4,0.4)$);
\end{scope}

\draw[dedge] (a1) to (b1);
\draw[dedge] (a1) to (b2);
\draw[dedge] (a2) to (b2);
\draw[dedge] (a2) to (b3);
\draw[dedge] (a3) to (b2);
\draw[dedge] (a3) to (b3);
\draw[dedge] (a3) to (b4);
\draw[dedge] (a4) to (b3);
\draw[dedge] (a4) to (b4);
\draw[dedge] (a2) to (b1);

\draw[cut] 
(2,5.2) arc (90:270:0.7);
\draw[cut] 
(2,3.7) arc (90:270:0.7);
\draw[cut] 
(2,0.7) arc (90:270:0.7);

\end{tikzpicture}
  \caption{$(D_2,S_2^t)$ with a barrier $X$ circled on the left; dicuts $F_1,\ldots, F_k$ are in dashed. All $F_i$'s except one are trivial.} 
  \label{fig:find-equiv-dicut}
\end{figure}
\end{cproof}
    \begin{CL}\label{claim:tight-source-maximal}
$\widebar{S_2^t}=S_2^t$.
    \end{CL}
    \begin{cproof}
    If $(D_2,S_2^t)$ has no nontrivial barrier dicut then it follows from \Cref{prop:new_tight} that $\widebar{S_2^t}=S_2^t$ and we are done. Otherwise, we find a nontrivial barrier dicut $F_j$ as above that is a maximal contractible dicut. Contract $F_j$ it to obtain $(D_{21},S_{21}^t)$ and $(D_{22},S_{22}^t)$. Pick an arbitrary $s\in S_2\setminus S^t_2$. Observe that $S_2\setminus S_2^t\subseteq S_{22}\setminus S_{22}^t$. This is because $S_{21}\subseteq S_2^t$ and all nodes in $S_2\setminus S_2^t$ belong to $(D_{22},S_{22}^t)$. Thus, $s\in S_{22}\setminus S_{22}^t$. By induction hypothesis $s$ is not a tight node in $(D_{22},S_{22}^t)$, i.e, there is a tight dijoin $J'$ such that $|J'\cap \delta(s)|>1$. Suppose $J'\cap F_j=\{e\}$. Since $(D_{21}, S_{21}^t)$ is tight dijoin-covered, it has a tight dijoin $J''$ such that $J''\cap F_j=\{e\}$. By \Cref{de-composition}, we can combine $J:=J'\cup J''$ into a tight dijoin of $(D_2,S_2^t)$ such that $|J\cap \delta(s)|>1$, which implies $s\notin \widebar{S_2^t}$. This completes the proof that $\widebar{S_2^t}=S_2^t$.
    \end{cproof}
Combining \Cref{claim:near-brick} and \Cref{claim:tight-source-maximal}, we have proved part (a): if $C$ is a maximal stable dicut, then $C$ is good. Now, we consider the case where $C$ is a maximal contractible dicut and prove (b) in the following.
    \begin{CL}
        If $C$ is not equivalent to a stable dicut, then $(D_1,S_1^t)$ is a near-brick and $\widebar{S_1^t}=S_1^t$. 
    \end{CL}
    \begin{cproof}
    The proof is identical to the one for $(D_2,S_2^t)$ up to the case where $(D_1,S_1^t)$ has a nontrivial barrier, where we conclude that exactly one $F_j$ is non-tight. The complication here is that $F_j$ could be a contractible dicut or a stable dicut. However, since $C$ is maximal and $F_j$ dominates $C$, $F_j$ is equivalent to $C$, which implies by assumption that $F_j$ is not stable. Thus, we can contract $F_j$ and use the same inductive proof to show that $(D_1,S_1^t)$ is a near-brick and $\widebar{S_1^t}=S_1^t$. We omit the details here.
    \end{cproof}
This completes the proof of (b). Finally, we prove (c) in the following.
\begin{CL}
    Starting from an arbitrary contractible dicut, one can reach a good dicut by a sequence of at most $|A|+|V|$ contractible or stable dicuts, each dominating the previous one.
\end{CL}
\begin{cproof}
    The proof gives an algorithm to find a good dicut in the following. First, starting from an arbitrary contractible dicut $C$, if $C$ is not good, we can find another contractible or stable dicut $F_i\succ C$ such that there is a tight dijoin $J_2$ with $|J_2\cap C|=2$ and $|J_2\cap F_i|=1$. We claim that the process terminates in $|A|$ steps. Let $C=C_1\prec\ldots \prec C_t$ be the dicuts found by the algorithm, where $C_t$ is a maximal contractible or stable dicut. We show that the characteristic vectors $\mathbf{1}_{C_i}\ i\in [t]$ are linearly independent, which implies $t\leq |A|$. Suppose for $\lambda\in \R^t$ such that $\sum_{j=1}^t\lambda_j \mathbf{1}_{C_j}=\mathbf{0}$. Let $J_0$ be a tight dijoin such that $|J_0\cap C_1|=1$, which exists because $C_1$ is contractible. The dominance relationships imply that $|J_0\cap C_1|=|J_0\cap C_2|=\ldots =|J_0\cap C_t|=1$. Therefore, \[0=\sum_{j=1}^t\lambda_j \mathbf{1}_{C_j}(J_0)=\sum_{j=1}^t\lambda_j.\]
    We inductively prove that $\lambda_j=0\ \forall j\in [t]$. Suppose we already have $\lambda_1=\ldots=\lambda_{i-1}=0$. Let $J_i$ be the tight dijoin from the proof such that $|J_i\cap C_i|=2$ and $|J_i\cap C_{i+1}|=1$. The dominance relationships imply that $|J_i\cap C_{i+1}|=|J_i\cap C_{i+2}|=\ldots =|J_i\cap C_t|=1$. Therefore, \[0=\sum_{j=i}^t\lambda_j \mathbf{1}_{C_j}(J_i)=2\lambda_i+\sum_{j=i+1}^t\lambda_j=\lambda_i+\sum_{j=i}^t\lambda_j.\]
    They together imply that $\lambda_i=0$. This completes the proof that $C_1,\ldots,C_t$ are linearly independent. Second, if the maximal dicut $C_t$ is good, then we are done. Otherwise, by (a) and (b), it is contractible and equivalent to some stable dicut $C_t'$. It takes at most $|V|$ more steps to find $C_t'$ because every time we contract the in-shore of a maximal contractible dicut and $|V_1|$ strictly decreases.
    \end{cproof}
\end{proof}

The next theorem puts everything established so far together and finishes the basis construction for bricks. The idea is that if $(D,S^t)$ is not robust, we can find a separating dicut. Then we use \Cref{thm:separating-contractible} and \Cref{thm:separating-good} to find a good dicut $C$. It follows from a simple dimension counting that $C$ being a good dicut guarantees that the dimension of $\lin(\{J\in \zJ(D,S^t): |J\cap C|=1\})$ is at least the dimension of $\lin(\zJ(D,S^t))$ minus one. Together with \Cref{lemma:integral_basis} and \Cref{lemma:continuity}, this gives a basis for $\lin(\zJ(D,S^t))$.
\begin{theorem}\label{thm:basis_brick}
    Let $(D,S^t)$ be a brick. Then, we can construct an integral basis $\zB$ for $\lin(\zJ(D,S^t))$ that consists of tight dijoins in polynomial time. Moreover, $|\zB|=|A|-|V^t|+1$.
\end{theorem}
\begin{proof}[Proof of \Cref{thm:basis_brick}]
    We prove the theorem by induction on the lexicographic pair $(|V|,\ |S\setminus S^t|)$. In each recursive call, either $|V|$ decreases, or $|V|$ stays the same and $|S\setminus S^t|$ decreases. The base case is when $(D,S^t)$ is a robust digraft, since when $|S^t|=|S|-1$, the digraft is elementary and thus robust. We use \Cref{thm:testing-robust} to test whether $(D,S^t)$ is robust. If it is robust we use \Cref{thm:robust_basis} to construct a basis. Otherwise, we find a separating dicut. We then use \Cref{thm:separating-contractible} to find a contractible dicut. Then we use \Cref{thm:separating-good} to find a good dicut $C$. 
    
    If $C=\delta^+(u)$ is a trivial dicut induced by some $u\in S\setminus S^t$, then we let $(D,(S^t)':=S^t\cup\{u\})$ be a new digraft, which is a near-brick satisfying $\widebar{(S^t)}'=(S^t)'$. If $(D,(S^t)')$ is a brick, then $|(S^t)'|>|S^t|$ and we use the induction hypothesis to construct a basis $\zB'$ with $|\zB'|=|A|-|\widebar{(V^t)'}|+1$ for it. Otherwise, we further reduce $(D,(S^t)')$ through tight dicut decomposition to a brick smaller than $(D,S^t)$ and a bunch of braces. Using the induction hypothesis on the smaller bricks, \Cref{CO:basis_brace} on the braces, and \Cref{lemma:non-basic}, we can construct a basis $\zB'$ for $(D,(S^t)')$ with $|\zB'|=|A|-|\widebar{(V^t)'}|+1$. Since $(D,S^t)$ is a brick, $\widebar{S^t}=S^t$ and thus $u$ is not tight. We use \Cref{lemma:continuity} to find a tight dijoin $J_0$ of $(D,S^t)$ such that $|J_0\cap C|=2$. Let $\zB:=\zB'\cup \{J_0\}$. It follows from \Cref{lemma:integral_basis} that $\zB$ is an integral basis for $\lin\big(\zJ(D,(S^t)')\cup \{J_0\}\big)$. Moreover, since $(D,S^t)$ is a brick, the dimension of $\lin(\zJ(D,S^t))$ is at most $|A|-|V^t|+1$. This is because in the tight dijoin polytope, the equalities defined by tight nodes $v\in V^t$ are linearly independent. Therefore,
    \[
    \begin{aligned}
        |\zB|=&~\dim\big(\lin\big(\zJ(D,(S^t)')\cup \{J_0\}\big)\big)\leq \dim\big(\lin\big(\zJ(D,S^t)\big)\big)\leq |A|-|V^t|+1\\
        =&~|A|-(|(V^t)'|-1)+1=|A|-(|\widebar{(V^t)}'|-1)+1=|\zB'|+1=|\zB|.
    \end{aligned}
    \]
    We conclude that $|\zB|$ is an integral basis for $\lin(\zJ(D,S^t))$. 

    If $C=\delta^+(U)$ is a nontrivial dicut for some $\emptyset\neq U\subsetneqq V$, then we let $(D_1,S_1^t)$, $(D_2,S_2^t)$ be the two $C$-contractions. Both of them are near-bricks satisfying $\widebar{S_i^t}=S_i^t$. 
Since each \((D_i,S_i^t)\) is a near-brick, tight dicut decomposition
reduces it to exactly one brick and some braces. The unique brick has fewer vertices than \((D,S^t)\). Hence, by the induction hypothesis on the
smaller brick, \Cref{CO:basis_brace} on the braces, and \Cref{lemma:non-basic}, we can
construct an integral basis \(B_i\) for
\(\operatorname{lin}(\mathcal J(D_i,S_i^t))\), \(i=1,2\), with
    $|B_i|=|A_i|-|V_i^t|+1$.
    It follows from \Cref{lemma:short_basis-going-up} that we can combine $\zB_1$ and $\zB_2$ into an integral basis $\zB'$ for $\mathcal{J}_C:=\{J\in \zJ(D,S^t): |J\cap C|=1\}$. Moreover,
    \[
    \begin{aligned}
|\zB'|=&|\zB_1|+|\zB_2|-|C|\\
=&\dim(\lin(\zJ(D_1,S_1^t)))+\dim(\lin(\zJ(D_2,S_2^t)))-|C|\\
=&(|A_1|-|\widebar{V_1^t}|+1)+(|A_2|-|\widebar{V_2^t}|+1)-|C|\\
=&(|A_1|-|V_1^t|+1)+(|A_2|-|V_2^t|+1)-|C|\\
=&(|A_1|+|A_2|-|C|)-(|V_1^t|-1+|V_2^t|-1)\\
=&|A|-|V^t|.
\end{aligned}
    \]
    Since $(D,S^t)$ is a brick, $C$ is not tight. We use \Cref{lemma:continuity} to find a tight dijoin $J_0$ of $(D,S^t)$ such that $|J_0\cap C|=2$. It follows from \Cref{lemma:integral_basis} that $\zB:=\zB'\cup \{J_0\}$ is an integral basis for $\lin(\zJ_C\cup \{J_0\})$. Moreover, since
    \[
    |\zB|=\dim(\lin(\zJ_C\cup \{J_0\}))\leq \dim(\lin(\zJ(D,S^t)))\leq |A|-|V^t|+1=|\zB'|+1=|\zB|,
    \]
    we conclude that $|\zB|$ is an integral basis for $\lin(\zJ(D,S^t))$. 
\end{proof}

\section{Ear decomposition of elementary digrafts}\label{sec:elemantary}

Recall that $(D,S^t)$ is elementary if it is tight dijoin-covered and  $|S^t|=|S|-1$. In this case, the degree sequence of a tight dijoin $J$ is fixed: $d_J(v)=1\ \forall v\in S^t$ and $d_J(s_0)=|T|-|S^t|=|T|-|S|+1$ for $\{s_0\}:=S\setminus S^t$. On the other hand, it follows from the existence of one tight dijoin and \Cref{lemma:b-matching} that a perfect $b$-matching is a tight dijoin if and only if
\begin{equation} \label{eq:elementary_b-matching}
\begin{aligned}
    b(v)&=1,\quad \forall v\in S^t\\
    b(s_0)&=|T|-|S|+1.
\end{aligned}
\end{equation}
Thus, it follows from characterization (2) of \Cref{thm:robust_digraft} that an elementary digraft is robust. In particular, when $|S|=|T|$, $J\subseteq A$ is a tight dijoin if and only if it is a perfect matching. The lattice of general \emph{$f$-factors} in bipartite graphs is characterized by \cite{rieder1992note}, where the basis does not necessarily consist of $f$-factors. We extend the ear decomposition of matching-covered bipartite graphs to elementary digrafts to construct a basis for $\lat(\zJ(D,S^t))$ that consists of tight dijoins. 

Let $F\subseteq \delta(s_0)$ be an arbitrary simple subset of arcs (no parallel arcs) such that $|F|=|T|-|S|+1$. Let $S_0:=\{s_0\}$ and let $T_0:=N(s_0)\cap F$ be the neighbors of $s_0$ in $F$. Let $(D_0=(S_0\dot\cup T_0,F),S^t_0=\emptyset)$ be the first elementary digraft; note that the underlying undirected graph of $D_0$ may not be $2$-edge-connected, as required in the definition of a digraft. We do not need the assumption of $2$-edge-connectivity throughout this section, so here by slight abuse of notation we say $(D,S^t)$ is a digraft without requiring $D$ to be $2$-edge-connected. We proceed inductively to build elementary digrafts in the following way. Let $(D_{i-1}, S_{i-1}^t)$ be an elementary digraft. 
We find an odd path \(P_i\) whose two endpoints lie in \(V(D_{i-1})\)
and whose internal vertices are disjoint from \(V(D_{i-1})\). We call
\(P_i\) the \(i\)-th \emph{ear}, and set    $D_i:=D_{i-1}\cup P_i$, $S_i^t:=S^t\cap V(D_i).$ We say such $F,P_1,P_2,\ldots, P_r$ is an \emph{ear decomposition} of $(D,S^t)$ if each $(D_i:=F\dot\cup P_1\dot\cup\ldots \dot\cup P_i,S_i^t)$ is an elementary digraft and $(D_r,S_r^t)=(D,S^t)$. Figure \ref{fig:ear-decomp} shows an example of such a process. In the following theorem we show that an ear decomposition can be found constructively.

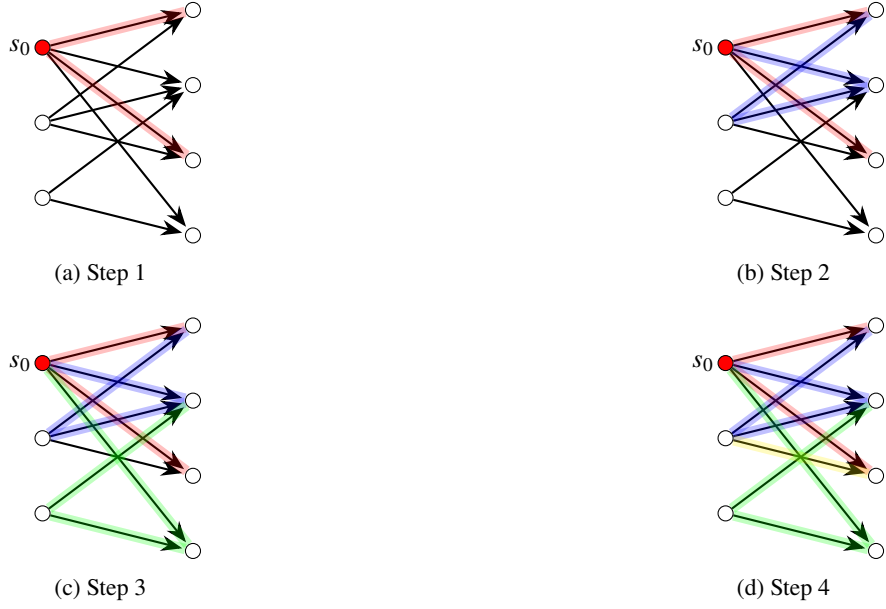
\begin{figure}[ht]
    \centering

    \begin{subfigure}{0.45\textwidth}
        \centering
        \begin{tikzpicture}[
  scale = 1,
  vx/.style = {draw, fill=white, circle, minimum size=2mm, inner sep=0pt},
  vxr/.style = {vx, fill=red}, 
  vxp/.style = {
    vx,
    fill=pink, opacity=0.25, transform shape, scale=1.8, 
  },
  dedge/.style = {thick, -{Stealth[length=8pt, sep]}},
  cut/.style = {blue, thick, dashed},
  ear/.style={line width=4pt, opacity=0.25, red}
]

\node[vxr] (a1) at (-1,2.5) {};  
\node (s) at (-1.3,2.5) {$s_0$};
\node[vx] (a2) at (-1,1.5) {};
\node[vx] (a3) at (-1,0.5) {};

\node[vx] (b1) at (1,3) {}; 
\node[vx] (b2) at (1,2) {};
\node[vx] (b3) at (1,1) {};
\node[vx] (b4) at (1,0) {};

        \draw[dedge] (a1) to (b1);
        \draw[dedge] (a1) to (b2);
        \draw[dedge] (a1) to (b3);
        \draw[dedge] (a1) to (b4);
        \draw[dedge] (a2) to (b1);
        \draw[dedge] (a2) to (b2);
        \draw[dedge] (a2) to (b3);
        \draw[dedge] (a3) to (b4);
        \draw[dedge] (a3) to (b2);

        \draw[ear] (a1) to (b1);
        \draw[ear] (a1) to (b3);

	\end{tikzpicture} \caption{Step $1$}
    \end{subfigure}
    \hfill
    \begin{subfigure}{0.45\textwidth}
        \centering
        \begin{tikzpicture}[
  scale = 1,
  vx/.style = {draw, fill=white, circle, minimum size=2mm, inner sep=0pt},
  vxr/.style = {vx, fill=red}, 
  vxp/.style = {
    vx,
    fill=pink, opacity=0.25, transform shape, scale=1.8, 
  },
  dedge/.style = {thick, -{Stealth[length=8pt, sep]}},
  cut/.style = {blue, thick, dashed},
  ear/.style={line width=4pt, opacity=0.25, red}
]

\node[vxr] (a1) at (-1,2.5) {};  
\node (s) at (-1.3,2.5) {$s_0$};
\node[vx] (a2) at (-1,1.5) {};
\node[vx] (a3) at (-1,0.5) {};

\node[vx] (b1) at (1,3) {}; 
\node[vx] (b2) at (1,2) {};
\node[vx] (b3) at (1,1) {};
\node[vx] (b4) at (1,0) {};

        \draw[dedge] (a1) to (b1);
        \draw[dedge] (a1) to (b2);
        \draw[dedge] (a1) to (b3);
        \draw[dedge] (a1) to (b4);
        \draw[dedge] (a2) to (b1);
        \draw[dedge] (a2) to (b2);
        \draw[dedge] (a2) to (b3);
        \draw[dedge] (a3) to (b4);
        \draw[dedge] (a3) to (b2);

        \draw[ear] (a1) to (b1);
        \draw[ear] (a1) to (b3);
        \draw[ear, blue] (b1) to (a2);
        \draw[ear, blue] (b2) to (a2);
        \draw[ear, blue] (b2) to (a1);

	\end{tikzpicture} \caption{Step $2$}
    \end{subfigure}

    \vspace{1em} 

    \begin{subfigure}{0.45\textwidth}
        \centering
        \begin{tikzpicture}[
  scale = 1,
  vx/.style = {draw, fill=white, circle, minimum size=2mm, inner sep=0pt},
  vxr/.style = {vx, fill=red}, 
  vxp/.style = {
    vx,
    fill=pink, opacity=0.25, transform shape, scale=1.8, 
  },
  dedge/.style = {thick, -{Stealth[length=8pt, sep]}},
  cut/.style = {blue, thick, dashed},
  ear/.style={line width=4pt, opacity=0.25, red}
]

\node[vxr] (a1) at (-1,2.5) {};  
\node (s) at (-1.3,2.5) {$s_0$};
\node[vx] (a2) at (-1,1.5) {};
\node[vx] (a3) at (-1,0.5) {};

\node[vx] (b1) at (1,3) {}; 
\node[vx] (b2) at (1,2) {};
\node[vx] (b3) at (1,1) {};
\node[vx] (b4) at (1,0) {};

        \draw[dedge] (a1) to (b1);
        \draw[dedge] (a1) to (b2);
        \draw[dedge] (a1) to (b3);
        \draw[dedge] (a1) to (b4);
        \draw[dedge] (a2) to (b1);
        \draw[dedge] (a2) to (b2);
        \draw[dedge] (a2) to (b3);
        \draw[dedge] (a3) to (b4);
        \draw[dedge] (a3) to (b2);

        \draw[ear] (a1) to (b1);
        \draw[ear] (a1) to (b3);
        \draw[ear, blue] (b1) to (a2);
        \draw[ear, blue] (b2) to (a2);
        \draw[ear, blue] (b2) to (a1);
        \draw[ear, green] (b2) to (a3);
        \draw[ear, green] (b4) to (a3);
        \draw[ear, green] (b4) to (a1);

	\end{tikzpicture} \caption{Step $3$}
    \end{subfigure}
    \hfill
    \begin{subfigure}{0.45\textwidth}
        \centering
        \begin{tikzpicture}[
  scale = 1,
  vx/.style = {draw, fill=white, circle, minimum size=2mm, inner sep=0pt},
  vxr/.style = {vx, fill=red}, 
  vxp/.style = {
    vx,
    fill=pink, opacity=0.25, transform shape, scale=1.8, 
  },
  dedge/.style = {thick, -{Stealth[length=8pt, sep]}},
  cut/.style = {blue, thick, dashed},
  ear/.style={line width=4pt, opacity=0.25, red}
]

\node[vxr] (a1) at (-1,2.5) {};  
\node (s) at (-1.3,2.5) {$s_0$};
\node[vx] (a2) at (-1,1.5) {};
\node[vx] (a3) at (-1,0.5) {};

\node[vx] (b1) at (1,3) {}; 
\node[vx] (b2) at (1,2) {};
\node[vx] (b3) at (1,1) {};
\node[vx] (b4) at (1,0) {};

        \draw[dedge] (a1) to (b1);
        \draw[dedge] (a1) to (b2);
        \draw[dedge] (a1) to (b3);
        \draw[dedge] (a1) to (b4);
        \draw[dedge] (a2) to (b1);
        \draw[dedge] (a2) to (b2);
        \draw[dedge] (a2) to (b3);
        \draw[dedge] (a3) to (b4);
        \draw[dedge] (a3) to (b2);

        \draw[ear] (a1) to (b1);
        \draw[ear] (a1) to (b3);
        \draw[ear, blue] (b1) to (a2);
        \draw[ear, blue] (b2) to (a2);
        \draw[ear, blue] (b2) to (a1);
        \draw[ear, green] (b2) to (a3);
        \draw[ear, green] (b4) to (a3);
        \draw[ear, green] (b4) to (a1);
        \draw[ear, yellow] (b3) to (a2);

	\end{tikzpicture}\caption{Step $4$}
    \end{subfigure}

    \caption{An example of an ear decomposition. Starting elementary digraft together with the choice of $s_0$ and $F$ are depicted on the top left. The next steps of the ear decomposition are depicted in the remaining pictures.}
    \label{fig:ear-decomp}
\end{figure}

\begin{theorem}\label{thm:ear_decomposition}
    Every elementary digraft $(D,S^t)$ has an ear decomposition. We can construct an integral basis $\zB$ for $\lin(\zJ(D,S^t))$ consisting of tight dijoins in polynomial time. Moreover, $|\zB|=|A|-|V^t|+1$.
\end{theorem}
To prove this, we need some lemmas. The following fact is a formalization of the discussion above.
\begin{lemma}\label{lemma:elementary}
    A digraft $(D,S^t)$ is elementary if and only if $|S^t|=|S|-1$ and letting $\{s_0\}:=S\setminus S^t$,
\begin{subequations}\label{eq:elementary}
        \begin{align}
        |N(X)|-|X|&\geq |T|-|S|+1,\quad \forall s_0\in  X\subsetneqq S
        \label{eq:elementary_1}
        \\
        |N(X)|-|X|&\geq 1, \quad \forall \emptyset \neq X\subsetneqq S, X \subseteq S^t.
        \label{eq:elementary_2}
    \end{align}
    \end{subequations}
\end{lemma}
\begin{proof}
    Assume a digraft $(D,S^t)$ has $|S^t|=|S|-1$. Define $b:S\rightarrow \Z_{\geq 0}$ as $b(v)=1\ \forall v\in S^t$ and $b(s_0)=|T|-|S|+1$.  The digraft $(D,S^t)$ is elmentary if and only if it is tight dijoin-covered, which by \Cref{lemma:tight-dijoin-covered} is equivalent that there exists a tight dijoin. Every tight dijoin is a $b$-dijoin. Thus, a $(D,S^t)$ is elementary if and only if there exists a $b$-dijoin, which by \Cref{lemma:b-matching} is equivalent to \eqref{eq:elementary}.
\end{proof}

Now, we state a helpful fact for a digraft obtained by removing at most two nodes of an elementary digraft.
\begin{lemma}\label{lemma:remove_two_nodes}
    Let $(D,S^t)$ be an elementary digraft and let $u\in S$, $w\in T$ be two arbitrary nodes. For $b \in \Z_{\geq 0}^S$ satisfying \eqref{eq:elementary_b-matching}, define $b':=b-\mathbf{1}_{u}$. Then, there exists a $b'$-matching in the digraft obtained from $(D,S^t)$ by deleting $w$ and deleting $u$ if $b'(u)=0$. 
\end{lemma}
\begin{proof}
    Let $(D'=(S'\dot\cup T',A'),(S^t)')$ be the digraft obtained from $(D,S^t)$ by deleting $w$ and deleting $u$ if $b'(u)=0$. 
    For an arbitrary $\emptyset\neq X\subsetneqq S'$, let $N'(X)$ denote the neighbors of $X$ in $D'$. Then, \[b'(X)\leq b(X)\leq |N(X)|-1\leq |N'(X)|,\]
    where the second inequality follows from \Cref{lemma:b-matching}. The third inequality follows from the fact that we delete at most one node from $N(X)$. Moreover, $b'(S')=b(S)-1=|T|-1=|T'|$. It follows from Hall's theorem that there exists a $b'$-matching in $D'$.
\end{proof}

We are now ready to prove the main theorem of this section.

\begin{proof}[Proof of \Cref{thm:ear_decomposition}]
    Let $(D=(S\dot\cup T,A),S^t)$ be an elementary digraft.
    Let $J_0$ be an arbitrary tight dijoin and let $F:=J_0\cap \delta(s_0)$. Let $S_0:=\{s_0\}$ and $T_0:=N_F(s_0)$. Let $(D_0=(S_0\dot\cup T_0,F),S^t_0=\emptyset)$ be the first elementary digraft. Suppose $(D_{i-1}=(S_{i-1}\dot\cup T_{i-1},A_{i-1}),S_{i-1}^t)$ is an elementary digraft that we have already built. We inductively maintain that $J_0\cap A_{i-1}$ is a tight dijoin of $(D_{i-1},S_{i-1}^t)$. If $D_{i-1}=D$, we are done. Otherwise, there exists $e\in A\setminus A_{i-1}$ that has at least one endpoint in $V(D_{i-1})$, say $u$. Pick an arbitrary tight dijoin $J$ such that $e\in J$. Such a tight dijoin exists because $(D,S^t)$ is tight dijoin-covered. Since each vertex has the same degree in each tight dijoin, the symmetric difference $J_0\Delta J$ is a collection of arc-disjoint cycles, each of which alternates between arcs in $J_0$ and $J$. 
    
    We first note that \(e\notin J_0\). Indeed, if the endpoint of \(e\) in
\(V(D_{i-1})\) is different from \(s_0\), then it is already covered by
the unique \(J_0\)-arc in \(A_{i-1}\). If the endpoint is \(s_0\), then
all arcs of \(J_0\) incident with \(s_0\) lie in
\(F=J_0\cap\delta(s_0)\subseteq A_{i-1}\). Hence no arc of
\(J_0\setminus A_{i-1}\) is incident with \(V(D_{i-1})\). Thus
\(e\in J\setminus J_0\subseteq J_0\Delta J\). Let $P_i$ be the alternating path in $J_0\Delta J$ obtained from starting at $u$, traversing $e$, and following the cycle until we first encounter a vertex in $V(D_{i-1})$. Suppose $f$ is the last arc of $P_i$ with the endpoint $w\in V(D_{i-1})$. 
    
    We claim that $f\in J$. Indeed, there are two cases possible. If $w\neq s_0$, then the only arc of $J_0$ incident to $w$ is in $A_{i-1}$ because $J_0\cap A_{i-1}$ is a tight dijoin of $(D_{i-1},S_{i-1}^t)$. Thus, $f$ had to come from $J$. 
    Alternatively, assume $w=s_0$. Since $f\notin A_{i-1}$, and all arcs of \(J_0\) incident with \(s_0\) belong to \(F=\delta_{J_0}(s_0)\subseteq A_{i-1}\),  we have $f\notin J_0$. The fact that $e,f\in J$ implies that $P_i$ is an odd path. Let $D_i:=D_{i-1}\cup P_i$ and $S_i^t:=S^t\cap V(D_i)$. Then $(D_i,S_i^t)$ is tight dijoin-covered because $J_0\cap A_i$ is a tight dijoin of $(D_i,S_i^t)$. Thus, $(D_i,S_i^t)$ is an elementary digraft. This concludes the proof that $F,P_1,\ldots, P_r$ is an ear decomposition of $(D,S^t)$. 

    We next show how to construct an integral basis $\zB$ for $\lin(\zJ(D,S^t))$. Starting from the last ear $P_r$, we inductively remove one ear and add one tight dijoin to $\zB$. Let $(D_i,S_i^t)$ be the current digraft. We describe how to construct a tight dijoin $J_i$ and remove $P_i$. Let $e,f$ be the first and last arc of $P_i$ and let $u,w$ be the endpoints of $e,f$ such that $u,w\in V(D_{i-1})$. Since $P_i$ is an odd ear, one of $u,w$ is a source and the other is a sink. We may assume w.l.o.g. that $u\in S$ and $w\in T$. Let $b:S_{i-1}\rightarrow \Z_{\geq 0}$ be defined as $b(v)=1\ \forall v\in S_{i-1}^t$ and $b(s_0)=|T|-|S|+1$. Let $b'=b-\mathbf{1}_{u}$.
    Applying \Cref{lemma:remove_two_nodes} to the elementary digraft $(D_{i-1}, S_{i-1}^t)$, we conclude that there exists a $b'$-matching $J'$ in the digraft obtained from deleting $w$ and $u$ if $b'(u)=0$. Let $J$ be obtained from adding odd arcs of $P_i$ to $J'$. Then, $J$ is a tight dijoin of $(D_i,S_i^t)$. Further let $J_i$ be obtained from adding even arcs of $P_{i+1},\ldots, P_r$ to $J$. Clearly, $J_i$ is a perfect $b$-matching for $b$ defined as \eqref{eq:elementary_b-matching}, and thus $J_i$ is a tight dijoin of $(D,S^t)$. This way, we obtain $r$ tight dijoins $J_r,\ldots J_1$. Add them to basis $\zB$. Finally, we add the tight dijoin $J_0$ to $\zB$. Now, we count the number of ears $r$. We claim that $r=|A|-|V|+1$. This is because $|A(F)|-|V(F)|=-1$ and for every ear $P_i$, we add $|P_i|$ new arcs and $|P_i|-1$ new nodes to $D_i$. Thus, $|A|-|V|=|A(F)|-|V(F)|+\sum_{i=1}^r (|P_i|-(|P_i|-1))=r-1$. Thus, $|\zB|=r+1=|A|-|V|+2$. Since $b$-matchings can be found in polynomial time and the total number of ears is at most $r=|A|-|V|+1$, this approach yields a polynomial-time algorithm to construct $\zB$.

    It remains to show that $\zB$ is indeed an integral basis for $\lin(\zJ(D,S^t))$. First, we prove that $\zB$ is linearly independent. Observe that for $i\in [r]$, $J_i$ is the only one in $\{J_0, J_1 \ldots, J_i\}$ that uses the odd arcs of the ear $P_i$. Let $e_i$ be any odd arc of $P_i$ and let $e_0$ be any arc of the starting set $F$. Considering the $(r+1)\times (r+1)$ incidence matrix $M$ whose columns are indexed by $\zB=\{J_0, J_1,\ldots, J_r\}$ and whose rows are indexed by $O:=\{e_0, e_1, \ldots, e_r\}$. As argued above, $M$ is an upper-triangular matrix whose diagonal entries are all equal to $1$. This implies $J_0,J_1,\ldots,J_r$ are linearly independent. Moreover, since $\dim(\lin(\zJ(D,S^t)))=|A|-|V^t|+1=|A|-|V|+2=|\zB|$, $\zB$ forms a linear basis for $\lin(\zJ(D,S^t))$. Finally, $M$ is invertible with its inverse having integer entries. For any $x\in \lin(\zJ(D,S^t))\cap \Z^A$, the equation $M\lambda=x|_O$ has an integer solution $\lambda \in \Z^{r+1}$. 
    This completes the proof that $\zB$ is an integral basis for $\lin(\zJ(D,S^t))$.
\end{proof}

As a special case of elementary digrafts, we show how to construct bases for braces.
\begin{CO}\label{CO:basis_brace}
    Let $(D,S^t)$ be a brace. We can construct an integral basis $\zB$ for $\lin(\zJ(D,S^t))$ consisting of tight dijoins in polynomial time. Moreover, $|\zB|=|A|-|\widebar{V^t}|+2$.
\end{CO}
\begin{proof}
    Let $(D,S^t)$ be a brace. Let  $(S^t)'\subseteq S$ be an arbitrary subset of sources. Since $|S|=|T|$, every tight dijoin of $(D,(S^t)')$ has degree exactly $1$ on each $v\in V$, which means $\widebar{(S^t)'}=S$. This means the set of tight dijoins of $(D,S^t)$ is independent of the choice of $S^t$. Therefore, we may assume without loss of generality that $|S^t|=|S|-1$. In particular, $(D,S^t)$ is elementary. Therefore, we can use \Cref{thm:ear_decomposition} to construct an integral basis $\zB$ for $\lin(\zJ(D,S^t))$. Moreover, $|\zB|=|A|-|V^t|+1=|A|-(|V|-1)+1=|A|-|\widebar{V^t}|+2$. 
\end{proof}

\section{Parity-constrained strongly connected orientations}\label{sec:parity}
Recall the parity strongly connected orientation (parity SCO) problem. Given a $2$-edge-connected undirected graph $G=(V,E)$ whose bidirected graph is $\vec{G}=(V,E^+\cup E^-)$, a family $\zF\subseteq 2^V\setminus \{\emptyset,V\}$, and a subset of red arcs $R\subseteq E^+\cup E^-$, we say a tight SCO of $(G,\zF)$ is \emph{even (odd)} if it has an odd (even) number of red arcs. The goal is to find a tight even (odd) SCO of $(G,\zF)$ if it exists. 
\begin{proof}[Proof of \Cref{thm:parity-sco}]
   First, we reduce parity SCO to \emph{odd SCO}, i.e., find a SCO with an odd number of red arcs. Suppose we are looking for a tight even SCO of $(G,\zF)$. We may assume without loss of generality that at least one $u\in V$ has $\{u\},V\setminus \{u\}\notin \zF$, since the out (in) degree of $u$ in any strongly connected orientation is implied by the out (in) degrees of vertices in $V\setminus \{u\}$. Let $G'$ be obtained from adding a new node $w$ and two parallel edges $e,f=\{u,w\}$, so that $e^+,f^+=(u,w)$ and $e^-,f^-=(w,u)$. Color $e^+,f^+$ red. Let 
   \[\zF':=\Big\{U:U\in \zF, u\notin U\Big\}\bigcup \Big\{U\cup \{w\}:U\in \zF, u\in U\Big\}.\]
   Every tight SCO $O'$ in $(G',\zF')$ is the union of a tight SCO $O$ in $(G,\zF)$ and either $\{e^+,f^-\}$ or $\{f^+,e^-\}$. Moreover, $O$ is a tight even SCO of $(G,\zF)$ if and only if $O'$ is a tight odd SCO of $(G',\zF')$.
    
   Now, we describe how to solve the odd SCO problem. 
   Let $G=(V,E)$ be a $2$-edge-connected graph, whose bidirected graph is $\vec{G}=(V,E^+\cup E^-)$. We use \Cref{main-sco-theorem} to construct a basis $\zB$ for $\lat(\sco(G,\zF))$ that consists of tight SCO's. Since every tight SCO $O^*$ can be written as $\mathbf{1}_{O^*}=\sum_{O\in \zB} \lambda_O\mathbf{1}_O$ for some $\lambda_O\in \Z\ \forall O\in \zB$, if $(G,\zF)$ has a tight odd SCO $O^*$,
   \[
   \sum_{O\in \zB} \lambda_O\mathbf{1}_O(R)=\mathbf{1}_{O^*}(R)\not\equiv 0 \pmod{2}.
   \]
   Therefore, there is some $O\in \zB$ such that $|O\cap R|$ is odd, which is a desired tight odd SCO.  
\end{proof}

\section*{Acknowledgment}
This work originated from a collaboration that included Ahmad Abdi and G\'erard Cornu\'ejols and the authors are grateful for their invaluable input and feedback on the initial version of this work. We thank Karthik Chandrasekaran for suggesting the applications. We are also grateful to the anonymous reviewers whose comments helped improve the presentation of the paper.

\bibliography{references}{}
\bibliographystyle{splncs04}

\appendix

\section{Reduction to digrafts}\label{sec:reduction-to-digrafts}
In this section, we prove the main \Cref{main-sco-theorem} regarding SCO's using the main \Cref{main-digraft} on digrafts. We then prove \Cref{main-scr-theorem} using \Cref{main-sco-theorem}. We start with a reduction from SCO's in an undirected graph to dijoins in a digraft. A similar reduction appears in \cite{Cornuejols24,abdi2025strongly}.

\begin{lemma}\label{sco->dij}
For an arbitrary $2$-edge-connected graph $G=(V,E)$ and a family $\zF\subseteq 2^V\setminus\{\emptyset, V\}$, we can construct a digraft $(D'=(V',A'),\zF')$ and an identity map $\phi:\Z^{E^+\cup E^-}\rightarrow \Z^A$ such that $O\in \sco(G,\zF)$ if and only if $\phi(O)\in \zJ(D',\zF')$.
\end{lemma}

\begin{proof}[Proof of \Cref{sco->dij}]
Let $\overrightarrow{G}=(V,E^+\cup E^-)$ be a bidirected graph of $G$ and let $D'=(V', A')$ be the digraph obtained from $G$ by replacing every edge $e:=\{r,s\}\in E$, by the two arcs $(r,t_e),(s,t_e)$, where $t_e$ is a new vertex. We have $V' = V\cup \{t_e:e\in E\}$ and $A' = \{(r,t_e),(s,t_e):e=\{r,s\}\in E\}$. Note that the nodes of $G$ are sources in $V'$, while the new nodes in $\{t_e:e\in E\}$ are sinks. Furthermore, the underlying undirected graph of $D'$ is $2$-edge-connected. 

There is a one-to-one correspondence between arcs in $E^+\cup E^-$ and arcs in $A'$: $(s,r)$ corresponds to $(r,t_e)$; $(r,s)$ corresponds to $(s,t_e)$. We define the bijection $\phi:\R^{E^+\cup E^-}\rightarrow \R^A$ to map the value on $E^+\cup E^-$ to its corresponding arcs in $A'$.
 
Given a proper subset $\emptyset\neq U\subsetneqq V$, define $$\varphi(U):=U\cup \{t_e:e=\{r,s\};\, r,s\in U\}.$$ Note that $\delta^+_{D'}(\varphi(U)) = \{(r,t_e):e=\{r,s\}\in \delta_G(U)\}$ and $\delta^-_{D'}(\varphi(U)) = \emptyset$. Thus, $\varphi$ maps every proper subset of $V$ to a dicut of $D'$. Conversely, it can be checked for $U'\subset V'$ that, if $\delta^+_{D'}(U')$ is a minimal dicut of $D'$ such that $|U'|<|V'|-1$, then $U:=U'\cap V$ is a proper subset of $V$, and $\delta^+_{D'}(U') = \delta^+_{D'}(\varphi(U))$. 
 
Using the mapping $\varphi$ defined earlier, it can be easily checked that $O$ is an SCO in $G$ if, and only if, $J:=\phi(O)$ is a 
dijoin of $D'$ such that $|J\cap \delta(t_e)|=1,\, \forall e\in E$.
Let $$\zF':=\{\varphi(U):U\in \zF\}\cup \{V'\setminus t_e:e\in E\}.$$
It follows that $\phi$ is a bijection between tight SCO in $(G,\zF)$ and tight dijoins in $(D',\zF')$.
\end{proof}


We are ready to prove \Cref{main-sco-theorem}.
\begin{proof}[Proof of \Cref{main-sco-theorem}]
    Let $(D',\zF')$ be the digraft and  $\phi$ the identity map from \Cref{sco->dij}. We may assume there is at least one tight $SCO$ in $(G,\zF)$, which implies $(D',\zF')$ has at least one tight dijoin, and thus is tight dijoin-covered. \Cref{main-sco-theorem} follows immediately from \Cref{main-digraft} and the fact that $\phi$ maps tight SCO's of $(G,\zF)$ to tight dijoins of $(D',\zF')$.
\end{proof}

We now prove \Cref{main-scr-theorem} using \Cref{main-sco-theorem}.
\begin{proof}[Proof of \Cref{main-scr-theorem}]
    Let $G$ be the underlying undirected graph of $D$ and let $\vec{G}=(V,A\cup A^{-1})$ be its bidirected graph. By \Cref{main-sco-theorem}, there exist $O_1,\ldots,O_k\in \sco(G,\zF)$ that form an integral basis for $\lin(\sco(G,\zF))$. Let $J_i:=A\setminus O_i$ for each $i\in [k]$, which are tight strengthenings in $(D,\zF)$ by definition. We show that $J_1,\ldots, J_k$ form an integral basis for $\lin(\scr(D,\zF))$.
    
    It follows from the condition $\gcd\{1-|\delta^-(U)|:U\in \zF\}=1$ that $\zF\neq\emptyset$. For each $U\in \zF$, we require that for every $J\in \scr(D,\zF)$, after flipping $J$ there are exactly one arc enters $U$, i.e., $|\delta_J^+(U)|+\big(|\delta^-(U)|-|\delta_J^-(U)|\big)=1$. This implies that
    \[\aff(\scr(D,\zF)\subseteq \big\{x:x(\delta^+(U))-x(\delta^-(U))=1-|\delta^-(U)|\quad \forall U\in \zF\big\}.\]
     Pick an arbitrary $x\in \lin(\scr(D,\zF))\cap \Z^A$. Then it satisfies $x(\delta^+(U))-x(\delta^-(U))=k(1-|\delta^-(U)|)$ for every $U\in \zF$ for some $k\in \R$. Since $\gcd\{1-|\delta^-(U)|:U\in \zF\}=1$, there exists $y\in \Z^{\zF}$ such that $\sum_{U\in \zF}y_U (1-|\delta^-(U)|)=1$, which implies
     \[k=k\cdot \sum_{U\in \zF}y_U (1-|\delta^-(U)|)=\sum_{U\in \zF}y_U \cdot k(1-|\delta^-(U)|)=\sum_{U\in \zF}y_U \cdot \big(x(\delta^+(U))-x(\delta^-(U))\big)\in \Z.\] 
     Define $x'\in \Z^{A\cup A^{-1}}$ as $x'(e):=x(e)$ if $e\in A^{-1}$; $x'(e):=k-x(e^{-1})$ if $e\in A$. Notice that $x'\in \lin(\sco(G,\zF))\cap \Z^{A\cup A^{-1}}$. Then there exists $\lambda\in \Z^k$ such that $x'=\sum_{i=1}^k\lambda_i O_i$, which implies $x=\sum_{i=1}^k\lambda_i J_i$. This completes the proof that $J_1,\ldots, J_k$ form an integral basis for $\lin(\scr(D,\zF))$.
\end{proof}

\section{Reduction to basic digrafts}\label{sec:reduction-to-basic}
The goal of this section is to prove the reduction to basic digrafts, namely \Cref{lemma:non-basic}. 


According to \Cref{lemma:tight_dicut_uncross}, in a digraft $(D,\zF)$, if there exist $U,W\in \zF$ that cross, we can replace them with $U\cap W$ and $U\cup W$ while preserving the set of tight dijoins. It follows from the standard uncrossing argument that we can reduce $\zF$ to a cross-free family in time $\poly(|\zF|,n)$. Therefore, we may assume without loss of generality that $\zF$ is a cross-free family, which consists of at most $2n$ dicut shores. This allows us to define $C$-contractions for general digrafts $(D,\zF)$, which generalize \Cref{def:dicut-contraction}.

\begin{DE}[$C$-contractions]\label{def:dicut-contraction-general}
    Let $C=\delta^+(U)$ be a nontrivial dicut of $(D,\zF)$ such that $\zF\cup \{U\}$ is cross-free. 
    Let $U_1:=V\setminus U$ and $U_2:=U$. Let $D_i=(V_i,A_i)$ be the bipartite digraph obtained from $D$ after contracting $U_i$ into a singleton $u_i$; so $V_i=\{u_i\}\cup U_{3-i}$. Let $$
    \zF_i := \big\{W:W\cap U_i=\emptyset, W\in \zF\cup \{U\}\big\}\cup \big\{(W\setminus U_i)\cup \{u_i\}:U_i\subseteq W, W\in \zF\cup \{U\}\big\}.
    $$ We refer to $(D_i,\zF_i),i=1,2$ as the \emph{$C$-contractions} of $(D,\zF)$. 
\end{DE} 
We need one more lemma which gives us a way to decompose and compose tight dijoins along nontrivial dicuts.
\begin{lemma}[\cite{abdi2024strconn}, Lemma 4.12]\label{de-composition}
    Let $(D=(V,A),\zF)$ be a digraft and $C=\delta^+(U)$ be a nontrivial dicut. Let $(D_i=(V_i,A_i),\zF_i),i=1,2$ be the $C$-contractions of $(D,\zF)$. Then the following statements hold: \begin{enumerate}
            \item {\bf Decomposition:} If $J\in \zJ(D,\zF\cup\{U\})$, then $J_i:=J\cap A_i$ satisfies $J_i\in  \zJ(D_i,\zF_i)$.
        \item {\bf Composition:} If $J_i\in \zJ(D_i,\zF_i),i=1,2$, and $J_1\cap C=J_2\cap C$, then $J:=J_1\cup J_2$ satisfies $J\in \zJ(D,\zF\cup \{U\})$. 
    \end{enumerate}
\end{lemma}

\begin{proof}[Proof of \Cref{lemma:non-basic}]
    We prove this by induction on $|V|$. In the first case, if there is some $U\in \zF$ with $1<|U|<|V|-1$, we contract $C=\delta^+(U)$. This reduces to digrafts of the form $(D,S^t)$. In the second case, where the digraft is already $(D,S^t)$, we use \Cref{CO:tight-dicut-polytime} to check whether $(D,S^t)$ has a nontrivial tight dicut. 
    If not, then $(D,S^t)$ is a basic digraft. We use the assumption in the lemma to construct an integral basis $\zB$ for $\lin(\zJ(D,S^t))$ such that $|\zB|=|A|-|\widebar{V^t}|-b(D,S^t)+2$.
    Otherwise, we find a nontrivial tight dicut $C$ and contract $C$. Now, we describe how to construct an integral basis for $\lin(\zJ(D,\zF))$ from the two $C$-contractions, $(D_i=(V_i,A_i),\zF_i),\ i=1,2$. Suppose $(D_1,\zF_1)$ is obtained from contracting $V\setminus U$ into $u_1$; $(D_2,\zF_2)$ is obtained from contracting $U$ into $u_2$. It follows from the induction hypothesis that we can construct a basis $\zB_i$ for $(D_i,\zF_i),\ i=1,2$. Using \Cref{lemma:short_basis-going-up}, we can combine $\zB_1$ and $\zB_2$ into an integral basis $\zB\subseteq \zJ(D,\zF\cup \{U\})$ for $\lin(\zJ(D,\zF\cup \{U\}))$. Since $\delta^+(U)$ is a tight dicut of $(D,\zF)$, $\zJ(D,\zF\cup \{U\})=\zJ(D,\zF)$. Thus, $\zB$ is an integral basis for $\lin(\zJ(D,\zF))$. Moreover,
    \[
    \begin{aligned}
|\zB|=&|\zB_1|+|\zB_2|-|C|\\
=&\big(|A_1|-|\widebar{V_1^t}|-b(D_1,\zF_1)+2\big)+\big(|A_2|-|\widebar{V_2^t}|-b(D_2,\zF_2)+2\big)-|C|\\
=&\big(|A_1|+|A_2|-|C|\big)-\big(|\widebar{V_1^t}|-1+|\widebar{V_2^t}|-1\big)-\big(b(D_1,\zF_1)+b(D_2,\zF_2)\big)+2\\
=&|A|-|\widebar{V^t}|-b(D,\zF)+2,
\end{aligned}
\]
where the first equality follows from \Cref{lemma:short_basis-going-up}; the second equality follows from the induction hypothesis; the fourth equality follows as a consequence of \Cref{de-composition}. Indeed, since $(D,\zF)$ is tight dijoin-covered and $\delta^+(U)$ is a tight dicut, it follows from \Cref{de-composition} ``decomposition" that both $(D_i,\zF_i)$ are tight dijoin-covered. Moreover, it follows from ``decomposition" that every tight dijoin $J\in\zJ(D,\zF)$ can be decomposed into $J_1\in\zJ(D_1,\zF_1)$ and $J_2\in\zJ(D_2,\zF_2)$; it follows from ``composition" that every tight dijoin of $J_i\in\zJ(D_i,\zF_i)$ can be composed with some tight dijoin $J_{3-i}\in\zJ(D_{3-i},\zF_{3-i})$ into a tight dijoin $J\in\zJ(D,\zF)$. This implies $v\in V_i\setminus \{u_i\}$ is tight in $(D_i,\zF_i)$ if and only if it is tight in $(D,\zF)$. In other words, $\widebar{V^t}=\big(\widebar{V_1^t}\setminus u_1\big)\dot\bigcup \big(\widebar{V_2^t}\setminus u_2\big)$. This completes the proof of the formula for the cardinality of $\zB$.
\end{proof}

\section{Proofs of results in \Cref{Sec:preliminary}}\label{sec:proof-prelim}
\begin{proof}[Proof of \Cref{lemma:minimal_dicuts}]
    For a minimal dicut $\delta^-(U)$, suppose that it is not in the form of $\delta^-(v)$, where $v\in T$. If $U\cap S=\emptyset$, then $\delta^-(U)=\cup_{v\in U}\delta^-(v)$. It implies that either $U$ is a cut induced by a single vertex in $T$ or $U$ is not minimal, a contradiction. Thus, $U\cap S\neq \emptyset$. Let $X=U\cap S$, then $N(X)\subseteq U$ since $\delta^+(U)=\emptyset$. Therefore, $\delta^-(X\cup N(X))\subseteq \delta^-(U)$. Suppose $U\cap T\neq N(X)$. Since $D$ is connected but $E\big(X, (U\cap T)\setminus N(X)\big)=\emptyset$, it must satisfy $E\big(S\setminus X, (U\cap T)\setminus N(X)\big)\neq \emptyset$. However, $E\big(S\setminus X, (U\cap T)\setminus N(X)\big)\subseteq \delta^-(U)$, contradicting the minimality of $\delta^-(U)$. Therefore, $U=X\cup N(X)$.
\end{proof}

\begin{proof}[Proof of \Cref{lemma:b-matching}]
    We first prove the ``only if" part. For any $X\subsetneqq S, X\neq \emptyset$, $\delta^-(X\cup N(X))$ is a nonempty dicut since $D$ is connected. Thus, any dijoin $J$ should satisfy $J\cap \delta^-(X\cup N(X))\neq \emptyset$. If $J$ is a perfect $b$-matching, then this is equivalent to $b(X)\leq |N(X)|-1$ because the number of edges of $J$ entering $X\cup N(X)$ is equal to $|N(X)|-b(X)$, which should be at least $1$. Moreover, it follows from $J$ being a perfect $b$-matching that $b(S)=|T|$.

    Next, we prove the ``if" part. Since $b(X)\leq |N(X)|-1\leq |N(X)|, \ \forall X\subsetneqq S, X\neq \emptyset$ and $b(S)=|T|$, by Hall's theorem, there exists a perfect $b$-matching. We now prove that every perfect $b$-matching $J$ where $b$ satisfies \eqref{eq:b-matching} intersects every dicut. Obviously $J$ intersects all the dicuts of the form $\delta^-(v)$, $v\in T$, since $J$ has one edge incident with each $v\in T$. By Lemma \ref{lemma:minimal_dicuts}, any other minimal dicut is of the form $\delta^-(X\cup N(X))$ for some $X\subsetneqq S, X\neq\emptyset$. Since $J$ is a perfect $b$-matching satisfying the first inequality of \eqref{eq:b-matching}, it has $|N(X)|-b(X)\geq 1$ edges entering $X\cup N(X)$. Thus, $J\cap \delta^-(X\cup N(X))\neq\emptyset$, $\forall X\subsetneqq S, X\neq \emptyset$. Thus, $J$ intersects all minimal dicuts, which means $J$ is a $b$-dijoin. 
\end{proof}

\begin{proof}[Proof of \Cref{lemma:tight-dijoin-covered}]
    ``Only if" is clear. We prove the ``if" part. Let $J$ be a tight dijoin. Define $b:S\rightarrow \Z_+$ such that $b(v)=d_J(v)$, i.e. the number of edges in $J$ that are incident to $v$, for each $v\in S$. Due to the tightness, we have $b(v)=1$, $\forall v\in S^t$. By Lemma \ref{lemma:b-matching}, it also satisfies that $b(X)\leq |N(X)|-1, \forall X\subsetneqq S, X\neq \emptyset$ and $b(S)=|T|$. Fix an arc $e=(u,w)$, let $D'=(S,T')$ be obtained from $D$ by deleting $w$. Let $b':S\rightarrow \Z_+$ be such that $b'(u)=b(u)-1$ and $b'(v)=b(v), \forall v\in S\setminus \{u\}$. Note that $b'$ satisfies $b'(X)\leq |N'(X)|, \forall X\subsetneqq S, X\neq \emptyset$ since $|N'(X)|\geq |N(X)|-1$ and $b(X)\leq |N(X)|-1$ for each $X\subsetneqq S, X\neq \emptyset$. Moreover, $b'(S)=|T'|$. By Hall's condition, there exists a $b'$-matching $J'$ in $D'$. Thus, $J:=J'\cup\{e\}$ is a perfect $b$-matching in $D$. Therefore, again by Lemma \ref{lemma:b-matching}, $J$ is a dijoin in $D$. Moreover, since $b(v)=1, \forall v\in S^t$, $J$ is a tight dijoin. Also, $e$ is contained in $J$. This holds for any $e\in A$, and thus $D$ is tight dijoin-covered.
\end{proof}

\begin{proof}[Proof of \Cref{lemma:continuity}]
    Let $b':S\rightarrow \R$ be the degree sequence of $J'$ in $S$ defined by $b'(v):=d_{J'}(v)$, i.e. the number of edges in $J'$ that are incident to $v$, for each $v\in S$. Similarly, define $b''$ to be the degree sequence of $J''$ in $S$. By Lemma \ref{lemma:b-matching}, a perfect $b$-matching is a tight dijoin if and only if $b$ is an integer point in the following polytope:
        \begin{equation*}
    \begin{aligned}
        P:=\Big\{b\in\R_{\geq 1}^{S}: b(X)&\leq |N(X)|-1, ~\forall \emptyset\neq X\subsetneqq S,\\
        b(S)&=|T|,\\
        b(v)&=1,~\forall v\in S^t.
        \Big\}
    \end{aligned}
    \end{equation*}
    Thus, $b',b''\in P$. Moreover, since the $|J'\cap C|\neq |J''\cap C|$, $C$ is not induced by a single node in $T$. By Lemma \ref{lemma:minimal_dicuts}, $C$ is induced by $Z\cup N(Z)$ for some $\emptyset\neq Z\subsetneqq S$. Also, one has $\lambda'=|J'\cap C|=|N(Z)|-b'(Z)$ and $\lambda''=|J''\cap C|=|N(Z)|-b''(Z)$. Therefore, $b'(Z)=|N(Z)|-\lambda'>|N(Z)|-\lambda''=b''(Z)$. Since $|N(Z)|-1$ is non-decreasing submodular, $P$ is a face of a polymatroid, and thus a polymatroid. By averaging there exists $v\in Z$ such that $b'(v)>b''(v)$. It follows from polymatroid exchange property (see e.g. \cite{welsh1976matroid} Ch 18.2), that there is some $u\in V$ with $b'(u)<b''(u)$ such that $b'-\mathbf{1}_v+\mathbf{1}_u\in P$. In particular, $b'(Z)\geq (b'-\mathbf{1}_v+\mathbf{1}_u)(Z)\geq b'(Z)$. Inductively, there is a sequence of integral $\{b_i\}_{i=1}^{k}\in P$ with $b'=b_1$ and $b''=b_k$ such that $b_i(Z) \geq b_{i+1}(Z)\geq b_i(Z)-1$ for any $i=1,...,k-1$. Therefore, for any $\lambda'\leq\lambda\leq\lambda''$, there is some $i\in\{1,...,k\}$ with $b_i(Z)=|N(Z)|-\lambda$. It follows from \Cref{lemma:b-matching} that there is a $b_i$-dijoin $J$, which satisfies $|J\cap C|=|N(Z)|-b_i(Z)=\lambda$.

    This proof naturally gives an algorithm to find such a tight dijoin, given $J'$ and $J''$.
    Alternatively, if $J'$ and $J''$ are not known beforehand, we can solve 
    $\min_{b\in P} (|N(Z)|-b(Z))$ and $\max_{b\in P} (|N(Z)|-b(Z))$ to obtain optimizers $b'$ and $b''$, resp. Then, any value between the minimum and maximum can be attained, and the corresponding tight dijoin can be found.
\end{proof}

\begin{proof}[Proof of \Cref{lemma:tight_dicut_uncross}]
Since $\delta^+(U),\delta^+(W)$ are dicuts, $\delta^+(U\cap W),\delta^+(U\cup W)$ are also dicuts. Thus, for every tight dijoin $J$ of $(D,\zF)$, $|\delta^+_J(U)|,|\delta^+_J(W)|,|\delta^+_J(U\cap W)|,|\delta^+_J(U\cup W)|\geq 1$. Moreover,
    \[
    |\delta^+_J(U\cap W)|+|\delta^+_J(U\cup W)|=|\delta^+_J(U)|+|\delta^+_J(W)|.
    \]
    Therefore, $|\delta^+_J(U)|=|\delta^+_J(W)|=1$ if and only if $|\delta^+_J(U\cap W)|=|\delta^+_J(U\cup W)|=1$. Thus, the tight dijoins of digrafts $\big(D,\zF\cup \{U\}\cup \{W\}\big)$ and $\big(D,\zF\cup \{U\cap W\}\cup \{U\cup W\}\big)$ are the same.
\end{proof}

\section{Proof of \Cref{thm:barrier_tight_equiv}}\label{sec:proof-tight-dicut}
In this section, we will prove that a tight dijoin-covered digraft has a nontrivial tight dicut if and only if it has a nontrivial barrier or a $2$-separation dicut. Our approach is inspired by an analogy for perfect matchings in \cite{Edmonds82}.

Let $(D,S^t)$ be a tight dijoin-covered digraft. We split the proof into two cases depending on whether $|S|=|T|$. We first deal with the easier case where $|S|=|T|$. This implies that $\widebar{S^t}=S$. It follows from \Cref{lemma:b-matching} that $J\subseteq A$ is a tight dijoin if and only if it is a perfect matching. In particular, the condition given in \eqref{eq:b-matching} with $b:=\mathbf{1}_S$ coincides with the condition for which a bipartite graph is \emph{matching-covered}, i.e., every edge is contained in some perfect matching: a bipartite graph is matching-covered if and only if
\begin{equation}\label{eq:matching-covered}
    |X|\leq |N(X)|-1,\quad \forall \emptyset\neq X\subseteq S: |X|<|S|-1.
\end{equation}
The characterization when a matching-covered bipartite graph is a brace is well-known; see, e.g., Lemma 1.4 in 
\cite{lovasz87}.
\begin{lemma}[\cite{lovasz87}]
    A matching-covered bipartite graph is a brace if and only if 
    \begin{equation}\label{eq:brace}
        |X|\leq |N(X)|-2,\quad \forall \emptyset\neq X\subseteq S: |X|<|S|-1.
    \end{equation}
\end{lemma}
This allows us to prove the following:
\begin{lemma}\label{lemma:brace_characterization}
    Let $(D,S^t)$ be a tight dijoin-covered digraft with $|S|=|T|$. Then, $(D,S^t)$ is a brace if and only if it does not contain any nontrivial barrier dicuts.
\end{lemma}
\begin{proof}
    A nontrivial barrier dicut is a nontrivial tight dicut, thus the ``only if" is trivial. We prove the ``if" direction in the following. Suppose $(D,S^t)$ is not a brace. We prove that there is a nontrivial barrier dicut. Since $(D,S^t)$ satisfies \eqref{eq:matching-covered} but violates \eqref{eq:brace} for some $\emptyset\neq X\subseteq S$ with $|X|<|S|-1$, we have $|X|=|N(X)|-1$. Observe that $N(X)$ forms a nontrivial barrier since after deleting $N(X)$, the nodes in $X$ become singleton components. Together with the components in $\emptyset\neq Y:=V\setminus N(X)\setminus X$, we have $|N(X)|\ge \sigma(V\setminus N(X))\ge |X|+1=|N(X)|$. This implies that $N(X)$ is a barrier and $Y$ is connected. Moreover, $|Y|\ge |S|-|X|>1$, which means that $Y$ induces a nontrivial barrier dicut.
\end{proof}

From now on, we assume $|S|<|T|$. Denote by $\mathscr{K}$ the set of dicuts of $D$. The convex hull of tight dijoins of a digraft $(D,S^t)$ is given as follows \cite{Lucchesi78}:
\begin{equation*}
    \dij(D,S^t):=\left\{x\in \R_{\geq 0}^A: \begin{array}{cc}
   x(C)\geq 1 & \forall C\in \mathscr{K};\\
    x(\delta(v))=1 & \forall v\in V^t
\end{array}\right\}.
\end{equation*}
A dicut $C_0$ is tight if and only if the following linear program (LP) has optimal value $1$: 
\[\max \{\mathbf{1}_{C_0}^\top x: x \in \dij(D,S^t)\}.\]
Its dual can be written as 
\begin{equation}\label{eq:dual_tight}
    \min \left\{\sum_{v\in V^t} y_v - \sum_{C\in \mathscr{K}} z_C: \begin{array}{ccc}
        & y_u+y_v-\sum_{C\in \mathscr{K}:\ e\in C} z_C \geq 1 & \forall e=(u,v) \in C_0, \\
        & y_u+y_v-\sum_{C\in \mathscr{K}:\ e\in C} z_C \geq 0 & \forall e=(u,v) \in A\setminus C_0, \\
        & z_C \geq 0 & \forall C\in \mathscr{K}
    \end{array} \right\}.
\end{equation}
\begin{lemma}\label{lemma:nonneg}
    The optimum value of $\eqref{eq:dual_tight}$ does not change if we additionally impose $y_v\geq 0$ for all $v\in V^t$.
\end{lemma}
\begin{proof}
    For any feasible solution $(y,z)$ of $\eqref{eq:dual_tight}$, if $y_v<0$ for some $v\in V^t$, let $(y',z')$ be obtained from $(y,z)$ by $y'=y-y_v\mathbf{1}_v$ and $z'=z-y_v\mathbf{1}_{\delta(v)}$. Such $(y',z')$ is still feasible to $\eqref{eq:dual_tight}$ and has the same objective value as $(y,z)$. Therefore, we may also impose the constraint $y_v\geq 0$ on $\eqref{eq:dual_tight}$ without changing the optimum.
\end{proof}

We extend the variables $y\in \R_{\ge 0}^{V^t}$ to $y\in \R_{\ge 0}^{V}$ by defining $y_v:=0\ \forall v\in V\setminus V^t$. Let $w_e:=y_u+y_v$ for all arcs $e=(u,v) \in A$.
We call such a $w\in \R_{\ge 0}^A$ \emph{node-induced} by $y$. Then, $(y,z)$ is feasible for $\eqref{eq:dual_tight}$ if and only if $t:=z+\mathbf{1}_{C_0} \in \R_{\ge 0}^\mathscr{K}$ is feasible for the following LP:
\begin{equation}\label{eq:dual_packing}
    \max\left\{\sum_{C\in \mathscr{K}} t_C: \begin{array}{cc}
        \sum_{C\in \mathscr{K}:\ e\in C} t_C\leq w_e & \forall e\in A \\
        t_C \geq 0 & \forall C\in \mathscr{K}
    \end{array} \right\}.\tag{\textsc{D($w$)}}
\end{equation}
Note that \eqref{eq:dual_packing} is an LP that solves the maximum \emph{fractional packing} of dicuts under weight $w$. It is the dual program of the following dicut covering LP, which solves the minimum weight of a dijoin:
\begin{equation}\label{eq:primal_covering}
    \min\left\{w^\top x: 
    \begin{array}{cc}
   x(C)\geq 1 & \forall C\in \mathscr{K} \\
   x_e \geq 0  & \forall e \in A
\end{array}\right\}.\tag{\textsc{P($w$)}}
\end{equation}

\begin{lemma}\label{lemma:strong-dual}
    Let $w$ be node-induced by $y\in \R_{\ge 0}^{V}$ where $y_v=0\ \forall v\in V\setminus V^t$. Then, for any optimal solution $t^*$ of \eqref{eq:dual_packing}, the following holds: $(a)$ $\sum_{C\in \mathscr{K}} t^*_C=\sum_{v\in V^t} y_v$; $(b)$ if $t^*_C>0$ then $C$ is a tight dicut; $(c)$ $\sum_{C\in \mathscr{K}:\ e\in C}t^*_C=w_e$ for all $e\in A$.
\end{lemma}
\begin{proof}
    Notice that \eqref{eq:primal_covering} exactly models the minimum weight of a dijoin under weight $w\in \R^A$. Since each dijoin uses at least one arc incident with each $v \in V^t$, the total weight of a dijoin is at least $\sum_{v\in V^t}y_v$. On the other hand, every tight dijoin $J$ has exactly one arc incident with each vertex in $V^t$, so for $x=\mathbf{1}_J$, the objective value of \eqref{eq:primal_covering} is exactly $\sum_{v\in V^t}y_v$.
    Therefore, $(a)$ follows by strong duality.
    
    To prove $(b)$, observe that $t_{\delta(v)}:=y_v$ for each $v \in V^t$ is a fractional packing of dicuts of value $\sum_{v\in V^t}y_v$, which is optimal for \eqref{eq:dual_packing}. By complementary slackness, for a dijoin $J$, $x=\mathbf{1}_J$ is optimal to \eqref{eq:primal_covering} if and only if it is tight. Therefore, $(b)$ holds by applying complementary slackness again.
    
    To see $(c)$, by assumption $(D,S^t)$ is tight dijoin-covered. Thus, $(c)$ follows from complementary slackness.
\end{proof}

\begin{lemma}\label{lemma:opt_node_induce}
    For every tight dicut $C_0$, there exists a node-induced weight $w\in \R_{\ge 0}^A$ with an optimal solution $t^*$ of \eqref{eq:dual_packing} such that $C_0\in \supp(t^*)$.
\end{lemma}
\begin{proof}
    Since $C_0$ is a tight dicut, there exists an optimal solution $(y^*,z^*)$ to $\eqref{eq:dual_tight}$ of value $1$. In other words, $\sum_{v\in V^t} y^*_v=\sum_{C\in\mathscr{K}}z^*_C+1$. By \Cref{lemma:nonneg}, we may assume $y^*\geq 0$. Let $w\in \R_{\ge 0}^A$ be node-induced by $y^*$ and let $t^*:=z^*+\mathbf{1}_{C_0}$. 
    
    We claim that $t^*$ is an optimal solution to \eqref{eq:dual_packing}. First, $t^*$ is feasible since $(y^*, z^*)$ is feasible for $\eqref{eq:dual_tight}$. Moreover, $\sum_{C\in \mathscr{K}}t^*_C = \sum_{C\in \mathscr{K}}z^*_C+1 = \sum_{v\in V^t}y^*_v$. It follows from \Cref{lemma:strong-dual} $(a)$ that $t^*$ is optimal to \eqref{eq:dual_packing}.
\end{proof}

We are ready to prove \Cref{thm:barrier_tight_equiv}.

\begin{proof}[Proof of \Cref{thm:barrier_tight_equiv}]
The case where $|S|=|T|$ is proved in \Cref{lemma:brace_characterization}, so we may assume $|S|<|T|$. One direction of the theorem is immediate: $2$-separation and barrier dicuts are always tight. For the rest of this section, assume for the sake of contradiction that there exists a digraft $(D,S^t)$ with a nontrivial tight dicut, but no nontrivial barrier or $2$-separation dicuts.

Let $C_0$ be a nontrivial tight dicut. Let $w\in \R_{\ge 0}^A$ be node-induced by $y^*\in\R_{\ge 0}^V$ where $y^*_v=0\ \forall v\in V\setminus V^t$, and let $t^*$ be an optimal solution of \eqref{eq:dual_packing} such that $C_0\in \supp(t^*)$, as given in \Cref{lemma:opt_node_induce}. By scaling, we may assume that $w\in \Z_{\ge 0}^A$ and the optimal packing of dicuts $t^*$ under weight $w$ is integral. Hence we think of $t^*$ as a collection of dicuts (with multiplicity) that includes $C_0$.
Among all $t^*$ optimal to \eqref{eq:dual_packing} and $\supp(t^*)$ contains at least one nontrivial dicut, select one that minimizes $\Phi(t^*):=\sum_{U: \delta^+(U)\in \mathscr{K}} t^*_{\delta^+(U)} |U||V\setminus U|$.

\begin{CL}\label{claim:1}
The family $\{U: \delta^+(U)\in \supp (t^*) \}$ is cross-free.    
\end{CL}
\begin{cproof}
    For contradiction, assume $\supp(t^*)$ contains two crossing sets $U$ and $W$. Obtain $t'$ from $t^*$ by decreasing $t^*_{\delta^+(U)}, t^*_{\delta^+(W)}$ by one and increasing $t^*_{\delta^+(U\cup W)}, t^*_{\delta^+(U\cap W)}$ by one. 
    Since $\mathbf{1}_{\delta^+(U\cup W)}+\mathbf{1}_{\delta^+(U\cap W)}=\mathbf{1}_{\delta^+(U)}+\mathbf{1}_{\delta^+(W)}$, $t'$ is feasible to \eqref{eq:dual_packing}. Moreover, the objective value does not change, which makes $t'$ optimal. 
    
    We claim that $t'$ contains at least one nontrivial dicut. It suffices to show that at least one of $\delta^+(U\cup W)$ and $\delta^+(U\cap W)$ is nontrivial. Suppose for the sake of contradiction that there exists vertices $u,v \in V$, such that $\{u\} = U\cap W$ and $V\setminus \{v\} = U\cup W$. There are no arcs between $U\setminus W$ and $W\setminus U$ since $\delta^+(U)$ and $\delta^+(W)$ are both dicuts, which means $\{u,v\}$ forms a $2$-cut. Moreover, both $u,v$ are tight, i.e., $u,v\in \widebar{V^t}$. This is because dicuts $\delta^+(U\cup W)$ and $\delta^+(U\cap W)$ are in the support of an optimal packing $t'$, and by \Cref{lemma:strong-dual} $(b)$, both of them are tight. We further claim that $u,v\in V^t$. Suppose not, which implies $u\in \widebar{S^t}\setminus S^t$. It follows from \Cref{prop:new_tight} that there is a nontrivial barrier. By the assumption that $|S|<|T|$, it leads to a nontrivial barrier dicut, a contradiction. Therefore, $\{u,v\}$ forms a $2$-separation dicut, a contradiction.

    Finally, notice that $\Phi(t')<\Phi(t^*)$, contradicting the our choice of $t^*$.
\end{cproof}


Among all the choices for $w$ and $t^*$ such that $\supp(t^*)$ contains at least one nontrivial dicut and $\{U: \delta^+(U)\in \supp (t^*) \}$ is cross-free, pick one that minimizes $\sum_{C\in \mathscr{K}}t^*_C$. Since $\supp(t^*)$ contains at least one nontrivial dicut, $t^*\neq 0$, which implies that there exists $e = (u,v)\in A$ such that $w_e>0$. This implies that $i\in V^t$ and $y_i^*>0$ holds for at least one $i\in\{u,v\}$. Pick $R\subseteq V$ such that $\delta(R)\in \supp(t^*)$, $e\in \delta(R)$, and $i\in R$. The existence of such $R$ is due to \Cref{lemma:strong-dual}~(c). Moreover, we know that $\delta(R)$ is a tight dicut due to \Cref{lemma:strong-dual}~$(b)$.
    
    Among all such choices of $(e,R)$, choose a pair such that $|R|$ is minimized. Assume w.l.o.g. that $i=u$, that is, $y^*_u>0$ and $\delta^+(R)\in \supp(t^*)$ (the proof of the other case is identical, and we omit it).

\begin{CL}\label{claim:2}
    $|R|>1$.
\end{CL}
\begin{cproof}
    Suppose towards a contradiction that $R=\{u\}$. Then decrease each of $y^*_u,t^*_{\delta^+(R)}$ by $1$, and update $w$ accordingly. According to \Cref{lemma:strong-dual}~$(a)$, $t^*$ remains an optimal solution for \eqref{eq:dual_packing}. It is clear that $\supp(t^*)$ remains cross-free and contains at least one nontrivial dicut. However, $\sum_{C\in \mathscr{K}}t^*_C$ has decreased by $1$, a contradiction to our choice of $t^*$.
\end{cproof}

Let $R_1, \ldots, R_k$ be the maximal subsets of $R$ such that each $\delta^-(R_i) \in \supp(t^*)$, for some integer $k\geq 0$. According to \Cref{lemma:strong-dual}~$(b)$, $\delta^-(R_i)$ is tight for all $i\in[k]$. Furthermore, let $R^+ := \{v \in R\setminus (R_1\cup\ldots\cup R_k): y^*_v>0\}$ and $R^0 = R \setminus (R^+\cup R_1\cup\ldots\cup R_k)$ (see an illustration in \Cref{fig:placeholder}). Notice that $R^+\subseteq V^t$ by the definition of $y^*$. In the following, we show that $R^+, R^0, R_1,\ldots, R_k$ is a partition of $R$, and then we prove several properties of this partition.
    
\begin{CL} 
    The $R_i$'s are pairwise disjoint.
\end{CL}
    \begin{cproof}
    If not, then there exists $R_i\cap R_j\neq \emptyset$. We also have $R_i\cup R_j\subseteq R\subsetneq V$. Since $\{U: \delta^+(U)\in \supp (t^*) \}$ is cross-free, $R_i$ and $R_j$ do not cross. Thus, one of $R_i$, $R_j$ is contained in the other, contradicting the maximality of the $R_i$'s.
    \end{cproof}
    
\begin{CL}\label{claim:arcs1} 
        There is no arc between distinct $R_i$'s, and there is no arc between $R^0$ and any of $R_i$.
\end{CL}
    \begin{cproof}
        We have no arc between different $R_i$'s since otherwise at least one $\delta^-(R_i)$ would not be a dicut. To prove the second statement, assume for the sake of contradiction that there is an arc $f=(a,b)$ between $R^0$ and $R_i$. We may assume w.l.o.g. that $a \in R^0$ and $b\in R_i$. Then $\delta^-(R_i)\in \supp(t^*)$ implies that $w_{f}>0$ and hence $y^*_b>0$. Therefore, instead of the pair $(e,R)$, we should have picked $(f,R_i)$ given that $|R_i|<|R|$, a contradiction.
    \end{cproof} 
    
\begin{CL}\label{claim:arcs2} 
        There is no arc $(a,b)$ with $a,b\in R^+$ or with $a\in R^+$ and $b\in R^0$.
\end{CL}
    \begin{cproof}
    Suppose otherwise. Let $f = (a,b)$ be an arc such that $a\in R^+$ and $b\in R^+\cup R^0$. Since $y^*_a>0$, we have $w_f>0$, so $f$ must be contained in some dicut $\delta^+(U)\in \supp(t^*)$ due to \Cref{lemma:strong-dual}~$(c)$. Notice that $a\in U\cap R$ and $b\in R \setminus U$, so since $\supp(t^*)$ is cross-free, we have either $U\subsetneq R$ or $U\cup R = V$.
        If $U\subsetneq R$, then the pair $(f,U)$ should have been picked instead of $(e,R)$ as $|U|<|R|$. Otherwise, $b \in V\setminus U \subseteq R$ and $\delta^-(V\setminus U) \in \supp(t^*)$, so $b$ must have belonged to one of $R_i$. Both cases lead to a contradiction.
    \end{cproof}

\begin{figure}
    \centering
    \begin{tikzpicture}[
    font=\small,
    dot/.style={circle,fill,inner sep=1.6pt},
    arc/.style={thick,-{Latex[length=2mm]}},
    cutarc/.style={thick,dashed,-{Latex[length=2mm]}},
    setbox/.style={draw,rounded corners,thick,inner sep=10pt},
    dsetbox/.style={draw,rounded corners,thick,dashed,inner sep=6pt}
]

\node[dot,label=left:$u$]       (u)  at (0, 1.55) {};
\node[dot] (s2) at (0, 0.95) {};
\node at (0,0.35) {$\vdots$};
\node[dot] (sk) at (0, -0.5) {};

\node[dsetbox,minimum width=1cm,minimum height=1cm] (Rzero)
    at (0,-2) {$\emptyset$};

\node[dot] (r11) at (3.05, 1.55) {};
\node[dot] (r12) at (3.45, 1.15) {};

\node[dot] (r21) at (3.05, 0.15) {};
\node[dot] (r22) at (3.45,-0.25) {};

\node at (3.25,-1.00) {$\vdots$};

\node[dot] (rk1) at (3.05,-2.05) {};
\node[dot] (rk2) at (3.45,-2.45) {};

\node[dot,label=right:$v$] (v) at (3.05, 3.95) {};
\node[dot] (vv) at (3.45,3.55) {};
\begin{scope}[on background layer]
    \node[setbox,fit=(u)(s2)(sk)] (Rplus) {};
    \node[setbox,fit=(r11)(r12)] (Rone) {};
    \node[setbox,fit=(r21)(r22)] (Rtwo) {};
    \node[setbox,fit=(rk1)(rk2)] (Rk) {};
    \node[setbox,fit=(v)(vv)] (VminusR) {};

    \node[setbox,inner xsep=30pt, inner ysep=15pt,
          fit=(Rplus)(Rzero)(Rone)(Rtwo)(Rk)] (Rbox) {};
\end{scope}

\node[anchor=south west] at ($(Rbox.north west)+(0.05,0.05)$) {$R$};
\node[anchor=south] at ($(Rplus.north)+(0,0)$) {$R^+$};

\node[anchor=east] at ($(Rzero.west)+(-0.04,0)$) {$R^0$};
\node[anchor=west] at ($(Rone.east)+(0.04,0)$) {$R_1$};
\node[anchor=west] at ($(Rtwo.east)+(0.04,0)$) {$R_2$};
\node[anchor=west] at ($(Rk.east)+(0.04,0)$) {$R_k$};

\node[anchor=west] at ($(VminusR.east)+(0.18,0)$) {$V\setminus R$};


\draw[arc] (u) to[bend left=7] node[above] {$e$} (v);

\draw[arc] (s2) to[bend left=7] (r11);
\draw[arc] (sk) to[bend right=4] (r21);
\draw[arc] (u)  to[bend left=15] (rk1);

\draw[decorate,decoration={brace,mirror,amplitude=5pt},thick]
    ($(Rplus.north west)+(-0.15,0.0)$) --
    ($(Rplus.south west)+(-0.15,-0.0)$);

\node[anchor=east] at ($(Rplus.west)+(-0.3,0)$) {$k+1$};

\end{tikzpicture}
    \caption{$R^+, R^0, R_1,\ldots, R_k$ partitions $R$. $R^+\subseteq S^t$ is a barrier that induces barrier dicuts $\delta^-(R_1)$, $\delta^-(R_2)$, ..., $\delta^-(R_k)$, and $\delta^-(V\setminus R)$.}
    \label{fig:placeholder}
\end{figure}
    
    \begin{CL} 
        $R^0=\emptyset$ and $R^+\subseteq S^t$.
    \end{CL}
    \begin{cproof}
        Notice that any shore of a tight dicut must induce a connected undirected graph (if a shore of a dicut has $p>1$ connected components, then the dicut must intersect every dijoin at least $p>1$ times and thus not tight).
        Thus, $R$ induces a connected undirected graph. However, it follows from \Cref{claim:arcs1} and \Cref{claim:arcs2} that $R^0\cup (R^+\cap T)$ is disconnected from the rest of $R$. Thus one of $R^0 \cup (R^+\cap T)$ or $R\setminus (R^0 \cup (R^+\cap T))$ is empty. Since $y^*_u>0$ and $u$ is a source, we have $u \in R\setminus (R^0 \cup (R^+\cap T)) \neq \emptyset$. Therefore, $R^0=\emptyset$ and $R^+$ has no sinks. Since $R^+\subseteq V^t$, we get that $R^+\subseteq S^t$.
    \end{cproof}
    
    \begin{CL}\label{claim:arcs3}
        There is no arc between $V\setminus R$ and any of $R_i$.
    \end{CL}
    \begin{cproof}
        This follows from $\delta^+(R_i)=\emptyset$ and $\delta^-(R)=\emptyset$.
    \end{cproof}
    \begin{CL}\label{claim:sizeRplus}
        $|R^+|=k+1$.
    \end{CL}
    \begin{cproof}
        Let $J$ be any tight dijoin containing $e$. Then, since $\delta^-(R_i)$ is a tight dicut, there is exactly one arc in $J\cap \delta^-(R_i)$, which has to come from $R^+$ by \Cref{claim:arcs1} and \Cref{claim:arcs3}. Since $\delta^+(R)$ is tight and $J$ already used $e \in \delta^+(R)$, $J$ does not use other arcs between $R$ and $V\setminus R$. This allows us to conclude that $|J\cap \delta^+(R^+)|=k+1$. On the other hand, $|J\cap \delta^+(R^+)| = |R^+|$ since all vertices in $R^+\subseteq V^t$ and  there is no arc inside $R^+$ by \Cref{claim:arcs2}. Therefore, we obtain $|R^+|=k+1$.
    \end{cproof}    
    
     This lets us conclude that $R^+$ is a barrier in $D$. Notice that $|R^+|\geq 2$, since otherwise $k=0$ and $|R|=|R^+|=1$, which is a contradiction to \Cref{claim:2}. Since $D$ does not contain any nontrivial barrier dicuts, $R_1,\ldots ,R_k$ and $V\setminus R$ are all singletons. Therefore, $S=R^+$ and $T=\{R_1,\ldots,R_k,V\setminus R\}$. Then \Cref{claim:sizeRplus} contradicts the assumption that $|S|<|T|$.
\end{proof}

\section{Proofs of results in \Cref{sec:tight-dicut-decomp}}\label{sec:proof-tight-dicut-decomp}

\begin{lemma}\label{lemma:submod-comp}
    Let $f(X):=|X|-\sigma(V\setminus X)$ for every subset $X\subseteq T$ ($X\subseteq S^t$). Then $f$ is submodular.
\end{lemma}
\begin{proof}
     We prove that $f$ has non-increasing margins, i.e. for any sets $X\subseteq Y \subseteq T $ (the proof for $S^t$ is analogous) and any vertex $v \in T\setminus Y$, we have \begin{equation*}
        f(X\cup \{v\})-f(X) \geq f(Y\cup \{v\})-f(Y).
    \end{equation*}
    Expanding, this is equivalent to \begin{equation*}
        1- \sigma(V\setminus (X\cup \{v\}))+\sigma(V\setminus X) \geq 1- \sigma(V\setminus (Y\cup \{v\}))+\sigma(V\setminus Y).
    \end{equation*} Moreover, $\sigma(V\setminus (X\cup \{v\}))-\sigma(V\setminus X)= \rho(V\setminus X,v)$ where by $\rho(V\setminus X,v)$ we denote the number of distinct components of $G-X-\{v\}$ that are connected to $v$ by at least one arc. Now, it suffices to prove that \begin{equation*}
        \rho(V\setminus X,v)\leq \rho(V\setminus Y,v),
    \end{equation*} which is true because every component of $G-X-\{v\}$ will stay a separate component in $G-Y-\{v\}$ since all neighbors of $v$ are in $S$ and $Y$ only contains vertices of $T$. Thus, this number can only increase.
\end{proof}

\begin{proof}[Proof of \Cref{CO:tight-dijoin-polytime}]
    Let $f(X):=|X|-\sigma(V\setminus X)$ for every $X\subseteq T$, and $g(X):=|X|-\sigma(V\setminus X)$ for every $X\subseteq S^t$. It suffices to check whether $\min\{f(X):\emptyset \neq X\subseteq T\}$ is at least $0$, and $\min\{g(X):\emptyset \neq X\subseteq S^t\}$ is at least $0$, which can be solved in polynomial time using submodular minimization \cite{IwataFleischerFujishige2001,Schrijver2000}.
\end{proof}

\begin{proof}[Proof of \Cref{CO:tight-dicut-polytime}]
    To find a nontrivial tight dicut it suffices to check existence of nontrivial barrier or $2$-separation dicuts.
    For $2$-separation dicuts, it suffices to check, for every pair of $u\in S^t, v\in T$, whether $D-\{u,v\}$ is disconnected.

    For barrier dicuts, let $f(X):=|X|-\sigma(V\setminus X)$ for every $X\subseteq T$, and $g(X):=|X|-\sigma(V\setminus X)$ for every $X\subseteq S^t$. To find a nontrivial sink barrier, it suffices to find a subset $X\subseteq T$ such that $2\le |X|\le |T|-1$ minimizing $f(X)$. The reason why we enforce $|X|\le |T|-1$ is because if $|S|<|T|$, $X=T$ is never a barrier; if $|S|=|T|$, we need to avoid finding $X=T$ as a barrier, in which case no nontrivial barrier dicut is induced. This can be done by solving $\min\{f(X): u,v\in X, w\notin X\}$ for all triples $u,v,w\in T$. By \Cref{lemma:submod-comp}, $f(X)$ is submodular and therefore its restriction to $\{X\subseteq T: u,v\in X,w\notin X\}$ should be submodular. To find a source barrier, it suffices to find a subset $X\subseteq S^t$ such that $|X|\ge 2$, which can be done similarly.
\end{proof}


Finally, we prove \Cref{prop:unique_decomposition}. Our proof is inspired by the proof of uniqueness of tight cut decomposition for matching-covered graphs in \cite{Edmonds82}.

Let $\mathcal{S}\subseteq 2^V\setminus \{\emptyset, V\}$ be a family of tight dicuts out-shores that we use in some tight dicut decomposition procedure. Then $\mathcal{S}$ is cross-free by construction. Furthermore, notice that any tight dicut in the contracted digraft corresponds to a tight dicut in the original digraft. This means that any tight dicut decomposition corresponds to a \emph{maximal} cross-free family of tight dicut out-shores. 

\begin{proof}[Proof of \Cref{prop:unique_decomposition}]

    Consider a digraft $(D=(V,A),\zF)$ where $\zF$ is cross-free. Let $\mathcal{S}$ and $\mathcal{S}'$ be two families of tight dicut out-shores used in two tight dicut decomposition procedures. We denote by $\mathcal{S}\sim \mathcal{S}'$ if decomposing along $\mathcal{S}$ and $\mathcal{S}'$ ends up with the same list of basic digrafts. We prove that $\mathcal{S}\sim \mathcal{S}'$ by induction on the size $|V|$.
    Consider the following cases:
    \begin{enumerate}
        \item $\mathcal{S}$ and $\mathcal{S'}$ have a common member $U$. 
        
        In this case, consider starting the decomposition process with $\delta^+(U)$ in both cases. After this step, we end up with two digrafts whose sizes are both less than $|V|$, so the inductive hypothesis applies.
        \item There exist $U\in \mathcal{S}$ and $U'\in\mathcal{S'}$ that do not cross. 
        
        Consider a maximal cross-free family $\mathcal{S}''$ of tight dicuts out-shores that contains both $U$ and $U'$. By the previous case, $\mathcal{S}''\sim \mathcal{S}$ as well as $\mathcal{S}''\sim \mathcal{S}'$, so $\mathcal{S}\sim \mathcal{S}'$.
        \item There exist $U\in \mathcal{S}$ and $U'\in \mathcal{S}'$ such that $|U\cap U'|\neq 1$ or $|U\cup U'|\neq |V(D)|-1$.  
        
        Let $U''$ be either $U\cap U'$ or $U\cup U'$ whichever induces a nontrivial dicut. By Lemma \ref{lemma:tight_dicut_uncross}, $\delta^+(U'')$ is a tight dicut. Let $\mathcal{S}''$ be any maximal cross-free family of nontrivial tight dicuts out-shores containing $U''$. Since $U''$ do not cross $U$ or $U'$, by the second case, $\mathcal{S}\sim\mathcal{S}''$ and $\mathcal{S}'\sim\mathcal{S}''$. Therefore, $\mathcal{S}\sim\mathcal{S}'$.
        \item In the remaining case, for any $U\in \mathcal{S}$ and $U'\in \mathcal{S}'$, one has $U\cap U'=\{u\} \subseteq S$ and $\overline{U\cup U'}=\{v\}\subseteq T$ for some $u,v \in V$. 
        
        By \Cref{lemma:tight_dicut_uncross}, both $u$ and $v$ are tight nodes. Further, since $\delta^+(U), \delta^+(U')$ are both dicuts, there is no arc in between $U\setminus U'$ and $U'\setminus U$, which means $\{u,v\}$ forms a $2$-cut. One can easily verify that breaking along $\delta^+(U)$ and $\delta^+(U')$ result in the same pair of digraphs up to arc multiplicities. Therefore, starting the decompositions corresponding to $\mathcal{S}$ and $\mathcal{S}'$ with $U$ and $U'$, respectively, we get two pairs of isomorphic (up to arc multiplicities) digrafts. Applying the inductive hypothesis to the two contracted digrafts implies $\mathcal{S}\sim \mathcal{S}'$.
    \end{enumerate}
\end{proof}

\section{Proof of results in \Cref{sec:robust-case}}\label{sec:proof-robust}
We first prove \Cref{lemma:tight_edge_cover}, which is giving a characterization of the existence of a tight edge cover. This is a direct consequence of the following lemma from \cite{schrijver2003combinatorial}.
\begin{lemma}[\cite{schrijver2003combinatorial}, Theorem 21.28]\label{lemma:ab-matching}
Let $D = (V,A)$ be a bipartite graph with bipartition $V=S\dot\cup T$ and let $a,b \in \mathbb{Z}^V$ and 
$d,c \in \mathbb{Z}^A$ with $a \le b$ and $d \le c$. Then there exists 
$x \in \mathbb{Z}^A$ satisfying 
\begin{equation}
\begin{aligned}\label{eq:abcd-matching}
    d_e\leq &~x_e\leq c_e\quad \forall e\in A\\
    a_v\leq &~x(\delta(v))\leq b_v\quad \forall v\in V.
\end{aligned}
\end{equation}
if and only if
\begin{subequations}\label{eq:ab-matching}
\begin{align}
  c(E(Z)) - d(E(V \setminus Z)) &\ge a(S \cap Z) - b(T \setminus Z)\quad \forall Z\subseteq V, \label{eq:ab-matching1}\\ 
  c(E(Z)) - d(E(V \setminus Z)) &\ge a(T \cap Z) - b(S \setminus Z)\quad \forall Z\subseteq V. \label{eq:ab-matching2}
\end{align}
\end{subequations}
\end{lemma}
\begin{proof}[Proof of \Cref{lemma:tight_edge_cover}]
Consider $a_v=1\ \forall v\in V$, $b_v=\begin{cases}
    1 & \forall v\in V^t \\ \infty\ & \forall v\in V\setminus V^t
\end{cases}$ and $c_e=\infty\ \forall e\in A$, $d_e=0\ \forall e\in A$.
Then, $x$ satisfies \eqref{eq:abcd-matching} if and only if $x$ is the characteristic vector of some tight edge cover (note that $x(\delta(v))\leq 1\ \forall v\in T$ implies $x_e\leq 1\ \forall e\in A$ and thus $x$ is the characteristic vector of some subset of arcs as opposed to a multi-set). By \Cref{lemma:ab-matching}, a tight edge cover exists if and only if \eqref{eq:ab-matching} holds. Let $X:=S\cap Z$ and $Y:=T\cap Z$. For $Z$ such that \eqref{eq:ab-matching} does not automatically hold, $E(Z)=\emptyset$, which means $N(X)\cap Y=\emptyset$. Then, \eqref{eq:ab-matching1} is equivalent to $|T\setminus Y|\geq |X|$ for every $X,Y$ such that $N(X)\cap Y=\emptyset$. For a fixed $X$, we only need to consider $Y=T\setminus N(X)$ which minimizes $|T\setminus Y|$. Therefore, \eqref{eq:ab-matching1} is equivalent to $|N(X)|\geq |X|\ \forall X\subseteq S$. For \eqref{eq:ab-matching2} to not hold automatically, we further need $S\setminus X\subseteq S^t$. Then, \eqref{eq:ab-matching2} is equivalent to $|S\setminus X|\geq |Y|$ for every $X,Y$ such that $N(Y)\cap X=\emptyset$ and $S\setminus X\subseteq S^t$. For a fixed $Y$, we only need to consider $X=S\setminus N(Y)$ which minimizes $|S\setminus X|$. Therefore, \eqref{eq:ab-matching2} is equivalent to $|N(Y)|\geq |Y|\ \forall Y\subseteq T, N(Y)\subseteq S^t$.
\end{proof}

Next, we prove \Cref{thm:testing-robust} that states an algorithm for testing robustness of a digraft.
\begin{proof}[Proof of \Cref{thm:testing-robust}]
    It follows from \Cref{thm:robust_digraft} that it suffices to check conditions \eqref{eq:robust_1} and \eqref{eq:robust_2}. Let $f(X):=|N(X)|-|X|$ for every $X\subseteq S$. One can easily check that $f$ is a submodular. To check \eqref{eq:robust_1}, it suffices to find a subset $X\subsetneqq S,\ X\not\subseteq S^t$ that minimizes $f(X)$. This can be done by solving $\min\{f(X):u\in X, v\notin X\}$ for every $u\in S\setminus S^t$ and $v\in S$. 
    If the minimum is smaller than $|T|-|S|+1$, then the digraft is not robust. In such a case we take the minimizer $X$ and output $C=\delta^-(X\cup N(X))$ as a separating dicut.

    To check \eqref{eq:robust_2}, it suffices to find a subset $\emptyset \neq X\subsetneqq S,\ X\subseteq S^t$ that minimizes $f(X)$. Note that $f$ is crossing submodular over the feasible $X$'s, whose minimization can be solved in polynomial time \cite{IwataFleischerFujishige2001,Schrijver2000}. If the minimum is smaller than $1$, then the digraft is not tight dijoin-covered. In such a case we can find an arbitrary tight edge cover $J$ and there will be a dicut $C$ with $J\cap C=\emptyset$. We output such a dicut as a separating dicut.
\end{proof}

Finally, we prove \Cref{lemma:integral_basis}.
\begin{proof}[Proof of \Cref{lemma:integral_basis}]
Recall, this lemma states that one can combine an integral basis for $\lin(\{J\in \zJ(D,S^t): |J\cap C|=1\})$ and a tight dijoin $J_0$ with $|J_0\cap C|=2$ into an integral basis for $\lin(\{J\in \zJ(D,S^t): |J\cap C|=1\}\cup \{J_0\})$.

    Set $\mathcal{J}_C:=\{J\in \zJ(D,S^t): |J\cap C|=1\}$. Suppose $\zB'=\{J_1,\ldots J_k\}\subseteq \mathcal{J}_C$.  First, we show that $\zB$ is linearly independent. Suppose for $\lambda\in \R^{k+1}$, $\mathbf{0}=\sum_{j=0}^k\lambda_j \mathbf{1}_{J_j}$. Then, \[0=\sum_{j=0}^k\lambda_j \mathbf{1}_{J_j}(C)=2\lambda_0+\sum_{j=1}^k \lambda_j=\lambda_0+\sum_{j=0}^k \lambda_j.\] However, pick any $v\in T$, we have \[0=\sum_{j=0}^k\lambda_j \mathbf{1}_{J_j}(\delta(v))=\sum_{j=0}^k \lambda_j.\] This implies $\lambda_0=0$. Then, we conclude that $\mathbf{0}=\sum_{j=1}^k\lambda_j \mathbf{1}_{J_j}$, which implies that $\lambda_{1}=\ldots=\lambda_k=0$ since $\zB'$ is linearly independent. This proves that $\zB$ is linearly independent. Since $\dim(\lin(\zJ_C\cup \{J_0\}))\leq \dim(\lin(\zJ_C))+1=|\zB'|+1=|\zB|$, $\zB$ forms a linear basis for $\lin(\zJ_C\cup \{J_0\})$.
    
    It remains to show that $\zB$ is also an integral basis. For an arbitrary $x\in \lin(\zJ_C\cup \{J_0\})\cap \Z^A$, there exists $\lambda\in \R^{k+1}$ such that $x=\sum_{j=0}^k \lambda_j\mathbf{1}_{J_j}$. Then, \[\lambda_0+\sum_{j=0}^k \lambda_j=\sum_{j=0}^k\lambda_j \mathbf{1}_{J_j}(C)=x(C)\in \Z.\]
    On the other hand, pick any $v\in T$, we have
    \[\sum_{j=0}^k \lambda_j=x(\delta(v))\in\Z.\]
    This implies $\lambda_0\in\Z$. Therefore, $\sum_{j=1}^k\lambda_j\mathbf{1}_{J_j}=x-\lambda_0\mathbf{1}_{J_0}\in\Z^A$. This implies that $\lambda_1, \ldots, \lambda_{k}\in\Z$ because $\zB'$ is an integral basis for $\lin(\zJ_C)$. This proves that $\zB$ is an integral basis for $\lin(\zJ_C\cup \{J_0\})$.
\end{proof}

\end{document}